\documentclass[11pt,reqno]{amsart}
\usepackage{amsmath,amssymb,latexsym,esint,cite,mathrsfs}
\usepackage{verbatim,wasysym}
\usepackage[left=2cm,right=2cm,top=2.5cm,bottom=2.5cm]{geometry}
\usepackage{tikz,enumitem,graphicx, subfig, microtype, color}
\usepackage{epic,eepic}

\usepackage[colorlinks=true,urlcolor=blue, citecolor=red,linkcolor=blue,
linktocpage,pdfpagelabels, bookmarksnumbered,bookmarksopen]{hyperref}
\usepackage[hyperpageref]{backref}
\usepackage[english]{babel}
\usepackage{appendix}

\numberwithin{equation}{section}

\newtheorem{thm}{Theorem}[section]
\newtheorem{lem}[thm]{Lemma}

\newtheorem{Prop}[thm]{Proposition}

\newtheorem{Rem}[thm]{Remark}

\newcommand{\R}{\mathbb{R}}

\newcommand\cq{\mathcal{Q}}

\begin{document}
	\baselineskip=14pt
	\title[Critical Hartree equation]{ Existence of infinitely many solutions for a critical Hartree type equation with potential : local Poho\v{z}aev identities methods}
	
	\author[D. Cassani]{Daniele Cassani}
	\author[M. Yang]{Minbo Yang$^*$}
	\author[X. Zhang]{Xinyun Zhang$^{**}$}
	
	\address{Daniele Cassani, 
	\newline\indent Dipartimento di Scienza e Alta Tecnologia,
	\newline\indent Universit\`{a} degli Studi dell'Insubria
	\newline\indent and
	\newline\indent RISM-Riemann International School of Mathematics
	\newline\indent Villa Toeplitz, Via G.B. Vico, 46 - 21100 Varese, Italy}
	\email{\href{mailto:daniele.cassani@uninsubria.it}{daniele.cassani@uninsubria.it}}
	
	\address{Minbo Yang \and  Xinyun Zhang \newline\indent Department of Mathematics, Zhejiang Normal University, \newline\indent
		Jinhua 321004, People's Republic of China}
	\email{\href{mailto:mbyang@zjnu.edu.cn}{mbyang@zjnu.edu.cn}}
	\email{\href{mailto:xyzhang@zjnu.edu.cn}{xyzhang@zjnu.edu.cn}}	
	
	
	\thanks{$^*$Minbo Yang is partially supported by NSFC (11971436, 12011530199) and ZJNSF(LZ22A010001).}
	\thanks{$^{**}$Xinyun Zhang is the corresponding author.}
	
	\begin{abstract}
		This paper deals with the following equation
		$$
		-\Delta u
		=K(|x'|, x'')\Big(|x|^{-\alpha}\ast
		(K(|x'|, x'')|u|^{2^{\ast}_{\alpha}})\Big)|u|^{2^{\ast}_{\alpha}-2}u\hspace{4.14mm}\mbox{in}\hspace{1.14mm} \mathbb{R}^N,
		$$
		where $N\geq5$, $\alpha>5-\frac{6}{N-2}$, $2^{\ast}_{\alpha}=\frac{2N-\alpha}{N-2}$ is the so-called upper critical exponent in the Hardy-Littlewood-Sobolev inequality and $K(|x'|, x'')$, where $(x',x'')\in \mathbb{R}^2\times\mathbb{R}^{N-2}$, is bounded and nonnegative. Under proper assumptions on the potential function $K$, we obtain the existence of  infinitely many solutions for the nonlocal critical equation by using a finite dimensional reduction argument and local Poho\v{z}aev identities. It is a remarkable fact that the order of the Riesz potential influences the existence/non-existence of solutions. 	
	\end{abstract}
	
	\maketitle
	
	\section{Introduction and main results}	
\noindent 	We consider the following equation
		\begin{equation}\label{eq1}
		-\Delta u
		=K(|x'|, x'')\Big(|x|^{-\alpha}\ast
		(K(|x'|, x'')|u|^{2^{\ast}_{\alpha}})\Big)|u|^{2^{\ast}_{\alpha}-2}u\hspace{4.14mm}\mbox{in}\hspace{1.14mm} \mathbb{R}^N, \ u>0,\ u\in D^{1,2}(\mathbb{R}^N)
		\end{equation}
		where $N\geq5$, $\alpha>5-\frac{6}{N-2}$, $2^{\ast}_{\alpha}=\frac{2N-\alpha}{N-2}$ is the so-called upper critical exponent in the Hardy-Littlewood-Sobolev inequality and $K(|x'|, x'')$, where $(x',x'')\in \mathbb{R}^2\times\mathbb{R}^{N-2}$, is bounded and nonnegative. $D^{1,2}(\mathbb{R}^N)$ is the completion of $C_{0}^{\infty}(\mathbb{R}^N)$ with respect to the norm $\| \cdot \|=\|\nabla \cdot \|_2$.
	Problem \eqref{eq1} is related to the Choquard equation
	\begin{equation}\label{Cho}
		-\Delta u + V(x)u
		=\Big(|x|^{-\alpha}\ast
		|u|^{p}\Big)u^{p-1}\hspace{4.14mm}\mbox{in}\hspace{1.14mm} \mathbb{R}^N,
	\end{equation}
in the zero mass case ($V\equiv0$) and where $K\equiv 1$. Equation \eqref{Cho} shows up in many different  fields of mathematical physics. For $\alpha = 1,p = 2$, it goes back to the description of the quantum theory of a polaron at rest by Pekar\cite{P1} in 1954 and the modeling of an electron trapped in its own hole in 1976 in the work of Choquard, as a certain approximation to the Hartree-Fock theory of one-component plasma. Lieb \cite{Lieb1} proved the existence and uniqueness of the ground state by rearrangement techniques when $\alpha = 1,p = 2$ and $V$ is a positive constant. Lions\cite{Ls} proved the existence of a sequence of radially symmetric solutions by variational methods. See \cite{MS1,MS2,guide,CZ,CT,R} and references therein for more details.
	
\noindent 	In order to make precise the notion of criticality, let us recall from  \cite[Theorem 4.3]{LL} the Hardy-Littlewood-Sobolev (HLS) inequality 
	\begin{Prop}\label{pro1.1}
		Let $t,\,r>1$ and $0<\alpha<N$ be such that $\frac{1}{t}+\frac{\alpha}{N}+\frac{1}{r}=2$.
		Then, there exists a constant $C(N,\alpha,t)$ such that, for $f\in L^{t}(\R^N)$
		and $h\in L^{r}(\R^N)$,
		$$
		\left|\int_{\R^{N}}\int_{\R^{N}}\frac{f(x)h(y)}{|x-y|^{\alpha}}dxdy\right|
		\leq C(N,\alpha,t) |f|_{L^t(\R^N)}|h|_{L^r(\R^N)}.
		$$
		If $t=r=2N/(2N-\alpha)$, then
		$$
		C(t,N,\alpha,r)=C(N,\alpha)=\pi^{\frac{\alpha}{2}}\frac{\Gamma(\frac{N}{2}-\frac{\alpha}{2})}{\Gamma(N-\frac{\alpha}{2})}\left\{\frac{\Gamma(\frac{N}{2})}{\Gamma(N)}\right\}^{-1+\frac{\alpha}{N}}.
		$$
		and the equality holds if and only if $f\equiv Ch$ and
		$$
		h(x)=A(\gamma^{2}+|x-a|^{2})^{-(2N-\alpha)/2}
		$$
		for some $A\in \mathbb{R}$, $0\neq\gamma\in\mathbb{R}$ and $a\in \mathbb{R}^{N}$.
	\end{Prop}
	
\noindent 	According to Proposition \ref{pro1.1}, the functional
	$$
	\int_{\R^{N}}\int_{\R^{N}}\frac{|u(x)|^{p}|v(y)|^{p}}{|x-y|^{\alpha}}dxdy
	$$
	turns out to be well defined in $H^1(\R^N)\times H^1(\R^N)$, provided $p$ ranges between the lower critical HLS exponent (see \cite{CVZ}) and the upper critical HLS exponent, namely $\frac{2N-\alpha}{N}\leq p\leq\frac{2N-\alpha}{N-2}=:2_{\alpha}^{*}$.
		A nonlocal version of the Sobolev inequality which originates from Proposition \ref{pro1.1} is the following 
	\begin{equation}
		\left(\int_{\mathbb{R}^{N}} \int_{\mathbb{R}^{N}} \frac{|u(x)|^{2_{\alpha}^{*}}|u(y)|^{2_{\alpha}^{*}}}{|x-y|^{\alpha}} d x d y\right)^{\frac{1}{2_{\alpha}^{*}}} \leqslant S_{H, L}\int_{\mathbb{R}^{N}}|\nabla u|^{2} d x.
	\end{equation}
	From \cite{GY}, the best constant $S_{H,L}$ is given by
	$$
	S_{H, L}:=\inf _{u \in D^{1,2}\left(\mathbb{R}^{N}\right) \backslash\{0\}} \frac{\int_{\mathbb{R}^{N}}|\nabla u|^{2} d x}{\left(\int_{\mathbb{R}^{N}} \int_{\mathbb{R}^{N}} \frac{|u(x)|^{2_{\alpha}^{*}}|u(y)|^{2_{\alpha}^{*}}}{|x-y|^{\alpha}} d x d y\right)^{\frac{1}{2_{\alpha}^{\ast}}}}
	$$
	and achieved if and only if $u$ has the from
	$$
	u=C\left(\frac{\lambda}{\lambda^{2}+|x-z|^{2}}\right)^{\frac{N-2}{2}}.
	$$
	Furthermore, the authors in \cite{GY} also proved the relationship between the best Sobolev constant $S$ and the best constant $S_{H, L}$, namely
	$$
	S_{H, L}=\frac{S}{C^{\frac{1}{2_{\alpha}^{\ast}}}(N, \alpha)}.
	$$
	
\noindent 	Notice that the local version of problem \eqref{eq1} is exactly the prescribed scalar curvature problem on $\mathbb{S}^{N}$, which by using the stereographic projection, can be written as follows 
	\begin{equation}\label{eq2}
		-\Delta u={K}u^{\frac{N+2}{N-2}},\hspace{0.5mm}u>0,\hspace{0.5mm}u\in D^{1,2}(\mathbb{R}^N).
	\end{equation}
	In the last four decades, the problem of determining conditions on $K$ to have solutions, has been widely studied, see \cite{BC,BG,CL,CY,H,DLY,CNY,L2,L3,YAN} and references therein. 
	There are also some multiplicity results concerning \eqref{eq2}. In the case $N\geq3$ and $K(x)$ is periodic in one variable, Li\cite{L1} proved the existence of infinitely many solutions of \eqref{eq2}. When $K(x)$ is a positive radially symmetric function with a strict local maximum at $|x|=r_{0}>0$ and satisfies as $r\to r_0$
	\begin{equation}\label{k}
		K(r)=K(r_{0})-c_{0}|r-r_{0}|^{\kappa}+O(|r-r_{0}|^{\kappa+\theta}),
	\end{equation}
for constants $c_{0}>0, \theta>0, \kappa\in[2,N-2)$. Wei and Yan\cite{WY1} developed a technique which enables one to use finite dimensional reduction arguments and constructed infinitely many non-radial positive solutions whose energy is arbitrary large. In \cite{GL}, Guo and Li considered the polyharmonic operator in place of the laplacian and obtained infinitely many solutions. Li, Wei and Xu \cite{LWX} obtained the existence of multi-bump solutions and then local uniqueness and periodicity of those solutions were proved in \cite{DLY}. The fractional case was considered by Guo et al. \cite{GN,GNNT} in the same setting of \cite{WY1}. In \cite{PWW}, Peng, Wang and Wei proved the existence of infinitely many solutions under more general conditions on $K$ (proposed by \cite{PWY}) which allow for saddle points to be stable critical points of $K$. Guo, Liu and Peng\cite{GLP} considered the existence and non-degeneracy of positive  multi-bubble solutions to critical elliptic systems of Hamiltonian type with $K$ satisfying \eqref{k}. Recently, Guo, Musso, Peng and Yan\cite{GMPY} proved the non-degeneracy of the positive bubble solutions and built a new type of solutions by gluing a large number of bubbles. 
Guo et al.\cite{GHLN} glued together bubbles with different concentration rates and constructed further solutions. 
	
	\medskip
	
\noindent Passing from the semilinear case \eqref{eq2} to the nonlocal Hartree case \eqref{eq1}, the situation becomes more complicated. In order to construct multi-bubble solutions to the Hartree equation, some delicate argument has been used by Gao et al. in \cite{GMYZ}, where the authors construct infinitely many solutions for the critical Hartree equation with axisymmetric potentials, namely 
	$$
	-\Delta u+ V(|x'|,x'')u
	=\Big(|x|^{-4}\ast |u|^{2}\Big)u\hspace{4.14mm}\mbox{in}\hspace{1.14mm} \mathbb{R}^6,
	$$
	where $(x',x'')\in \mathbb{R}^2\times\mathbb{R}^{4}$ and $V(|x'|, x'')$ is a bounded nonnegative function. 
	To the best of our knowledge no results are known for \eqref{eq1}. Inspired by \cite{GMYZ,PWW}, we aim at constructing multi-bubble positive solutions of 
	\begin{equation}\label{CFL}
		-\Delta u
		=K(|x'|, x'')\Big(|x|^{-\alpha}\ast K(|x'|, x'')|u|^{2^{\ast}_{\alpha}}\Big)u^{2^{\ast}_{\alpha}-1}\hspace{4.14mm}\mbox{in}\hspace{1.14mm} \mathbb{R}^N,
	\end{equation} 
    where $N\geq 5$, $\alpha>5-\frac{6}{N-2}$ and we consider the case $K(x)=K(|x'|, x'')=K(r,x'')$, where $(x',x'')\in \mathbb{R}^2\times\mathbb{R}^{N-2}$. We further assume that $K$ is bounded and satisfies the following conditions:
	\begin{itemize}
		\item[$\textbf{(K1)}$]  The function $K (r, x'')$ has a critical point $(r_0, x_0'')$ such that $r_0 > 0$ and $K (r_0, x_0'') =1 $, and $$\deg(\nabla(K(r, x'')),(r_0, x_0'')) \neq 0\ ;$$
		\item[$\textbf{(K2)}$]  The function $K (r, x'') \in  C^{3}(B_{\vartheta }(r_0, x_0'')) $, where $\vartheta >0$ is sufficiently small and $$\Delta K (r_0, x_0'')=\frac{\partial^2 K (r_0, x_0'')}{\partial r^2}+\sum_{i=3}^{N}\frac{\partial^2 K (r_0, x_0'')}{\partial {x_i}^2}<0\ .$$
	\end{itemize}
\medskip	

\noindent The main result of this paper is the following 
	
	\begin{thm}\label{EXS}
		Suppose $N\geq5$ and that $K(x)$ satisfies assumptions $\textbf{(K1)}$ and $\textbf{(K2)}$. Then, problem \eqref{CFL} has infinitely many solutions whose energy can be arbitrarily large.
	\end{thm}
    \begin{Rem}
    	While we construct infinitely many bubble solutions lying near the circle $|x|=r_0$ in $(x_1,x_2)$-plane, the energy of these solutions can be arbitrarily large and the distance between different bubbles can be arbitrarily small. It fails in dimension $N=3,4$ since equation \eqref{eq1} has the same conformal property as prescribed scalar curvature problem on $S^N$. On the other hand, we choose $\lambda \sim m^{\frac{N-2}{N-4}}$ as the scaling parameter which in turn requests $N\geq 5$. So the dimensional condition in this paper is optimal in both theoretical and technical sense.
    \end{Rem}
\noindent 	We will construct infinitely many solutions of \eqref{CFL} by means of a finite dimensional reduction argument and find algebraic equations which determine the location of bubble shaped solutions by establishing new local Poho\v{z}aev type indentities.	It is well known how finite dimensional reduction arguments rely on the non-degeneracy of the solutions of the limiting equation.  
	For the following critical Hartree equation
	\begin{equation}\label{eq1.5}
		-\Delta u=\Big(|x|^{-\alpha}\ast |u|^{2_{\alpha}^{\ast}}\Big)|u|^{2_{\alpha}^{\ast}-2}u
		\ \ \mbox{in}\ \R^N\ ,
	\end{equation}
	by using the moving plane method in integral form developed in \cite{CLO1,CLO2} ,  Lei \cite{Lei}, Du and Yang \cite{DY}, Guo et al. \cite{GHPS} classified independently the positive solutions of \eqref{eq1.5}
	and proved that any positive solution of \eqref{eq1.5} has the form
	\begin{equation}\label{REL}
		U_{z,\lambda}(x)=S^{\frac{(N-\alpha)(2-N)}{4(N-\alpha+2)}}C(N,\alpha)^{\frac{2-N}{2(N-\alpha+2)}}[N(N-2)]^{\frac{N-2}{4}}\Big(\frac{\lambda}{1+\lambda^{2}|x-z |^{2}}\Big)^{\frac{N-2}{2}},\ z\in \R^N, \lambda>0\ .
	\end{equation}
    When $N=6$ and $\alpha=4$, Yang and Zhao\cite{YZ1} proved that the solution of \eqref{eq1.5} is non-degenerate in the sense that the linearization of equation \eqref{eq1.5} around the solution $U_{0,1}$:
	\begin{equation}
		-\Delta \varphi=\frac{1}{2}(I_{2}\ast U_{0,1}^2)\varphi+(I_{2}\ast U_{0,1}\varphi)U_{0,1}
	\end{equation}
	only admits solutions in $D^{1,2}(\R^6)$ of the form
	$$
	\varphi=\overline{a}D_{\lambda}U_{0,1}+\mathbf{b}\cdot\nabla U_{0,1},\ \ \overline{a}\in \R,\ \ \mathbf{b}\in \R^6 .	
	$$
    Recently, Li et al.\cite{LLTX} extended the non-degeneracy result to the general case $N\geq3$ and $0<\alpha<N$.

	\noindent Define
	$$\aligned
	H:=\Big\{&u\in D^{1,2}(\mathbb{R}^N)\ |\ u(x_{1},-x_{2},x'')=u(x_{1},x_{2},x''),\\
	&\hspace{4mm}u(r\cos\theta,r\sin\theta,x'')=u\Big(r\cos(\theta+\frac{2j\pi}{m}),r\sin(\theta+\frac{2j\pi}{m}),x''\Big)\Big\}
	\endaligned$$
	and let
	$$
	z_{j}=\Big(\overline{r}\cos\frac{2(j-1)\pi}{m},\overline{r}\sin\frac{2(j-1)\pi}{m},\overline{x}''\Big), \ j=1,\cdot\cdot\cdot,m,
	$$
	where $\overline{x}''$ is a vector in $\mathbb{R}^{N-2}$. By the weak symmetry of $K(x)$, we have $K (z_j ) = K (\overline{r}, \overline{x}'')$, $j = 1,\cdots, m$.
	In order to construct solutions which concentrate at $(r_0, x_0'')$, we use $U_{z_j,\lambda}$ (see \eqref{REL}) as an approximate solution. Let $\delta> 0$ be a sufficiently small
	constant, such that $K (r, x'') > 0$ if $|(r, x'')-(r_0, x_0'' )|\leq 10\delta$. Let $\xi(x) = \xi(|x'|, x'')$ be a
	smooth function satisfying $\xi= 1$ if $|(r, x'')-(r_0, x_0'' )|\leq\delta$, $\xi= 0$ if $|(r, x'')-(r_0, x_0'' )| \geq2\delta$,
	and $0\leq \xi\leq 1$. Let us use the following notation 
	$$
	Z_{z_j,\lambda}(x)=\xi U_{z_j,\lambda}(x), \
	Z_{\overline{r},\overline{x}'',\lambda}(x)=\sum_{j=1}^{m}Z_{z_j,\lambda}(x), \
	Z_{\overline{r},\overline{x}'',\lambda}^{\ast}(x)=\sum_{j=1}^{m}U_{z_j,\lambda}(x),
	$$
	and
	$$
	Z_{j,1}=\frac{\partial Z_{z_j,\lambda}}{\partial \lambda}, Z_{j,2}=\frac{\partial Z_{z_j,\lambda}}{\partial \overline{r}},
	Z_{j,k}=\frac{\partial Z_{z_j,\lambda}}{\partial \overline{x}_{k}''},
	\ \mbox{for }k=3,\cdot\cdot\cdot,N, \ j=1,\cdot\cdot\cdot,m.
	$$

	\noindent In this paper, we will always assume that $m > 0$ is a large integer, $\lambda\in[L_{0}m^{\frac{N-2}{N-4}},L_{1}m^{\frac{N-2}{N-4}}]$
	for some constants $L_1 > L_0 > 0$ and
	\begin{equation}\label{uxi}
		|(\overline{r}, \overline{x}'')-(r_0, x_0'')|\leq \frac{1}{\lambda^{1-\theta}} ,
	\end{equation}
	where $\theta> 0$ is a sufficiently small constant.

\noindent 	In order to prove Theorem \ref{EXS}, we will prove the following 
	\begin{lem}\label{EXS1}
		Under the assumptions of Theorem \ref{EXS}, there exists a positive integer $m_0>0$, such that for any integer $m\geq m_0$, \eqref{CFL} has a solution $u_m$ of the form
		\begin{equation}\label{form}
			u_m=Z_{\overline{r}_{m},\overline{x}_{m}'',\lambda_{m}}+
			\phi_{\overline{r}_{m},\overline{x}_{m}'',\lambda_{m}}=\sum_{j=1}^{m}\xi U_{z_j,\lambda_{m}}+\phi_{\overline{r}_{m},\overline{x}_{m}'',\lambda_{m}},
		\end{equation}
		where $\phi_{\overline{r}_{m},\overline{x}_{m}'',\lambda_{m}}\in H$ and $\lambda_{m}\in\big[L_0m^{\frac{N-2}{N-4}},L_1m^{\frac{N-2}{N-4}}\big]$. Moreover, as $m\rightarrow\infty$, $(\overline{r}_m, \overline{x}_m'')\rightarrow(r_0, x_0'')$, and $\lambda_{m}^{-\frac{N-2}{2}}\|\phi_{\overline{r}_{m},\overline{x}_{m}'',\lambda_{m}}\|_{L^{\infty}}\rightarrow0$.
	\end{lem}

\medskip
	
\subsection*{Overview} Let us briefly outline main ideas. The energy functional corresponding to \eqref{CFL} is given by 
	$$
	J(u)=\frac{1}{2}\int_{\mathbb{R}^{N}}|\nabla u|^{2}dx
	-\frac{1}{2\cdot2^{\ast}_{\alpha}}
	\int_{\mathbb{R}^N}\int_{\mathbb{R}^N}
	\frac{K(r,x'')K(r,y'')|u(x)|^{2^{\ast}_{\alpha}}|u(y)|^{2^{\ast}_{\alpha}}}{|x-y|^{\alpha}}dxdy.
	$$
	After performing the finite-dimensional reduction argument, the problem to find critical points of $J(u)$ of the form \eqref{form} turns into finding critical points for
	$$
	F(\overline{r},\overline{x}'',\lambda):=J(Z_{\overline{r},\overline{x}'',\lambda}+
	\phi_{\overline{r},\overline{x}'',\lambda})
	$$
	where $\lambda\in[L_{0}m^{\frac{N-2}{N-4}},L_{1}m^{\frac{N-2}{N-4}}]$ for some constants $L_1 > L_0 > 0$ and $(\overline{r},\overline{x}'')$ satisfies \eqref{uxi}. In order to determine the location of bubbles, we can not differrentiate directly the reduced function $F(\overline{r},\overline{x}'',\lambda)$ with respect to $\overline{r}$, $\overline{x}''$ as it yields further issues. By means of local Poho\v{z}aev identities introduced in \cite{PWY}, we prove that if $(\overline{r},\overline{x}'')$ satisfies in a suitable neighborhood $D_{\rho}$ of $(\overline{r}_{0},\overline{x}''_{0})$ the following 
	\begin{equation}
		\int_{D_{\rho}}\Big(-\Delta u
		-K(r,x'')\Big(|x|^{-\alpha}\ast (K(r,x'')|u|^{2^{\ast}_{\alpha}})\Big)u^{2^{\ast}_{\alpha}-1}\Big)\langle x,\nabla u\rangle dx=0,
	\end{equation}
	and
	\begin{equation}
		\int_{D_{\rho}}\Big(-\Delta u
		-K(r,x'')\Big(|x|^{-\alpha}\ast (K(r,x'')|u|^{2^{\ast}_{\alpha}})\Big)u^{2^{\ast}_{\alpha}-1}\Big)\frac{\partial u}{\partial x_{i}} dx=0, \ i=3,\cdots,N
	\end{equation}
	where $u=Z_{\overline{r},\overline{x_{i}}'',\lambda}+\phi_{\overline{r},\overline{x}'',\lambda}$
	is the function obtained by reduction argument, then 
	$$\frac{\partial F}{\partial \overline{r}}=0\quad \text{ and } \quad \frac{\partial F}{\partial \overline x_{i}''}=0\ .$$
	\noindent By using this method, we mainly need the estimate of error term away from the concentration point.

	\noindent This paper is organized as follows. In Section 2, we will perform the finite dimensional reduction argument for \eqref{CFL}. Then, we will establish energy estimates in Section 3. In Section 4, we will study the reduced problem and prove Lemma \ref{EXS1}. Some basic estimates are collected in the Appendix.

	\section{Finite-dimensional reduction}
\noindent	In this section, we carry out the finite-dimensional reduction in a suitable weighted space introduced in \cite{WY1} and we establish a fine estimate for the error term. Let
	$$
	\|u\|_{\ast}=\sup_{x\in\mathbb{R}^N}\Big(\sum_{j=1}^{m}
	\frac{1}{(1+\lambda|x-z_{j}|)^{\frac{N-2}{2}+\tau}}\Big)^{-1}\lambda^{-\frac{N-2}{2}}|u(x)|,
	$$
	and
	$$
	\|h\|_{\ast\ast}=\sup_{x\in\mathbb{R}^N}\Big(\sum_{j=1}^{m}
	\frac{1}{(1+\lambda|x-z_{j}|)^{\frac{N+2}{2}+\tau}}\Big)^{-1}\lambda^{-\frac{N+2}{2}}|h(x)|,
	$$
	where $\tau=\frac{N-4}{N-2}$.
	
	Consider
	\begin{equation}\label{c1}
		\left\{\begin{array}{l}
			\displaystyle -\Delta \phi
			-(2^{\ast}_{\alpha}-1)K(x)\Big(|x|^{-\alpha}\ast (K(x)|Z_{\overline{r},\overline{x}'',\lambda}|^{2^{\ast}_{\alpha}})\Big)Z_{\overline{r},\overline{x}'',\lambda}^{2^{\ast}_{\alpha}-2}\phi-2^{\ast}_{\alpha}K(x)\Big(|x|^{-\alpha}\ast (K(x)Z_{\overline{r},\overline{x}'',\lambda}^{2^{\ast}_{\alpha}-1}
			\phi)\Big)Z_{\overline{r},\overline{x}'',\lambda}^{2^{\ast}_{\alpha}-1}
			\\
			\displaystyle \hspace{10.14mm}=h+
			\sum_{l=1}^{N}c_{l}\sum_{j=1}^{m}[(2^{\ast}_{\alpha}-1)\Big(|x|^{-\alpha}\ast |Z_{z_j,\lambda}|^{2^{\ast}_{\alpha}}\Big)Z_{z_j,\lambda}^{2^{\ast}_{\alpha}-2}Z_{j,l}+2^{\ast}_{\alpha}\Big(|x|^{-\alpha}\ast (Z_{z_j,\lambda}^{2^{\ast}_{\alpha}-1}
			Z_{j,l})\Big)Z_{z_j,\lambda}^{2^{\ast}_{\alpha}-1}]\hspace{4.14mm}\mbox{in}\hspace{1.14mm} \mathbb{R}^N,\\
			\displaystyle \phi\in H, \ \ \sum_{j=1}^{m}\int_{\mathbb{R}^{N}}\Big[(2^{\ast}_{\alpha}-1)\Big(|x|^{-\alpha}\ast |Z_{z_j,\lambda}|^{2^{\ast}_{\alpha}}\Big)Z_{z_j,\lambda}^{2^{\ast}_{\alpha}-2}Z_{j,l}\phi +2^{\ast}_{\alpha}\Big(|x|^{-\alpha}\ast |Z_{z_j,\lambda}^{2^{\ast}_{\alpha}-1}Z_{j,l}|\Big)Z_{z_j,\lambda}^{2^{\ast}_{\alpha}-1}\phi\Big] dx=0,
			\\
			\displaystyle \hspace{10.14mm}~~l=1,2,\cdots,N,
		\end{array}
		\right.
	\end{equation}
	for some real numbers $c_{l}$.

	\begin{lem}\label{C1}
		Suppose that $\phi_{m}$ solves \eqref{c1} for $h = h_m$. If $\|h_{m}\|_{\ast\ast}\to 0$ as $m\to \infty$, then $\|\phi_{m}\|_{\ast}\to 0$.
	\end{lem}
	\begin{proof}
		We argue by contradiction. Suppose that there exist $m\rightarrow+\infty$, $\overline{r}_m\rightarrow r_{0}$, $\overline{y}_{m}''\rightarrow y_0''$, $\lambda_{m}\in[L_{0}m^{\frac{N-2}{N-4}},L_{1}m^{\frac{N-2}{N-4}}]$ and $\phi_{m}$ solving \eqref{c1} for $h=h_{m}$, $\lambda=\lambda_{m}$, $\overline{r}=\overline{r}_{m}$, $\overline{y}''=\overline{y}_{m}''$,
		with $\|h_{m}\|_{\ast\ast}\rightarrow0$ and $\|\phi_{m}\|_{\ast}\geq c>0$. We may assume that $\|\phi_{m}\|_{\ast}=1$.
		
\noindent 		By \eqref{c1}, we have
		\begin{equation}\label{c2}
			\aligned
			|\phi_m(x)|&\leq
			C\int_{\mathbb{R}^N}\frac{K(|y)}{|y-x|^{N-2}}\Big(|y|^{-\alpha}\ast (K(y)Z_{\overline{r},\overline{y}'',\lambda}^{2^{\ast}_{\alpha}-1}
			|\phi_m|)\Big)Z_{\overline{r},\overline{x}'',\lambda}^{2^{\ast}_{\alpha}-1}(y)dy\\
			&+C	\int_{\mathbb{R}^N}\frac{K(y)}{|y-x|^{N-2}}\Big(|y|^{-\alpha}\ast (K(y)|Z_{\overline{r},\overline{x}'',\lambda}|^{2^{\ast}_{\alpha}})\Big)Z_{\overline{r},\overline{x}'',\lambda}^{2^{\ast}_{\alpha}-2}|\phi_m(y)| dy
			\\
			&+C\sum_{l=1}^{N}|c_{l}|\Bigg[\Big|\sum_{j=1}^{m}\int_{\mathbb{R}^N}\frac{1}{|y-x|^{N-2}}\Big(|y|^{-\alpha}\ast (Z_{z_j,\lambda}^{2^{\ast}_{\alpha}-1}
			Z_{j,l})\Big)Z_{z_j,\lambda}^{2^{\ast}_{\alpha}-1}(y)dy\Big|\\
			&+\Big|\sum_{j=1}^{m}\int_{\mathbb{R}^N}\frac{1}{|y-x|^{N-2}}\Big(|y|^{-\alpha}\ast |Z_{z_j,\lambda}|^{2^{\ast}_{\alpha}}\Big)Z_{z_j,\lambda}^{2^{\ast}_{\alpha}-2}Z_{j,l}(y)dy\Big|\Bigg]\\
			&+C\int_{\mathbb{R}^N}\frac{1}{|y-x|^{N-2}}|h_m(y)|dy.
			\endaligned
		\end{equation}

\noindent 				Define
		$$
		\Omega_{j}=\left\{x=(x',x'')\in\mathbb{R}^{2}\times\mathbb{R}^{N-2}:\Big\langle\frac{x'}{|x'|},
		\frac{z_{j}'}{|z_{j}'|}\Big\rangle\geq\cos\frac{\pi}{m}\right\}, \ j=1,\cdot\cdot\cdot,m.
		$$
		For $y\in\Omega_{1}$, we have $|y- z_{j}|\geq|y- z_{1}|$. According to Lemma \ref{B2}, we have
		\begin{equation}\label{5141}		
			\aligned
			\sum_{j=2}^{m}\frac{1}{(1+\lambda|y- z_{j}|)^{N-2}}&\leq \frac{1}{(1+\lambda|y- z_{1}|)^{\frac{N-2}{2}}}\sum_{j=2}^{m}\frac{1}{(1+\lambda|y- z_{j}|)^{\frac{N-2}{2}}}\\
			&\leq\frac{C}{(1+\lambda|y- z_{1}|)^{N-2-\tau_1}}\sum_{j=2}^{m}\frac{1}{(\lambda|z_{1}-z_{j}|)^{\tau_1}}\\
			&\leq\frac{C}{(1+\lambda|y- z_{1}|)^{N-2-\tau_1}},
			\endaligned
		\end{equation}
		where we used the result from \cite{WY1} 
		$$
		\sum_{j=2}^{m}\frac{1}{(\lambda|z_{1}-z_{j}|)^{\tau_1}}\leq C.
		$$
	   By Lemma \ref{B4} and the fact that
		$$
		Z_{\overline{r},\overline{y}'',\lambda}\leq C\sum_{j=1}^{m}\frac{\lambda^{\frac{N-2}{2}}}{(1+\lambda|x-z_{j}|)^{N-2}},
		$$
    we have
		$$
		Z_{\overline{r},\overline{x}'',\lambda}^{2^{\ast}_{\alpha}-1}\leq C\frac{\lambda^{\frac{N+2-\alpha}{2}}}{(1+\lambda|y-z_{1}|)^{N+2-\alpha-\tau_1\frac{N+2-\alpha}{N-2}}}.
		$$
		Then using again Lemma \ref{B4}, we get
		\begin{equation}\label{51411}
			\aligned
			&|y|^{-\alpha}\ast (K(y)Z_{\overline{r},\overline{y}'',\lambda}^{2^{\ast}_{\alpha}-1}
			|\phi_m|)\\
			&\leq C\|\phi_m\|_{\ast}\sum_{j=1}^{m}\int_{\Omega_{j}}\frac{1}{|x-y|^{\alpha}}\Big(\frac{\lambda^{\frac{N+2-\alpha}{2}}}{(1+\lambda|y-z_{j}|)^{N+2-\alpha-\tau_1\frac{N+2-\alpha}{N-2}}}\sum_{i=1}^{m}
			\frac{\lambda^{\frac{N-2}{2}}}{(1+\lambda|y-z_{i}|)^{\frac{N-2}{2}+\tau}}\Big)dy \\
			&\leq C\|\phi_m\|_{\ast}\sum_{j=1}^{m}\int_{\Omega_{j}}\frac{1}{|x-y|^{\alpha}}\frac{\lambda^{N-\frac{\alpha}{2}}}{(1+\lambda|y-z_{j}|)^{\frac{3N+2}{2}-\alpha-\tau_1\frac{2N-\alpha}{N-2}+\tau}}\\
			&\leq C\|\phi_m\|_{\ast}\sum_{j=1}^{m}\frac{\lambda^{\frac{\alpha}{2}}}{(1+\lambda|y-z_{j}|)^{\min\{\alpha, \frac{N+2}{2}\}}},
			\endaligned
		\end{equation}
		and in $\Omega_{j}$, we have
		$$\aligned
		&\frac{\lambda^{\frac{N+2-\alpha}{2}}}{(1+\lambda|y-z_{j}|)^{N+2-\alpha-\tau_1\frac{N+2-\alpha}{N-2}}}\sum_{i=1}^{m}
		\frac{\lambda^{\frac{N-2}{2}}}{(1+\lambda|y-z_{i}|)^{\frac{N-2}{2}+\tau}}\\
		&\leq \frac{\lambda^{N-\frac{\alpha}{2}}}{(1+\lambda|y-z_{j}|)^{\frac{3N+2}{2}-\alpha-\tau_1\frac{N+2-\alpha}{N-2}+\tau}}
		+\sum_{i=2}^{m}\frac{C}{(\lambda|z_{i}-z_{j}|)^{\tau_1}}\frac{\lambda^{N-\frac{\alpha}{2}}}{(1+\lambda|y-z_{j}|)^{\frac{3N+2}{2}-\alpha-\tau_1\frac{N+2-\alpha}{N-2}-\tau_1+\tau}}\\
		&\leq\frac{C\lambda^{N-\frac{\alpha}{2}}}{(1+\lambda|y-z_{j}|)^{\frac{3N+2}{2}-\alpha-\tau_1\frac{2N-\alpha}{N-2}+\tau}}.
		\endaligned$$
		Here $\tau_1$ can be chosen small enough such that $\eta=-\tau_1\frac{2N-\alpha}{N-2}+\tau>0$. When $\min\{\alpha, \frac{N+2}{2}\}=\alpha$, we can obtain the estimate of first term in the right hand side of  \eqref{c2} as follows
		\begin{equation}\label{c02}
			\aligned
			&\int_{\mathbb{R}^N}\frac{K(y)}{|y-x|^{N-2}} \Big(|y|^{-\alpha}\ast (K(y)Z_{\overline{r},\overline{y}'',\lambda}^{2^{\ast}_{\alpha}-1}
			|\phi_m|)\Big)Z_{\overline{r},\overline{x}'',\lambda}^{2^{\ast}_{\alpha}-1}(y)dy\\
			&\leq C\|\phi_m\|_{\ast}\lambda^{\frac{N-2}{2}}\int_{\mathbb{R}^N}\frac{1}{|y-\lambda x|^{N-2}}
			\sum_{j=1}^{m}\frac{1}{(1+|y-\lambda z_{j}|)^{\alpha}}
			Z_{\overline{r},\overline{x}'',\lambda}^{2^{\ast}_{\alpha}-1}(y) dy\\
			&\leq C\|\phi_m\|_{\ast}\lambda^{\frac{N-2}{2}}\sum_{j=1}^{m}\int_{\Omega_{j}}\frac{1}{|y-\lambda x|^{N-2}}
			\sum_{j=1}^{m}\frac{1}{(1+|y-\lambda z_{j}|)^{\alpha}}
			\frac{1}{(1+|y-\lambda z_{j}|)^{N+2-\alpha-\tau_1\frac{N+2-\alpha}{N-2}}} dy\\
			&\leq C\|\phi_m\|_{\ast}\lambda^{\frac{N-2}{2}}\sum_{j=1}^{m}\int_{\Omega_{j}}\frac{1}{|y-\lambda x|^{N-2}}
			\frac{1}{(1+|y-\lambda z_{j}|)^{N+2-\tau_1\frac{2N-\alpha}{N-2}}} dy
			\\
			&\leq C\|\phi_m\|_{\ast}\lambda^{\frac{N-2}{2}}\sum_{j=1}^{m}
			\frac{1}{(1+\lambda|y- z_{j}|)^{\frac{N-2}{2}+\tilde{\tau}+\theta}} dy.
			\endaligned
		\end{equation}
		By setting $\tilde{\tau}=\tau_1\frac{2N-\alpha}{N-2}$ and  $\theta=\frac{N+2}{2}-2\tilde{\tau}>0$, the last inequality is a consequence of Lemma \ref{B3}. Similarly, we have
		$$\aligned
		&\int_{\mathbb{R}^N}\frac{K(y)}{|y-x|^{N-2}} \Big(|y|^{-\alpha}\ast (K(y)Z_{\overline{r},\overline{y}'',\lambda}^{2^{\ast}_{\alpha}-1}
		|\phi_m|)\Big)Z_{\overline{r},\overline{x}'',\lambda}^{2^{\ast}_{\alpha}-1}(y)dy\\
		&\leq C\|\phi_m\|_{\ast}\lambda^{\frac{N-2}{2}}\int_{\mathbb{R}^N}\frac{1}{|y-\lambda x|^{N-2}}
		\sum_{j=1}^{m}\frac{1}{(1+|y-\lambda z_{j}|)^{\frac{N+2}{2}}}
		Z_{\overline{r},\overline{x}'',\lambda}^{2^{\ast}_{\alpha}-1}(y) dy\\
		&\leq C\|\phi_m\|_{\ast}\lambda^{\frac{N-2}{2}}\sum_{j=1}^{m}\int_{\Omega_{j}}\frac{1}{|y-\lambda x|^{N-2}}
		\sum_{j=1}^{m}\frac{1}{(1+|y-\lambda z_{j}|)^{\frac{N+2}{2}}}
		\frac{1}{(1+|y-\lambda z_{j}|)^{N+2-\alpha-\tau_1\frac{N+2-\alpha}{N-2}}} dy\\
		&\leq C\|\phi_m\|_{\ast}\lambda^{\frac{N-2}{2}}\sum_{j=1}^{m}\int_{\Omega_{j}}\frac{1}{|y-\lambda x|^{N-2}}
		\frac{1}{(1+|y-\lambda z_{j}|)^{\frac{3(N+2)}{2}-\tau_1\frac{2N-\alpha}{N-2}}} dy
		\\
		&\leq C\|\phi_m\|_{\ast}\lambda^{\frac{N-2}{2}}\sum_{j=1}^{m}
		\frac{1}{(1+\lambda|y- z_{j}|)^{\frac{N-2}{2}+\tilde{\tau}+\theta}} dy\ ,
		\endaligned
		$$
		where we denote by $\theta=N+2-2\tilde{\tau}>0$. For the second term in the right hand side of  \eqref{c2}, in a similar fashion we obtain
		\begin{equation}\label{5143}
			\int_{\mathbb{R}^N}\frac{K(y)}{|y-x|^{N-2}}\Big(|y|^{-\alpha}\ast (K(y)|Z_{\overline{r},\overline{x}'',\lambda}|^{2^{\ast}_{\alpha}})\Big)Z_{\overline{r},\overline{x}'',\lambda}^{2^{\ast}_{\alpha}-2}|\phi_m(y)| dy\leq C \|\phi_m\|_{\ast}\lambda^{2}\sum_{j=1}^{m}
			\frac{1}{(1+\lambda|x-z_{j}|)^{2+\tilde{\tau}+\theta}}.
		\end{equation}
		
		For the third term, by using Lemma \ref{B3} and \ref{B4}, we have
		\begin{equation}\label{5144}
			\aligned &\left|\sum_{j=1}^{m}\int_{\mathbb{R}^N}\frac{1}{|y-x|^{N-2}}\Big(|y|^{-\alpha}\ast (Z_{z_j,\lambda}^{2^{\ast}_{\alpha}-1}
			Z_{j,l})\Big)Z_{z_j,\lambda}^{2^{\ast}_{\alpha}-1}(y)dy\right|\\
			&\leq C \lambda^{n_{l}}\sum_{j=1}^{m}\int_{\mathbb{R}^N}\frac{1}{|y-x|^{N-2}}\Big(|y|^{-\alpha}\ast {W^{2^{\ast}_{\alpha}}_{z_j,\lambda}})\Big)W^{2^{\ast}_{\alpha}-1}_{z_j,\lambda}(y)dy\\
			&\leq C \lambda^{n_{l}}\sum_{j=1}^{m}\int_{\mathbb{R}^N}\frac{1}{|y-x|^{N-2}}\frac{\lambda^\frac{\alpha}{2}}{(1+\lambda|y-z_{j}|)^{\alpha}}	\frac{\lambda^{\frac{N+2-\alpha}{2}}}{(1+\lambda|y-z_{j}|)^{N+2-\alpha}}dy\\
			&\leq C \lambda^{n_{l}}\sum_{j=1}^{m}\int_{\mathbb{R}^N}\frac{1}{|y-x|^{N-2}}			\frac{\lambda^{\frac{N+2}{2}}}{(1+\lambda|y-z_{j}|)^{N+2}}dy\\
			&\leq C \lambda^{\frac{N-2}{2}+n_{l}}\sum_{j=1}^{m}			\frac{1}{(1+\lambda|x-z_{j}|)^{\frac{N-2}{2}+\tau}},
			\endaligned
		\end{equation}
		where $n_{1}=-1$, $n_{l}=1,l=2,\cdot\cdot\cdot,N$. Similarly we have
		\begin{equation}\label{5145}
			\left|\sum_{j=1}^{m}\int_{\mathbb{R}^N}\frac{1}{|y-x|^{N-2}}\Big(|y|^{-\alpha}\ast |Z_{z_j,\lambda}|^{2^{\ast}_{\alpha}}\Big)Z_{z_j,\lambda}^{2^{\ast}_{\alpha}-2}Z_{j,l}(y)dy\right|
			\leq C \lambda^{\frac{N-2}{2}+n_{l}}\sum_{j=1}^{m}
			\frac{1}{(1+\lambda|x-z_{j}|)^{\frac{N-2}{2}+\tau}}.
		\end{equation}
		
	\noindent 	For the last term,  by Lemma \ref{B3} we have 
		\begin{equation}\label{5146}
			\int_{\mathbb{R}^N}\frac{1}{|y-x|^{N-2}}|h_m(y)|dy\leq C \|h_m\|_{\ast\ast}\lambda^{\frac{N-2}{2}}\sum_{j=1}^{m}
			\frac{1}{(1+\lambda|x-z_{j}|)^{\frac{N-2}{2}+\tau}}.
		\end{equation}

	\noindent 	In the following, we are going to estimate $c_l$, $l = 1, 2, \cdot\cdot\cdot,N$. Multiplying \eqref{c1} by $Z_{1,t} (t = 1, 2, \cdot\cdot\cdot,N)$ and integrating, we see that $c_l$ satisfies
		\begin{equation}\label{c3}\small
			\aligned
			\sum_{l=1}^{N}&\sum_{j=1}^{m}	\Big\langle (2^{\ast}_{\alpha}-1)\Big(|x|^{-\alpha}\ast |Z_{z_j,\lambda}|^{2^{\ast}_{\alpha}}\Big)Z_{z_j,\lambda}^{2^{\ast}_{\alpha}-2}Z_{j,l}+2^{\ast}_{\alpha}\Big(|x|^{-\alpha}\ast (Z_{z_j,\lambda}^{2^{\ast}_{\alpha}-1}
			Z_{j,l})\Big)Z_{z_j,\lambda}^{2^{\ast}_{\alpha}-1}, Z_{1,t}\Big\rangle c_{l}\\
			=&
			\Big\langle-\Delta \phi_m
			-(2^{\ast}_{\alpha}-1)K(x)\Big(|x|^{-\alpha}\ast (K(x)|Z_{\overline{r},\overline{x}'',\lambda}|^{2^{\ast}_{\alpha}})\Big)Z_{\overline{r},\overline{x}'',\lambda}^{2^{\ast}_{\alpha}-2}\phi_m
		\\	&-2^{\ast}_{\alpha}K(x)\Big(|x|^{-\alpha}\ast (K(x)Z_{\overline{r},\overline{x}'',\lambda}^{2^{\ast}_{\alpha}-1}
			\phi_m)\Big)Z_{\overline{r},\overline{x}'',\lambda}^{2^{\ast}_{\alpha}-1}, Z_{1,t}\Big\rangle-\langle h_m, Z_{1,t}\rangle.
			\endaligned
		\end{equation}
		
	\noindent 	From Lemma 2.1 in \cite{PWY} we get 
		\begin{equation}\label{c4}
			|\langle h_m, Z_{1,t}\rangle|\leq C\lambda^{n_{t}}\|h_m\|_{\ast\ast}.
		\end{equation}
		By standard calculations concerning a cut-off function $\xi$, we have
		$$\aligned
		&\Big|K(x)\Big(|x|^{-\alpha}\ast (K(x)Z_{\overline{r},\overline{x}'',\lambda}^{2^{\ast}_{\alpha}-1}
		\phi_m)\Big)\Big|\\
		\leq& C\|\phi_m\|_{\ast}\int_{\mathbb{R}^{N}}\frac{1}{|y|^{\alpha}}\sum_{j=1}^{m}\frac{\xi(x-y)\lambda^{\frac{N+2-\alpha}{2}}}{(1+\lambda|x-y-z_{j}|)^{N+2-\alpha-\tau_1\frac{N+2-\alpha}{N-2}}}\sum_{j=1}^{m}
		\frac{\lambda^{\frac{N-2}{2}}}{(1+\lambda|x-y-z_{j}|)^{\frac{N-2}{2}+\tau}}dy\\
		\leq& C\|\phi_m\|_{\ast}\sum_{i=1}^{m}\int_{\mathbb{R}^{N}}\frac{1}{|y|^{\alpha}}\frac{\xi(x-y)\lambda^{N-\frac{\alpha}{2}}}{(1+\lambda|x-y-z_{i}|)^{\frac{3N+2}{2}-\alpha+\tau-\tau_1\frac{N+2-\alpha}{N-2}}}dy\\
		&+C\|\phi_m\|_{\ast}\sum_{j\neq i}\int_{\mathbb{R}^{N}}\frac{1}{|y|^{\alpha}}\frac{\xi(x-y)\lambda^{\frac{N+2-\alpha}{2}}}{(1+\lambda|x-y-z_{i}|)^{N+2-\alpha-\tau_1\frac{N+2-\alpha}{N-2}}}
		\frac{\lambda^{\frac{N-2}{2}}}{(1+\lambda|x-y-z_{j}|)^{\frac{N-2}{2}+\tau}}dy\\
		=&O(\frac{m^{2}\|\phi_m\|_{\ast}}{\lambda^{\frac{\alpha}{2}-1}}),
		\endaligned$$
		where
		$$\aligned
		&\int_{\mathbb{R}^{N}}\frac{1}{|y|^{\alpha}}\frac{\xi(x-y)\lambda^{N-\frac{\alpha}{2}}}{(1+\lambda|x-y-z_{i}|)^{\frac{3N+2}{2}-\alpha+\tau-\tau_1\frac{N+2-\alpha}{N-2}}}dy\\
		\leq&\int_{B_{2\delta}(x-(r_0, x_0'' ))}\frac{1}{|y|^{\alpha}}\frac{\lambda^{N-\frac{\alpha}{2}}}{(1+\lambda|x-z_{i}-y|)^{\frac{3N+2}{2}-\alpha+\tau-\tau_1\frac{N+2-\alpha}{N-2}}}dy\\
		\leq&\int_{B_{2\delta}(x-(r_0, x_0'' ))}\frac{\lambda^{N-\frac{\alpha}{2}}}{(\lambda|x-y-z_{i}|)^{N-1}}dy
		=O(\frac{1}{\lambda^{\frac{\alpha}{2}-1}})
		\endaligned$$
		and
		$$\aligned
		&\int_{\mathbb{R}^{N}}\frac{1}{|y|^{\alpha}}\frac{\xi(x-y)\lambda^{\frac{N+2-\alpha}{2}}}{(1+\lambda|x-y-z_{i}|)^{N+2-\alpha-\tau_1\frac{N+2-\alpha}{N-2}}}
		\frac{\lambda^{\frac{N-2}{2}}}{(1+\lambda|x-y-z_{j}|)^{\frac{N-2}{2}+\tau}}dy\\		\leq&\frac{C}{(\lambda|z_{i}-z_{j}|)^{\frac{\tau}{2}}}\int_{\mathbb{R}^{N}}\frac{1}{|y|^{\alpha}}
		\Big(\frac{\xi(x-y)\lambda^{N-\frac{\alpha}{2}}}{(1+\lambda|x-y-z_{i}|)^{\frac{3N+2}{2}-\alpha+\frac{\tau}{2}-\tau_1\frac{N+2-\alpha}{N-2}}}\\
		&\hspace{8mm}+\frac{\xi(x-y)\lambda^{N-\frac{\alpha}{2}}}{(1+\lambda|x-y-z_{j}|)^{\frac{3N+2}{2}-\alpha+\frac{\tau}{2}}-\tau_1\frac{N+2-\alpha}{N-2}}\Big)dy\\
		=&O(\frac{1}{\lambda^{\frac{\alpha}{2}-1}}),  \ \  j\neq i.
		\endaligned$$
		Then, by Lemma \ref{B2} we get
		$$\aligned
		&\int_{\mathbb{R}^N}K(x)\Big(|x|^{-\alpha}\ast (K(x)Z_{\overline{r},\overline{x}'',\lambda}^{2^{\ast}_{\alpha}-1}
		\phi_m)\Big)Z_{\overline{r},\overline{y}'',\lambda}^{2^{\ast}_{\alpha}-1} Z_{1,t}dx\\
		&\leq C\|\phi_m\|_{\ast}\|Z_{\overline{r},\overline{x}'',\lambda}\|_{\ast}\frac{m^{2}}{\lambda^{\frac{\alpha}{2}-1}}
		\int_{\mathbb{R}^N}\sum_{j=1}^{m}
		\frac{\lambda^{\frac{N+2-\alpha}{2}}}{(1+\lambda|x-z_{j}|)^{N+2-\alpha-\tau_1\frac{N+2-\alpha}{N-2}}}
		\frac{\xi\lambda^{\frac{N-2}{2}+n_{t}}}{(1+\lambda|x-z_{1}|)^{N-2}}dx\\
		&\leq C\|\phi_m\|_{\ast}\frac{m^{2}\lambda^{n_{t}}}{\lambda^{\frac{\alpha}{2}-1}}
		\frac{m}{\lambda^{\frac{\alpha}{2}}}\int_{\mathbb{R}^N}
		\frac{1}{(1+|x-\lambda z_{1}|)^{2N-\alpha-\tau_1\frac{2N-\alpha}{N-2}}}dx\\
		&=O(\frac{\lambda^{n_{t}}\|\phi_m\|_{\ast}}{\lambda^{1+\varepsilon}}),
		\endaligned$$
		where $\alpha>5-\frac{6}{N-2}$ for some small constant $\varepsilon>0$. Thus,
		\begin{equation}\label{c61}
			\Big\langle\Big(K(x)|x|^{-\alpha}\ast (K(x)Z_{\overline{r},\overline{x}'',\lambda}^{2^{\ast}_{\alpha}-1}
			\phi_{m})\Big)Z_{\overline{r},\overline{y}'',\lambda}^{2^{\ast}_{\alpha}-1}, Z_{1,t}\Big\rangle= O(\frac{\lambda^{n_{t}}\|\phi_m\|_{\ast}}{\lambda^{1+\varepsilon}}).
		\end{equation}
		Similarly, we also have
		\begin{equation}\label{c62}
			\Big\langle K(x)\Big(|x|^{-\alpha}\ast (K(x)|Z_{\overline{r},\overline{x}'',\lambda}|^{2^{\ast}_{\alpha}})\Big)Z_{\overline{r},\overline{x}'',\lambda}^{2^{\ast}_{\alpha}-2}\phi_m, Z_{1,t}\Big\rangle= O\Big(\frac{\lambda^{n_{t}}\|\phi_m\|_{\ast}}{\lambda^{1+\varepsilon}}\Big).
		\end{equation}
		Notice that
		$$\aligned
		\int_{\mathbb{R}^{N}}\Delta Z_{1,t}\phi_mdx
		=\int_{\mathbb{R}^{N}}\big(\xi\Delta (U_{z_{1},\lambda})_{t}+(U_{z_{1},\lambda})_{t}\Delta\xi +2\nabla\xi\nabla (U_{z_{1},\lambda})_{t}\big) \phi_mdx,
		\endaligned$$
		where
		$$
		(U_{z_{1},\lambda})_{1}=\frac{\partial U_{z_{1},\lambda}}{\partial \lambda}, (U_{z_{1},\lambda})_{2}=\frac{\partial U_{z_{1},\lambda}}{\partial \overline{r}},
		(U_{z_{1},\lambda})_{k}=\frac{\partial U_{z_{1},\lambda}}{\partial \overline{x}_{k}''},
		\ \mbox{for }k=3,\cdot\cdot\cdot,N.
		$$
		On the one hand, applying the above argument we have
		$$\aligned
		&\int_{\mathbb{R}^{N}}\xi\Delta (U_{z_{1},\lambda})_{t} \phi_mdx\\
		&\leq C\|\phi_m\|_{\ast}\sum_{j=1}^{m}\int_{\mathbb{R}^N}\int_{\mathbb{R}^N}
		\frac{\xi\lambda^{N-\frac{\alpha}{2}}}{(1+\lambda|x-z_{1}|)^{2N-\alpha}}
		\frac{1}{|x-y|^{\alpha}}\frac{\lambda^{\frac{N+2-\alpha}{2}+n_{t}}}{(1+\lambda|y-z_{1}|)^{N+2-\alpha}}
		\frac{\lambda^{\frac{N-2}{2}}}{(1+\lambda|y-z_{j}|)^{\frac{N-2}{2}+\tau}}dxdy\\
		&\leq C\|\phi_m\|_{\ast}\sum_{j=1}^{m}\int_{\mathbb{R}^N}
		\frac{\lambda^{\frac{N+2}{2}+n_{t}}}{(1+\lambda|y-z_{1}|)^{N+2}}
		\frac{\lambda^{\frac{N-2}{2}}}{(1+\lambda|y-z_{j}|)^{\frac{N-2}{2}+\tau}}dy\\
		&\leq C\lambda^{n_{t}}\|\phi_m\|_{\ast}\sum_{j=1}^{m}\frac{1}{(\lambda|z_{1}-z_{j}|)^{\frac{N}{2}}}\int_{\Omega_{j}}
		\frac{1}{(1+|y-\lambda z_{j}|)^{N+1}}dy\\
		&=O(\frac{\lambda^{n_{t}}\|\phi_m\|_{\ast}}{\lambda^{\frac{N}{N-2}}})
		=O(\frac{\lambda^{n_{t}}\|\phi_m\|_{\ast}}{\lambda^{1+\varepsilon}}).
		\endaligned$$
		On the other hand, a direct calculation gives
		$$\aligned
		\int_{\mathbb{R}^{N}}(U_{z_{1},\lambda})_{t}\Delta\xi \phi_mdx
		\leq& C\|\phi_m\|_{\ast}
		\sum_{j=1}^{m}\int_{\mathbb{R}^N}
		\frac{|\Delta\xi|\lambda^{\frac{N-2}{2}+n_{t}}}{(1+\lambda|x-z_{1}|)^{N-2}}
		\frac{\lambda^{\frac{N-2}{2}}}{(1+\lambda|x-z_{j}|)^{\frac{N-2}{2}+\tau}}dx\\
		\leq& C\frac{\lambda^{n_{t}}}{\lambda}\|\phi_m\|_{\ast}
		\sum_{j=1}^{m}\int_{\mathbb{R}^N}
		\frac{|\Delta\xi|\lambda^{N}}{(1+\lambda|x- z_{1}|)^{N-2}}
		\frac{1}{(1+\lambda|x- z_{j}|)^{\frac{N}{2}+\tau}}dx\\
		=& O(\frac{\lambda^{n_{t}}\|\phi_m\|_{\ast}}{\lambda^{1+\varepsilon}}),
		\endaligned$$
		where we use the assumption 		
		$$
		\frac{1}{\lambda}\leq \frac{C}{1+\lambda|x-z_{j}|}
		$$
		which holds for any $|(r,x'')-(r_{0},x_{0}'')|\leq 2\delta$. 

		\noindent Similarly, we have
		$$
		\int_{\mathbb{R}^{N}}\nabla\xi\nabla (U_{z_{1},\lambda})_{t} \phi_mdx
		\leq C\|\phi_m\|_{\ast}
		\sum_{j=1}^{m}\int_{\mathbb{R}^N}
		\frac{\lambda^{\frac{N}{2}+n_{t}}}{(1+\lambda|x-z_{1}|)^{N-1}}
		\frac{|\nabla\xi|\lambda^{\frac{N-2}{2}}}{(1+\lambda|x-z_{j}|)^{\frac{N-2}{2}+\tau}}dx
		= O(\frac{\lambda^{n_{t}}\|\phi_m\|_{\ast}}{\lambda^{1+\varepsilon}}).
		$$
		\noindent Thus, we end up with 
		\begin{equation}\label{c63}
			\langle-\Delta \phi_m, Z_{1,t}\rangle= O\Big(\frac{\lambda^{n_{t}}\|\phi_m\|_{\ast}}{\lambda^{1+\varepsilon}}\Big).
		\end{equation}
		
		
	\noindent 	As a consequence of \eqref{c4}-\eqref{c63}, we have
		\begin{equation}\label{c7}\small
			\aligned
			&\Big\langle-\Delta \phi_m
			-(2^{\ast}_{\alpha}-1)K(x)\Big(|x|^{-\alpha}\ast (K(x)|Z_{\overline{r},\overline{x}'',\lambda}|^{2^{\ast}_{\alpha}})\Big)Z_{\overline{r},\overline{x}'',\lambda}^{2^{\ast}_{\alpha}-2}\phi_m
			-2^{\ast}_{\alpha}K(x)\Big(|x|^{-\alpha}\ast (K(x)Z_{\overline{r},\overline{x}'',\lambda}^{2^{\ast}_{\alpha}-1}
			\phi_m)\Big)Z_{\overline{r},\overline{x}'',\lambda}^{2^{\ast}_{\alpha}-1}, Z_{1,t}\Big\rangle\\
			&\hspace{8mm}-\langle h_m, Z_{1,t}\rangle
			=O\Big(\frac{\lambda^{n_{t}}\|\phi_m\|_{\ast}}{\lambda^{1+\varepsilon}}+\lambda^{n_{t}}\|h_m\|_{\ast\ast}\Big).
			\endaligned
		\end{equation}
		Moreover, one can easily check that the following holds 
		\begin{equation}\label{c8}
			\sum_{j=1}^{m}\Big\langle K(x)\Big(|x|^{-\alpha}\ast (K(x)|Z_{z_j,\lambda}|^{2^{\ast}_{\alpha}}\Big)Z_{z_j,\lambda}^{2^{\ast}_{\alpha}-2}Z_{j,l}), Z_{1,t}\Big\rangle=(\overline{c}+o(1))\delta_{tl}\lambda^{n_{l}}\lambda^{n_{t}},
		\end{equation}
		as well as 
		\begin{equation}\label{c81}
			\sum_{j=1}^{m}\Big\langle K(x)\Big(|x|^{-\alpha}\ast (K(x)Z_{z_j,\lambda}^{2^{\ast}_{\alpha}-1}
			Z_{j,l})\Big)Z_{z_j,\lambda}^{2^{\ast}_{\alpha}-1}, Z_{1,t}\Big\rangle=(\overline{c}'+o(1))\lambda^{n_{l}}\lambda^{n_{t}},
		\end{equation}
		for some constant $\overline{c} > 0$ and $\overline{c}' > 0$.
		Substituting \eqref{c7}, \eqref{c8} and \eqref{c81} in \eqref{c3}, we obtain
		\begin{equation}\label{c9}
			c_{l}=\frac{1}{\lambda^{n_{l}}}(o(\|\phi_m\|_{\ast})+O(\|h_m\|_{\ast\ast})).
		\end{equation}
		Thus,
		\begin{equation}\label{c10}
			\|\phi_m\|_{\ast}\leq o(1)+\|h_m\|_{\ast\ast}+\frac{\sum_{j=1}^{N}\frac{1}{(1+\lambda|x-z_{j}|)^{\frac{N-2}{2}+\tau+\theta}}}
			{\sum_{j=1}^{N}\frac{1}{(1+\lambda|x-z_{j}|)^{\frac{N-2}{2}+\tau}}}.
		\end{equation}
		Since $\|\phi_m\|_{\ast}= 1$, we obtain from \eqref{c10} that there is $R > 0$ such that
		\begin{equation}\label{c11}
			\|\lambda^{-\frac{N-2}{2}}\phi_m\|_{L^{\infty}(B_{\frac{R}{\lambda}}(z_{j}))}\geq a>0,
		\end{equation}
		for some $j$. However, $\widetilde{\phi}_m (x)=\lambda^{-\frac{N-2}{2}}\phi_m(\lambda(x-z_{j}))$ 
		converges uniformly, say  to $v\in D^{1,2}(\mathbb{R}^{N})$ in any compact set, and $v$ satisfies
		\begin{equation}\label{ib5}
			-\Delta v
			=(2^{\ast}_{\alpha}-1)\Big(\int_{\mathbb{R}^{N}}\frac{|U_{0,\Lambda}(y)|^{2^{\ast}_{\alpha}}}{|x-y|^{\alpha}}dy\Big)U_{0,\Lambda}^{2^{\ast}_{\alpha}-2}v
			+2^{\ast}_{\alpha}\Big(\int_{\mathbb{R}^{N}}\frac{U_{0,\Lambda}^{2^{\ast}_{\alpha}-1}(y)v(y)}{|x-y|^{\alpha}}dy\Big)U_{0,\Lambda}^{2^{\ast}_{\alpha}-1}\hspace{4.14mm}\mbox{in}\hspace{1.14mm} \mathbb{R}^N,
		\end{equation}
		for some $\Lambda\in[\Lambda_{1}, \Lambda_{1}]$. Since $v$ is perpendicular to the kernel of \eqref{ib5}, by the non-degeneracy of $U_{0,1}$, we deduce that $v = 0$ which is contradiction with \eqref{c11}.
	\end{proof}

	\noindent Together with Lemma \ref{C1}, a direct consequence of Proposition 4.1 in\cite{DFM} is the following 
	\begin{lem}\label{C2}
		There exist $m_0 > 0$ and a constant $C > 0$ independent of $m$, such that for all
		$m \geq m_0$ and all $h\in L^{\infty}(\R^N)$, problem \eqref{c1} has a unique solution $\phi \equiv L_m(h)$. Moreover, one has 
		\begin{equation}\label{c13}
			\|L_m(h)\|_{\ast}\leq C\|h\|_{\ast\ast},\quad |c_l|\leq \frac{C}{\lambda^{n_{l}}}\|h\|_{\ast\ast}.
		\end{equation}
	\end{lem}
	
	\noindent Next we consider:
	\begin{equation}\label{c14}
		\left\{\begin{array}{l}
			\displaystyle -\Delta (Z_{\overline{r},\overline{x}'',\lambda}+\phi)
			-K(r,x'')\Big(|x|^{-\alpha}\ast (K(r,x'')|(Z_{\overline{r},\overline{x}'',\lambda}+\phi)|^{2^{\ast}_{\alpha}})\Big)(Z_{\overline{r},\overline{x}'',\lambda}+\phi)^{2^{\ast}_{\alpha}-1}\\
			\displaystyle \hspace{10.14mm}=\sum_{l=1}^{N}c_{l}\sum_{j=1}^{m}\Big[(2^{\ast}_{\alpha}-1)\Big(|x|^{-\alpha}\ast |Z_{z_j,\lambda}|^{2^{\ast}_{\alpha}}\Big)Z_{z_j,\lambda}^{2^{\ast}_{\alpha}-2}Z_{j,l}+2^{\ast}_{\alpha}\Big(|x|^{-\alpha}\ast (Z_{z_j,\lambda}^{2^{\ast}_{\alpha}-1}
			Z_{j,l})\Big)Z_{z_j,\lambda}^{2^{\ast}_{\alpha}-1}\Big]\hspace{4.14mm}\mbox{in}\hspace{1.14mm} \mathbb{R}^N,\\
			\displaystyle \phi\in H, \ \ \sum_{j=1}^{m}\int_{\mathbb{R}^{N}}\Big[(2^{\ast}_{\alpha}-1)\Big(|x|^{-\alpha}\ast |Z_{z_j,\lambda}|^{2^{\ast}_{\alpha}}\Big)Z_{z_j,\lambda}^{2^{\ast}_{\alpha}-2}Z_{j,l}\phi +2^{\ast}_{\alpha}\Big(|x|^{-\alpha}\ast (Z_{z_j,\lambda}^{2^{\ast}_{\alpha}-1}
			Z_{j,l})\Big)Z_{z_j,\lambda}^{2^{\ast}_{\alpha}-1}\phi\Big] dx=0,\\
			\displaystyle \hspace{10.14mm}l=1,2,\cdots,N.
		\end{array}
		\right.
	\end{equation}
	We can rewrite \eqref{c14} as
	\begin{equation}\label{c16}
		\aligned
		&-\Delta \phi
		-(2^{\ast}_{\alpha}-1)K(r,x'')\Big(|x|^{-\alpha}\ast (K(r,x'')|Z_{\overline{r},\overline{x}'',\lambda}|^{2^{\ast}_{\alpha}})\Big)Z_{\overline{r},\overline{x}'',\lambda}^{2^{\ast}_{\alpha}-2}\phi\\
		&\hspace{1.5cm}-2^{\ast}_{\alpha}K(r,x'')\Big(|x|^{-\alpha}\ast (K(r,x'')Z_{\overline{r},\overline{x}'',\lambda}^{2^{\ast}_{\alpha}-1}\phi)\Big)Z_{\overline{r},\overline{x}'',\lambda}^{2^{\ast}_{\alpha}-1}\\
		&=N(\phi)+l_{m}+\sum_{l=1}^{N}c_{l}\sum_{j=1}^{m}\Big[(2^{\ast}_{\alpha}-1)\Big(|x|^{-\alpha}\ast |Z_{z_j,\lambda}|^{2^{\ast}_{\alpha}}\Big)Z_{z_j,\lambda}^{2^{\ast}_{\alpha}-2}Z_{j,l}\\
		&\quad +2^{\ast}_{\alpha}\Big(|x|^{-\alpha}\ast (Z_{z_j,\lambda}^{2^{\ast}_{\alpha}-1}
		Z_{j,l})\Big)Z_{z_j,\lambda}^{2^{\ast}_{\alpha}-1}\Big]
		\hspace{1.14mm}\mbox{in}\hspace{1.14mm} \mathbb{R}^N,
		\endaligned
	\end{equation}
	where
	$$\aligned
	N(\phi)=&K(r,x'')\Big(|x|^{-\alpha}\ast (K(r,x'')|(Z_{\overline{r},\overline{x}'',\lambda}+\phi)|^{2^{\ast}_{\alpha}})\Big)(Z_{\overline{r},\overline{x}'',\lambda}+\phi)^{2^{\ast}_{\alpha}-1}\\
	&-K(r,x'')\Big(|x|^{-\alpha}\ast (K(r,x'')|Z_{\overline{r},\overline{x}'',\lambda}|^{2^{\ast}_{\alpha}})\Big)Z_{\overline{r},\overline{x}'',\lambda}^{2^{\ast}_{\alpha}-1}
	-2^{\ast}_{\alpha}K(r,x'')\Big(|x|^{-\alpha}\ast (K(r,x'')Z_{\overline{r},\overline{x}'',\lambda}^{2^{\ast}_{\alpha}-1}\phi)\Big)Z_{\overline{r},\overline{x}'',\lambda}^{2^{\ast}_{\alpha}-1}\\
	&-(2^{\ast}_{\alpha}-1)K(r,x'')\Big(|x|^{-\alpha}\ast (K(r,x'')|Z_{\overline{r},\overline{x}'',\lambda}|^{2^{\ast}_{\alpha}})\Big)Z_{\overline{r},\overline{x}'',\lambda}^{2^{\ast}_{\alpha}-2}\phi
	\endaligned$$
	and
	$$\aligned
	l_{m}=&K(r,x'')\Big(|x|^{-\alpha}\ast (K(r,x'')|Z_{\overline{r},\overline{x}'',\lambda}|^{2^{\ast}_{\alpha}})\Big)Z_{\overline{r},\overline{x}'',\lambda}^{2^{\ast}_{\alpha}-1}
	-\sum_{j=1}^{m}\xi\Big(|x|^{-\alpha}\ast |U_{z_j,\lambda}|^{2^{\ast}_{\alpha}}\Big)U_{z_j,\lambda}^{2^{\ast}_{\alpha}-1}\\
	&+Z_{\overline{r},\overline{x}'',\lambda}^{\ast}\Delta\xi+2\nabla\xi\nabla Z_{\overline{r},\overline{x}'',\lambda}^{\ast}.
	\endaligned$$
	
	\noindent In order to apply the contraction mapping theorem to have that \eqref{c16} is uniquely solvable, we need to estimate $N(\phi)$ and $l_m$.

	\begin{lem}\label{C4}
		There is a constant $C> 0$, such that
		\begin{equation}\label{c17}
			\|N(\phi)\|_{\ast\ast}\leq C\|\phi\|_{\ast}^{2}.
		\end{equation}
	\end{lem}
	\begin{proof}
		Notice that $2^{\ast}_{\alpha}-3\leq 0$, so we have
		$$\aligned
		|N(\phi)|\leq &C\left|\Big(|x|^{-\alpha}\ast (Z_{\overline{r},\overline{x}'',\lambda}^{2^{\ast}_{\alpha}-1}\phi)\Big)Z_{\overline{r},\overline{x}'',\lambda}^{2^{\ast}_{\alpha}-2}\phi+\Big(|x|^{-\alpha}\ast (Z_{\overline{r},\overline{x}'',\lambda}^{2^{\ast}_{\alpha}-2}\phi^2)\Big)Z_{\overline{r},\overline{x}'',\lambda}^{2^{\ast}_{\alpha}-1}\right|\\
		&+C\left|\Big(|x|^{-\alpha}\ast (Z_{\overline{r},\overline{x}'',\lambda}^{2^{\ast}_{\alpha}-2}\phi^2)\Big)Z_{\overline{r},\overline{x}'',\lambda}^{2^{\ast}_{\alpha}-2}\phi+\Big(|x|^{-\alpha}\ast \phi^{2^{\ast}_{\alpha}}\Big)\phi^{2^{\ast}_{\alpha}-1}\right|.
		\endaligned$$
		Using \eqref{51411}, Lemma \ref{B4} and the H\"older inequality, we have
		$$\aligned
		&\Big(|x|^{-\alpha}\ast (Z_{\overline{r},\overline{x}'',\lambda}^{2^{\ast}_{\alpha}-1}\phi)\Big)Z_{\overline{r},\overline{x}'',\lambda}^{2^{\ast}_{\alpha}-2}|\phi|\\
		\leq &C \|\phi\|^2_{\ast}\lambda^{\frac{N+2}{2}}\sum_{j=1}^{m}\frac{1}{(1+\lambda|x-z_{j}|)^{\min\{\alpha,\frac{N+2}{2}\}}}\Big(\sum_{j=1}^{m}
		\frac{1}{(1+\lambda|x-z_{j}|)^{\frac{N-2}{2}+\tau}}\Big)^{2^{\ast}_{\alpha}-1}\\
		\leq &C \|\phi\|_{\ast}^{2}
		\lambda^{\frac{N+2}{2}}
		\sum_{j=1}^{m}\frac{1}{(1+\lambda|x-z_{j}|)^{\min\{\alpha,\frac{N+2}{2}\}}}\sum_{j=1}^{m}
		\frac{1}{(1+\lambda|x-z_{j}|)^{\frac{N+2-\alpha}{2}+\tau}}\Big(\sum_{j=1}^{m}\frac{1}{(1+\lambda|x-z_{j}|)^{\tau}}\Big)^{2^{\ast}_{\alpha}-2}\\
		\leq &C \|\phi\|_{\ast}^{2}
		\lambda^{\frac{N+2}{2}}\sum_{j=1}^{m}
		\frac{1}{(1+\lambda|x-z_{j}|)^{\frac{N+2}{2}+\tau}}.
		\endaligned$$
		By the similar argument in \eqref{51411} and Lemma \ref{B4}, we obtain
		$$
		\aligned
		&|x|^{-\alpha}\ast (Z_{\overline{r},\overline{y}'',\lambda}^{2^{\ast}_{\alpha}-2}
		|\phi|^2)\\&\leq C\|\phi\|^2_{\ast}\bigg(|x|^{-\alpha}\ast\Big(\big(\sum_{j=1}^{m}
		\frac{\lambda^{\frac{N-2}{2}}}{(1+\lambda|x-z_{j}|)^{N-2}}\big)^{2^{\ast}_{\alpha}-2}\big(\sum_{i=1}^{m}
		\frac{\lambda^{\frac{N-2}{2}}}{(1+\lambda|x-z_{i}|)^{\frac{N-2}{2}+\tau}}\big)^2\Big)\bigg) \\
		&\leq C\|\phi\|^2_{\ast}\sum_{j=1}^{m}\int_{\Omega_{j}}\frac{1}{|x-y|^{\alpha}}\Big(\frac{\lambda^{\frac{4-\alpha}{2}}}{(1+\lambda|x-z_{j}|)^{4-\alpha-\tau_1\frac{4-\alpha}{N-2}}}
		\frac{\lambda^{N-2}}{(1+\lambda|x-z_{j}|)^{N-2+2\tau-2\tau_1}}\Big)dy \\
		&\leq C\|\phi\|^2_{\ast}\sum_{j=1}^{m}\int_{\Omega_{j}}\frac{1}{|x-y|^{\alpha}}\frac{\lambda^{N-\frac{\alpha}{2}}}{(1+\lambda|x-z_{j}|)^{N+2-\alpha-\tau_1\frac{2N-\alpha}{N-2}+2\tau}}dy \\
		&\leq C\|\phi\|^2_{\ast}\sum_{j=1}^{m}\frac{\lambda^{\frac{\alpha}{2}}}{(1+\lambda|x-z_{j}|)^{\min\{\alpha,2\}+\tau}}.
		\endaligned
		$$
		Then,
		$$\aligned
		&\Big(|x|^{-\alpha}\ast (Z_{\overline{r},\overline{x}'',\lambda}^{2^{\ast}_{\alpha}-2}\phi^2)\Big)Z_{\overline{r},\overline{x}'',\lambda}^{2^{\ast}_{\alpha}-1}\\
		\leq &C \|\phi\|^2_{\ast}\sum_{j=1}^{m}\int_{\Omega_{j}}\frac{1}{|x-y|^{\alpha}}\frac{\lambda^{N-\frac{\alpha}{2}}}{(1+\lambda|y-z_{j}|)^{N+2-\alpha-\tau_1\frac{2N-\alpha}{N-2}+\tau}}dy\Big(\sum_{j=1}^{m}\frac{\lambda^{\frac{N-2}{2}}}{(1+\lambda|x-z_{j}|)^{\frac{N-2}{2}+\tau}}\Big)^{2^{\ast}_{\alpha}-1}\\
		\leq &C \|\phi\|^2_{\ast}\lambda^{\frac{N+2}{2}}\sum_{j=1}^{m}\frac{1}{(1+\lambda|x-z_{j}|)^{\min\{\alpha,2\}+\tau}}\sum_{j=1}^{m}
		\frac{1}{(1+\lambda|x-z_{j}|)^{\frac{N+2-\alpha}{2}+\tau}}\Big(\sum_{j=1}^{m}\frac{1}{(1+\lambda|x-z_{j}|)^{\tau}}\Big)^{2^{\ast}_{\alpha}-2}\\
		\leq &C \|\phi\|_{\ast}^{2}
		\lambda^{\frac{N+2}{2}}\sum_{j=1}^{m}
		\frac{1}{(1+\lambda|x-z_{j}|)^{\frac{N+2}{2}+\tau}}.
		\endaligned$$
		Similarly, we also have
		$$
		\Big(|x|^{-\alpha}\ast (Z_{\overline{r},\overline{x}'',\lambda}^{2^{\ast}_{\alpha}-2}\phi^2)\Big)Z_{\overline{r},\overline{x}'',\lambda}^{2^{\ast}_{\alpha}-2}\phi
		\leq C \|\phi\|_{\ast}^{2}
		\lambda^{\frac{N+2}{2}}\sum_{j=1}^{m}
		\frac{1}{(1+\lambda|x-z_{j}|)^{\frac{N+2}{2}+\tau}}.
		$$
		We also have for the last term, 
		$$
		\aligned
		|x|^{-\alpha}\ast |\phi|^{2^{\ast}_{\alpha}}&\leq C\|\phi\|^{2^{\ast}_{\alpha}}_{\ast}\Big(|x|^{-\alpha}\ast\Big(\sum_{j=1}^{m}\frac{\lambda^{\frac{N-2}{2}}}{(1+\lambda|x-z_{j}|)^{\frac{N-2}{2}+\tau}}\Big)^{2^{\ast}_{\alpha}}\Big)\\
		&\leq C\|\phi\|^{2^{\ast}_{\alpha}}_{\ast}\bigg(|x|^{-\alpha}\ast\Big(\sum_{j=1}^{m}
		\frac{\lambda^{N-\frac{\alpha}{2}}}{(1+\lambda|x-z_{j}|)^{\frac{2N-\alpha}{2}+\tau}}\Big(\sum_{j=1}^{m}\frac{1}{(1+\lambda|x-z_{j}|)^{\tau}}\Big)^{2^{\ast}_{\alpha}-1}\Big) \bigg)\\
		&\leq C\|\phi\|^{2^{\ast}_{\alpha}}_{\ast}\sum_{j=1}^{m}\Big(|x|^{-\alpha}\ast\frac{\lambda^{N-\frac{\alpha}{2}}}{(1+\lambda|x-z_{j}|)^{N-\frac{\alpha}{2}+\tau}} \Big)\\
		&\leq C\|\phi\|^{2^{\ast}_{\alpha}}_{\ast}\sum_{j=1}^{m}\frac{\lambda^{\frac{\alpha}{2}}}{(1+\lambda|x-z_{j}|)^{\frac{\alpha}{2}}},
		\endaligned
		$$
		and eventually the following 
		$$\aligned
		&\Big(|x|^{-\alpha}\ast |\phi|^{2^{\ast}_{\alpha}}\Big)|\phi|^{2^{\ast}_{\alpha}-1}\\
		\leq &C \|\phi\|^{2\cdot2^{\ast}_{\alpha}-1}_{\ast}\lambda^{\frac{N+2}{2}}\sum_{j=1}^{m}\frac{1}{(1+\lambda|x-z_{j}|)^{\frac{\alpha}{2}}}\sum_{j=1}^{m}
		\frac{1}{(1+\lambda|x-z_{j}|)^{\frac{N+2-\alpha}{2}+\tau}}\Big(\sum_{j=1}^{m}\frac{1}{(1+\lambda|x-z_{j}|)^{\tau}}\Big)^{2^{\ast}_{\alpha}-2}\\
		\leq &C \|\phi\|_{\ast}^{2\cdot2^{\ast}_{\alpha}-1}
		\lambda^{\frac{N+2}{2}}\sum_{j=1}^{m}
		\frac{1}{(1+\lambda|x-z_{j}|)^{\frac{N+2}{2}+\tau}}.
		\endaligned$$
		Joining the above estimates yields 
		$$
		\|N(\phi)\|_{\ast\ast}\leq C\|\phi\|_{\ast}^{\min\{2\cdot2^{\ast}_{\alpha}-1,2\}}\leq C\|\phi\|_{\ast}^{2}
		$$
		and in turn 
		$$
		\|N(\phi)\|_{\ast\ast}\leq C\|\phi\|_{\ast}^{2}.
		$$
	\end{proof}

	\begin{lem}\label{C5}
		There is a small constant $\varepsilon> 0$, such that
		\begin{equation}\label{c18}
			\|l_{m}\|_{\ast\ast}\leq C(\frac{1}{\lambda})^{1+\varepsilon}.
		\end{equation}
	\end{lem}
	\begin{proof}
		Observe that
		$$\aligned
		l_{m}=&K(r,x'')\Big(|x|^{-\alpha}\ast (K(r,x'')|Z_{\overline{r},\overline{x}'',\lambda}|^{2^{\ast}_{\alpha}})\Big)Z_{\overline{r},\overline{x}'',\lambda}^{2^{\ast}_{\alpha}-1}
		-K(r,x'')\sum_{j=1}^{m}\xi\Big(|x|^{-\alpha}\ast(K(r,x'') |U_{z_j,\lambda}|^{2^{\ast}_{\alpha}})\Big)U_{z_j,\lambda}^{2^{\ast}_{\alpha}-1}\\
		&+K(r,x'')\sum_{j=1}^{m}\xi\Big(|x|^{-\alpha}\ast(K(r,x'') |U_{z_j,\lambda}|^{2^{\ast}_{\alpha}})\Big)U_{z_j,\lambda}^{2^{\ast}_{\alpha}-1}-K(r,x'')\sum_{j=1}^{m}\xi\Big(|x|^{-\alpha}\ast |U_{z_j,\lambda}|^{2^{\ast}_{\alpha}}\Big)U_{z_j,\lambda}^{2^{\ast}_{\alpha}-1}\\
		&+K(r,x'')\sum_{j=1}^{m}\xi\Big(|x|^{-\alpha}\ast |U_{z_j,\lambda}|^{2^{\ast}_{\alpha}}\Big)U_{z_j,\lambda}^{2^{\ast}_{\alpha}-1}-\sum_{j=1}^{m}\xi\Big(|x|^{-\alpha}\ast |U_{z_j,\lambda}|^{2^{\ast}_{\alpha}}\Big)U_{z_j,\lambda}^{2^{\ast}_{\alpha}-1}
		\\ 
		&+\big(Z_{\overline{r},\overline{x}'',\lambda}^{\ast}\Delta\xi+2\nabla\xi\nabla Z_{\overline{r},\overline{x}'',\lambda}^{\ast}\big)\\
		=&K(r,x'')\Big[\Big(|x|^{-\alpha}\ast (K(r,x'')|Z_{\overline{r},\overline{x}'',\lambda}|^{2^{\ast}_{\alpha}})\Big)Z_{\overline{r},\overline{x}'',\lambda}^{2^{\ast}_{\alpha}-1}
		-\sum_{j=1}^{m}\xi\Big(|x|^{-\alpha}\ast(K(r,x'') |U_{z_j,\lambda}|^{2^{\ast}_{\alpha}})\Big)U_{z_j,\lambda}^{2^{\ast}_{\alpha}-1}\Big]\\
		&+K(r,x'')\sum_{j=1}^{m}\xi\Big(|x|^{-\alpha}\ast\big((K(r,x'')-1) |U_{z_j,\lambda}|^{2^{\ast}_{\alpha}}\big)\Big)U_{z_j,\lambda}^{2^{\ast}_{\alpha}-1}\\
		&+(K(r,x'')-1)\sum_{j=1}^{m}\xi\Big(|x|^{-\alpha}\ast |U_{z_j,\lambda}|^{2^{\ast}_{\alpha}}\Big)U_{z_j,\lambda}^{2^{\ast}_{\alpha}-1}+\big(Z_{\overline{r},\overline{x}'',\lambda}^{\ast}\Delta\xi+2\nabla\xi\nabla Z_{\overline{r},\overline{x}'',\lambda}^{\ast}\big)\\
		=:&J_1+J_2+J_3+J_4.
		\endaligned$$
		By symmetry, we may assume $x \in\Omega_{1}$, and hence $|x-z_{j}|\geq|x-z_{1}|$. For the term  $J_1$, we have
		$$\aligned
		\left|J_1\right|\leq& C\Big(|x|^{-\alpha}\ast |U_{z_1,\lambda}|^{2^{\ast}_{\alpha}}\Big)U_{z_1,\lambda}^{2^{\ast}_{\alpha}-2}\sum_{j=2}^{m}U_{z_j,\lambda}+C\Big(|x|^{-\alpha}\ast (U_{z_1,\lambda}^{2^{\ast}_{\alpha}-1}\sum_{j=2}^{m}U_{z_j,\lambda})\Big)U_{z_1,\lambda}^{2^{\ast}_{\alpha}-1}\\
		&+C\Big(|x|^{-\alpha}\ast \Big(\sum_{j=2}^{m}U_{z_j,\lambda}\Big)^{2^{\ast}_{\alpha}}\Big)\Big(\sum_{i\neq1,j}U_{z_i,\lambda}\Big)^{2^{\ast}_{\alpha}-1}\\
		:=&J_{11}+J_{12}+J_{13}.
		\endaligned$$
		Then, taking $0 < \gamma \leq \min\{4,N-2\}$, by Lemma \ref{B2} and Lemma \ref{P0}, we obtain that for any $x\in\Omega_{1}$ and $j>1$
		$$\aligned
		\Big(|x|^{-\alpha}\ast |U_{z_1,\lambda}|^{2^{\ast}_{\alpha}}\Big)U_{z_1,\lambda}^{2^{\ast}_{\alpha}-2}\sum_{j=2}^{m}U_{z_j,\lambda}\leq& C \frac{\lambda^{\frac{\alpha}{2}}}{(1+\lambda^2|x-z_{1}|^2)^{\frac{\alpha}{2}}}\frac{\lambda^{\frac{4-\alpha}{2}}}{(1+\lambda^2|x-z_{1}|^2)^{\frac{4-\alpha}{2}}}\sum_{j=2}^{m}\frac{\lambda^{\frac{N-2}{2}}}{(1+\lambda^2|x-z_{j}|^2)^{\frac{N-2}{2}}}\\
		\leq&\frac{\lambda^{2}}{(1+\lambda|x-z_{1}|)^{4}}\sum_{j=2}^{m}\frac{\lambda^{\frac{N-2}{2}}}{(1+\lambda|x-z_{j}|)^{N-2}}\\
		\leq&C\frac{\lambda^{\frac{N+2}{2}}}{(1+\lambda|x-z_{1}|)^{N+2-\gamma}}\sum_{j=2}^{m}\frac{1}{(\lambda|z_{1}-z_{j}|)^{\gamma}}\\
		\leq&C\frac{\lambda^{\frac{N+2}{2}}}{(1+\lambda|x-z_{1}|)^{\frac{N+2}{2}+\tau}}\Big(\frac{m}{\lambda}\Big)^{\gamma}\\
		\leq&C\frac{\lambda^{\frac{N+2}{2}}}{(1+\lambda|x-z_{1}|)^{\frac{N+2}{2}+\tau}}\Big(\frac{1}{\lambda}\Big)^{1+\varepsilon}.
		\endaligned$$
		Here we choose $\gamma> \frac{N-2}{2}$ satisfying $N+2-\gamma\geq \frac{N+2}{2}+\tau$. 
		
	\noindent 	For the term $J_{12}$, by taking $0 < \gamma \leq \min\{\alpha,N+2-\alpha\}$ and applying Lemma \ref{B2} again, we obtain that for any $x\in\Omega_{1}$
		$$\aligned
		&\Big(|x|^{-\alpha}\ast (U_{z_1,\lambda}^{2^{\ast}_{\alpha}-1}\sum_{j=2}^{m}U_{z_j,\lambda})\Big)U_{z_1,\lambda}^{2^{\ast}_{\alpha}-1}\\
		\leq&C\Big[|x|^{-\alpha}\ast\Big(\frac{\lambda^\frac{N+2-\alpha}{2}}{(1+\lambda|x-z_{1}|)^{N+2-\alpha}}\sum_{j=2}^{m}\frac{\lambda^{\frac{N-2}{2}}}{(1+\lambda|x-z_{j}|)^{N-2}}\Big)\Big]\frac{\lambda^{\frac{N+2-\alpha}{2}}}{(1+\lambda|x-z_{1}|)^{N+2-\alpha}}\\
		\leq&C\sum_{j=2}^{m}\frac{1}{(\lambda|z_{1}-z_{j}|)^{\frac{N-2}{2}+\varepsilon}}\int_{\Omega_{1}}\frac{1}{|x-y|^{\alpha}}\frac{\lambda^{N-\frac{\alpha}{2}}}{(1+\lambda|y-z_{1}|)^{\frac{3N+2}{2}-\alpha-\varepsilon}}dy\frac{\lambda^{\frac{N+2-\alpha}{2}}}{(1+\lambda|x-z_{1}|)^{N+2-\alpha}}\\
		&+C\sum_{j=2}^{m}\int_{\Omega_{j}}\frac{1}{|x-y|^{\alpha}}\frac{\lambda^{N-\frac{\alpha}{2}}}{(1+\lambda|y-z_{j}|)^{2N-\alpha-\tau_1}}dy\frac{\lambda^{\frac{N+2-\alpha}{2}}}{(1+\lambda|x-z_{1}|)^{N+2-\alpha}}\\
		\leq&C\sum_{j=2}^{m}\frac{1}{(\lambda|z_{1}-z_{j}|)^{\frac{N-2}{2}+\varepsilon}}\frac{\lambda^{\frac{\alpha}{2}}}{(1+\lambda|x-z_{1}|)^{\min\{\alpha,\frac{N+2}{2}-\varepsilon-\varepsilon_1\}}}\frac{\lambda^{\frac{N+2-\alpha}{2}}}{(1+\lambda|x-z_{1}|)^{N+2-\alpha}}\\
		&+C\sum_{j=2}^{m}\frac{\lambda^{\frac{\alpha}{2}}}{(1+\lambda|x-z_{j}|)^{\alpha}}\frac{\lambda^{\frac{N+2-\alpha}{2}}}{(1+\lambda|x-z_{1}|)^{N+2-\alpha}}\\
		\leq&C\sum_{j=2}^{m}\frac{1}{(\lambda|z_{1}-z_{j}|)^{\frac{N-2}{2}+\varepsilon}}\frac{\lambda^{\frac{N+2}{2}}}{(1+\lambda|x-z_{1}|)^{\frac{N+2}{2}+\tau}}
		+C\frac{\lambda^{\frac{N+2}{2}}}{(1+\lambda|x-z_{1}|)^{\frac{N+2}{2}+\tau}}\sum_{j=2}^{m}\frac{1}{(\lambda|z_{1}-z_{j}|)^{\gamma}}\\
		\leq& C\Big(\frac{1}{\lambda}\Big)^{1+\varepsilon}\frac{\lambda^{\frac{N+2}{2}}}{(1+\lambda|x-z_{1}|)^{\frac{N+2}{2}+\tau}},
		\endaligned$$
		where $\varepsilon,\varepsilon_1>0$ small and we choose $\gamma> \frac{N-2}{2}$ satisfying $N+2-\gamma\geq\frac{N+2}{2}+\tau$.
		By H\"older's inequality, for the term $J_{13}$ we have
		$$\aligned
		&\Big(|x|^{-\alpha}\ast \Big(\sum_{j=2}^{m}U_{z_j,\lambda}\Big)^{2^{\ast}_{\alpha}}\Big)\Big(\sum_{i\neq1,j}U_{z_i,\lambda}\Big)^{2^{\ast}_{\alpha}-1}\\
		\leq& C\sum_{j=1}^{m}\int_{\Omega_{j}}\frac{1}{|x-y|^{\alpha}} \frac{\lambda^{N-\frac{\alpha}{2}}}{(1+\lambda|y-z_{j}|)^{2N-\alpha-\tau_1\frac{2N-\alpha}{N-2}}}dy\Big(\sum_{i\neq1,j}U_{z_i,\lambda}\Big)^{2^{\ast}_{\alpha}-1}\\
		\leq& C\sum_{j=1}^{m}\frac{\lambda^{\frac{\alpha}{2}}}{(1+\lambda|x-z_{j}|)^{\alpha}}\sum_{i\neq1,j}
		\frac{\lambda^{\frac{N+2-\alpha}{2}}}{(1+\lambda|x-z_{j}|)^{\frac{N+2-\alpha}{2}+\tau}}\Big(\sum_{i\neq1,j}\frac{1}{(1+\lambda|x-z_{i}|)^{\frac{N+2-\alpha}{4-\alpha}(\frac{N-2}{2}-\frac{N-2}{N+2-\alpha}\tau)}}\Big)^{\frac{4-\alpha}{N-2}}\\
		\leq&C\Big(\frac{m}{\lambda}\Big)^{\frac{N+2-\alpha}{2}-\tau}\frac{\lambda^{\frac{\alpha}{2}}}{(1+\lambda|x-z_{1}|)^{\alpha}}\sum_{i\neq1}
		\frac{\lambda^{\frac{N+2-\alpha}{2}}}{(1+\lambda|x-z_{i}|)^{\frac{N+2-\alpha}{2}+\tau}}+\\
		&+C\Big(\frac{m}{\lambda}\Big)^{\frac{N+2-\alpha}{2}-\tau}\sum_{j=2}^{m}\frac{\lambda^{\frac{\alpha}{2}}}{(1+\lambda|x-z_{j}|)^{\alpha}}\sum_{i\neq1,j}
		\frac{\lambda^{\frac{N+2-\alpha}{2}}}{(1+\lambda|x-z_{i}|)^{\frac{N+2-\alpha}{2}+\tau}}\\
		\leq& C\Big(\frac{m}{\lambda}\Big)^{\frac{N+2}{2}-\tau} \Big(\frac{\lambda^{\frac{N+2}{2}}}{(1+\lambda|x-z_{1}|)^{\frac{N+2}{2}+\tau}}+\sum_{j=2}^{m}\frac{\lambda^{\frac{N+2}{2}}}{(1+\lambda|x-z_{j}|)^{\frac{N+2}{2}+\tau}}\Big)
		\leq C (\frac{1}{\lambda})^{1+\varepsilon}\sum_{j=1}^{m}\frac{\lambda^{\frac{N+2}{2}}}{(1+\lambda|x-z_{j}|)^{\frac{N+2}{2}+\tau}}.
		\endaligned$$
		So we conclude that
		\begin{equation}\label{J1}
			\|J_{1}\|_{\ast\ast}\leq C(\frac{1}{\lambda})^{1+\varepsilon}.
		\end{equation}
		
	\noindent	Next, we estimate the term $J_{3}$. We rewrite $K(x)$ in the neighborhood of $x_0$ using the Taylor expansion as follows
		$$K(x)=K(x_0)+\nabla K (x_0)\cdot\left(x-x_{0}\right)+\frac{1}{2} \frac{\partial^{2} K\left(x_{0}\right)}{\partial x_{i} \partial x_{j}}\left(x_{i}-x_{0 i}\right)\left(x_{j}-x_{0 j}\right)+o\left(\left|x-x_{0}\right|^{2}\right) ,
		$$
		
\noindent 		in the region $|(r, x'')-(r_0, x_0'')| \leq \frac{\delta}{\lambda^{\frac{1}{2}+\epsilon}}$ , where  $\delta>0$  is a fixed constant. Recall that we assume $K(x_{0})=1$. So we have
		\begin{equation}\label{J3.1}
			\aligned
			|J_{3}| & =\Big|\Big(\sum_{i, j=1}^{N} \frac{1}{2} \frac{\partial^{2} K(x_0)}{\partial x_{i} \partial x_{j}}(x_{i}-x_{0 i})(x_{j}-x_{0 j})+o(|x-x_0|^{2})\Big) \sum_{j=1}^{m} \xi\Big(|x|^{-\alpha}\ast |U_{z_j,\lambda}|^{2^{\ast}_{\alpha}}\Big)U_{z_j,\lambda}^{2^{\ast}_{\alpha}-1}\Big| \\
			& \leq \frac{C}{\lambda^{1+2 \varepsilon}} \sum_{j=1}^{m} \frac{\xi \lambda^{\frac{\alpha}{2}}}{\left(1+\lambda\left|x-z_{j}\right|\right)^{\alpha}}\sum_{j=1}^{m} \frac{ \lambda^{\frac{N+2-\alpha}{2}}}{\left(1+\lambda\left|x-z_{j}\right|\right)^{N+2-\alpha}}\\
			& \leq \frac{C}{\lambda^{1+2 \varepsilon}} \sum_{j=1}^{m} \frac{\xi \lambda^{\frac{N+2}{2}}}{\left(1+\lambda\left|x-z_{j}\right|\right)^{N+1}}\\
			&\leq C (\frac{1}{\lambda})^{1+\varepsilon}\sum_{j=1}^{m}\frac{\lambda^{\frac{N+2}{2}}}{(1+\lambda|x-z_{j}|)^{\frac{N+2}{2}+\tau}}.
			\endaligned
		\end{equation}
		
\noindent 		On the other hand, in the region  $\frac{\sigma}{\lambda^{\frac{1}{2}+\varepsilon}} \leq |(r, x'')-(r_0, x_0'')| \leq 2\delta $, we have
		$$\frac{1}{1+\lambda|x-z_{j}|} \leq \frac{C}{\lambda^{\frac{1}{2}-\varepsilon}},
		$$
		where we used the following fact
		$$
		|x-z_{j}| \geq\left|(r, x'')-(r_0, x_0'')\right|-\left|(r_0, x_0'')-(\bar{r}, \bar{x}'')\right| \geq \frac{\sigma}{2 \lambda^{\frac{1}{2}+\varepsilon}}.
		$$
		Then, we have
		\begin{equation}\label{J3.2}
			\aligned
			|J_{3}| & \leq C\sum_{j=1}^{m} \frac{\xi \lambda^{\frac{\alpha}{2}}}{\left(1+\lambda\left|x-z_{j}\right|\right)^{\alpha}}\sum_{j=1}^{m} \frac{ \lambda^{\frac{N+2-\alpha}{2}}}{\left(1+\lambda\left|x-z_{j}\right|\right)^{N+2-\alpha}}\\
			& \leq C \sum_{j=1}^{m} \frac{\xi \lambda^{\frac{N+2}{2}}}{\left(1+\lambda\left|x-z_{j}\right|\right)^{N+1}}\\
			&\leq C (\frac{1}{\lambda})^{1+\varepsilon}\sum_{j=1}^{m}\frac{\lambda^{\frac{N+2}{2}}}{(1+\lambda|x-z_{j}|)^{\frac{N+2}{2}+\tau}}\frac{\xi\lambda^{1+\varepsilon}}{(1+\lambda|x-z_{j}|)^{\frac{N}{2}-\tau}}\\
			&\leq C (\frac{1}{\lambda})^{1+\varepsilon}\sum_{j=1}^{m}\frac{\lambda^{\frac{N+2}{2}}}{(1+\lambda|x-z_{j}|)^{\frac{N+2}{2}+\tau}}\frac{1}{\lambda^{\frac{N^2-8N+16}{4(N-2)}-\frac{N^2-2N+4}{2}\varepsilon}}\\
			&\leq C (\frac{1}{\lambda})^{1+\varepsilon}\sum_{j=1}^{m}\frac{\lambda^{\frac{N+2}{2}}}{(1+\lambda|x-z_{j}|)^{\frac{N+2}{2}+\tau}}.
			\endaligned
		\end{equation}
		Combine \eqref{J3.1} and \eqref{J3.2} to have 
		\begin{equation}\label{J3}
			\|J_{3}\|_{\ast\ast}\leq C(\frac{1}{\lambda})^{1+\varepsilon}.
		\end{equation}
		
		\noindent The same argument applies to $J_2$ and thus 
		\begin{equation}\label{J2}
			\|J_{2}\|_{\ast\ast}\leq C(\frac{1}{\lambda})^{1+\varepsilon}.
		\end{equation}
		Finally, Lemma 2.4 in \cite{PWW} shows that the last term $J_{4}$ satisfies 
		\begin{equation}\label{J4}
			\|J_4\|_{\ast\ast}\leq C(\frac{1}{\lambda})^{1+\varepsilon}.
		\end{equation}
		
		\noindent From  \eqref{J1}, \eqref{J3}, \eqref{J2} and \eqref{J4}, we have
		$$
		\|l_{m}\|_{\ast\ast}\leq C(\frac{1}{\lambda})^{1+\varepsilon}.
		$$
	\end{proof}

\noindent 	Next we apply the contraction mapping argument and prove the main result of this section.
	\begin{lem}\label{C3}
		There is an integer $m_0 > 0$, such that for each $m \geq m_0$, $\lambda\in[L_0m^{\frac{N-2}{N-4}},L_1m^{\frac{N-2}{N-4}}]$,
		$\left|(\overline{r}, \overline{x}'')-(r_0, x_0'')\right|\leq \frac{1}{\lambda^{1-\theta}}$, where $\theta> 0$ is a fixed small constant, \eqref{c14} has a unique solution $\phi= \phi_{\overline{r},\overline{x}'',\lambda}\in H$ satisfying
		\begin{equation}\label{c15}
			\|\phi\|_{\ast}\leq C(\frac{1}{\lambda})^{1+\varepsilon}, \ \ |c_{l}|\leq C(\frac{1}{\lambda})^{1+n_{l}+\varepsilon},
		\end{equation}
		where $\varepsilon> 0$ is a small constant.
	\end{lem}
	\begin{proof}
		Recall that $\lambda\in [L_0m^{\frac{N-2}{N-4}},L_1m^{\frac{N-2}{N-4}}]$. Let
		$$\aligned
		\mathcal{N}=&\Big\{w:w\in C(\mathbb{R}^N)\cap H,\|w\|_{\ast}\leq\frac{1}{\lambda},\\
		&\sum_{j=1}^{m}\int_{\mathbb{R}^{N}} \Big[(2^{\ast}_{\alpha}-1)\Big(|x|^{-\alpha}\ast |Z_{z_j,\lambda}|^{2^{\ast}_{\alpha}}\Big)Z_{z_j,\lambda}^{2^{\ast}_{\alpha}-2}Z_{j,l}w +2^{\ast}_{\alpha}\Big(|x|^{-\alpha}\ast (Z_{z_j,\lambda}^{2^{\ast}_{\alpha}-1}
		Z_{j,l})\Big)Z_{z_j,\lambda}^{2^{\ast}_{\alpha}-1}w\Big] dx=0.\Big\},
		\endaligned$$
		where $l=1,2,\cdots, N$. 
		Then, by Lemma \ref{C2}, problem \eqref{c16} is equivalent to
		\begin{equation}\label{c19}
			\phi=\mathcal{T}(\phi)=:L_{m}(N(\phi))+L_{m}(l_{m}),
		\end{equation}
		where $L_m$ is defined in Lemma \ref{C2}. We want to show that  $\mathcal{T}$ is a contraction map from $\mathcal{N}$ to $\mathcal{N}$.
		
		\noindent For any $\phi \in\mathcal{N}$,
		$$
		\|\mathcal{T}(\phi)\|_{\ast}\leq C(\|N(\phi)\|_{\ast\ast}+\|l_{m}\|_{\ast\ast})\leq C\big(\|\phi\|_{\ast}^{2}+(\frac{1}{\lambda})^{1+\varepsilon}\big)\leq
		\frac{1}{\lambda},
		$$
		hence, $\mathcal{T}$ maps $\mathcal{N}$ into $\mathcal{N}$.
		
		\noindent For any $\phi_1,\phi_2 \in \mathcal{N}$,
		$$
		\|\mathcal{T}(\phi_{1})-\mathcal{T}(\phi_{2})\|_{\ast}
		=\|L_{m}(N(\phi_{1}))-L_{m}(N(\phi_{2}))\|_{\ast}\leq C\|N(\phi_{1})-N(\phi_{2})\|_{\ast\ast}.
		$$
		It is easy to check that
		$$\aligned
		\left|N(\phi_{1})-N(\phi_{2})\right|
		\leq& \left|N'(\phi_1+\theta(\phi_2-\phi_1))\right|\left|\phi_2-\phi_1\right|\\
		\leq& C\big(G(\phi_1)+G(\phi_2)\big)\left|\phi_2-\phi_1\right|,
		\endaligned$$
		where
		$$\aligned
		G(\phi)=&\Big(|x|^{-\alpha}\ast Z_{\overline{r},\overline{x}'',\lambda}^{2^{\ast}_{\alpha}-1}\Big)Z_{\overline{r},\overline{x}'',\lambda}^{2^{\ast}_{\alpha}-2}\phi+\Big(|x|^{-\alpha}\ast (Z_{\overline{r},\overline{x}'',\lambda}^{2^{\ast}_{\alpha}-1}\phi)\Big)Z_{\overline{r},\overline{x}'',\lambda}^{2^{\ast}_{\alpha}-2}
		\\
		&+\Big(|x|^{-\alpha}\ast (Z_{\overline{r},\overline{x}'',\lambda}^{2^{\ast}_{\alpha}-2}\phi)\Big)(Z_{\overline{r},\overline{x}'',\lambda}^{2^{\ast}_{\alpha}-1}+Z_{\overline{r},\overline{x}'',\lambda}^{2^{\ast}_{\alpha}-2}\phi)+\Big(|x|^{-\alpha}\ast (Z_{\overline{r},\overline{x}'',\lambda}^{2^{\ast}_{\alpha}-2}\phi^{2})\Big)Z_{\overline{r},\overline{x}'',\lambda}^{2^{\ast}_{\alpha}-2}\\
		&+\Big(|x|^{-\alpha}\ast \phi^{2^{\ast}_{\alpha}-1}\Big)\phi^{2^{\ast}_{\alpha}-1}+\Big(|x|^{-\alpha}\ast \phi^{2^{\ast}_{\alpha}}\Big)\phi^{2^{\ast}_{\alpha}-2}.
		\endaligned$$
		
		\noindent According to the proof of Lemma \ref{C4}, we have
		$$
		\|\mathcal{T}(\phi_{1})-\mathcal{T}(\phi_{2})\|_{\ast}\leq C\|N(\phi_{1})-N(\phi_{2})\|_{\ast\ast}\leq C(\|\phi_1\|_{\ast}+\|\phi_2\|_{\ast})\|\phi_2-\phi_1\|_{\ast}\leq \frac{1}{2}\|\phi_{1}-\phi_{2}\|_{\ast},
		$$
		which means that $\mathcal{T}$ is a contraction map. Thus by the contraction mapping theorem, there exists a unique $\phi\in\mathcal{N}$ such that \eqref{c19} holds. Moreover, by Lemmas \ref{C2}, \ref{C4} and \ref{C5}, we obtain
		$$
		\|\phi\|_{\ast}\leq C(\frac{1}{\lambda})^{1+\varepsilon}
		$$
		and the estimate of $c_l$ from \eqref{c13}.
	\end{proof}

	%
	
	\section{The energy expansion}
	\noindent In this section we establish energy estimates. Due to the nonlocal convolution part with non constant potential some extra challenges show up. For this reason, we need to prove fine enough estimates in order to handle small terms. Recall that the functional corresponding to \eqref{eq1} is given by 
	$$
	J(u)=\frac{1}{2}\int_{\mathbb{R}^{N}}|\nabla u|^{2}dx
	-\frac{1}{2\cdot2^{\ast}_{\alpha}}
	\int_{\mathbb{R}^N}\int_{\mathbb{R}^N}
	\frac{K(r,x'')K(r,y'')|u(x)|^{2^{\ast}_{\alpha}}|u(y)|^{2^{\ast}_{\alpha}}}{|x-y|^{\alpha}}dxdy.
	$$
	\begin{lem}\label{P2}
		We have
		\begin{multline*}
		\int_{\mathbb{R}^N}\int_{\mathbb{R}^N}
		\frac{|Z_{\overline{r},\overline{x}'',\lambda}^{\ast}(x)|^{2^{\ast}_{\alpha}}
			(Z_{\overline{r},\overline{x}'',\lambda}^{\ast})^{2^{\ast}_{\alpha}-1}(y)\frac{\partial Z_{\overline{r},\overline{x}'',\lambda}^{\ast}}{\partial \lambda}(y)}{|x-y|^{\alpha}}dxdy
			\\ -\int_{\mathbb{R}^N}\int_{\mathbb{R}^N}
		\frac{|Z_{\overline{r},\overline{x}'',\lambda}(x)|^{2^{\ast}_{\alpha}}
			Z_{\overline{r},\overline{x}'',\lambda}^{2^{\ast}_{\alpha}-1}(y)\frac{\partial Z_{\overline{r},\overline{x}'',\lambda}}{\partial \lambda}(y)}{|x-y|^{\alpha}}dxdy=O(\frac{m}{\lambda^{3+\varepsilon}}).
		\end{multline*}
	\end{lem}
	\begin{proof}
		It is easy to check that
		$$\aligned
		&\int_{\mathbb{R}^N}\int_{\mathbb{R}^N}
		\frac{|U_{z_j,\lambda}(x)|^{2^{\ast}_{\alpha}}
			U_{z_j,\lambda}^{2^{\ast}_{\alpha}-1}(y)\frac{\partial U_{z_j,\lambda}}{\partial \lambda}(y)}{|x-y|^{\alpha}}dxdy-\int_{\mathbb{R}^N}\int_{\mathbb{R}^N}
		\frac{|Z_{z_j,\lambda}(x)|^{2^{\ast}_{\alpha}}
			Z_{z_j,\lambda}^{2^{\ast}_{\alpha}-1}(y)\frac{\partial Z_{z_j,\lambda}}{\partial \lambda}(y)}{|x-y|^{\alpha}}dxdy\\
		=&\int_{\mathbb{R}^N}\int_{\mathbb{R}^N}
		\frac{(1-\xi^{2^{\ast}_{\alpha}}(x))|U_{z_j,\lambda}(x)|^{2^{\ast}_{\alpha}}
			U_{z_j,\lambda}^{2^{\ast}_{\alpha}-1}(y)\frac{\partial U_{z_j,\lambda}}{\partial \lambda}(y)}{|x-y|^{\alpha}}dxdy\\
			&+\int_{\mathbb{R}^N}\int_{\mathbb{R}^N}
		\frac{|\xi U_{z_j,\lambda}(x)|^{2^{\ast}_{\alpha}}
			(1-\xi^{2^{\ast}_{\alpha}})U_{z_j,\lambda}^{2^{\ast}_{\alpha}-1}(y)\frac{\partial U_{z_j,\lambda}}{\partial \lambda}(y)}{|x-y|^{\alpha}}dxdy,
		\endaligned$$
		where $j=1,2, \cdot\cdot\cdot,m$. By the Hardy-Littlewood-Sobolev inequality, we have
		$$\aligned
		&\left|\int_{\mathbb{R}^N}\int_{\mathbb{R}^N}
		\frac{(1-\xi^{2^{\ast}_{\alpha}}(x))|U_{z_j,\lambda}(x)|^{2^{\ast}_{\alpha}}
			U_{z_j,\lambda}^{2^{\ast}_{\alpha}-1}(y)\frac{\partial U_{z_j,\lambda}}{\partial \lambda}(y)}{|x-y|^{\alpha}}dxdy\right|\\
		\leq&C\frac{1}{\lambda}\int_{\mathbb{R}^{N}}\int_{\mathbb{R}^{N}}\frac{(1-\xi^{2^{\ast}_{\alpha}}(x+ z_{j}))\lambda^{N-\frac{\alpha}{2}}}{(1+\lambda^{2}|x |^{2})^{N-\frac{\alpha}{2}}}\frac{1}{|x-y|^{\alpha}}\frac{\lambda^{N-\frac{\alpha}{2}}}
		{(1+\lambda^{2}|y|^{2})^{N-\frac{\alpha}{2}}}dxdy\\
		\leq &C\frac{1}{\lambda} \left(\int_{\mathbb{R}^{N}}\left[\frac{(1-\xi^{2^{\ast}_{\alpha}}(x+z_{j}))^{\frac{1}{2}}\lambda^{N-\frac{\alpha}{2}}}{(1+\lambda^{2}|x |^{2})^{N-\frac{\alpha}{2}}}\right]^{\frac{N}{N-\frac{\alpha}{2}}}dx\right)^{\frac{N-\frac{\alpha}{2}}{N}}
		\left(\int_{\mathbb{R}^{N}}\left[\frac{(1-\xi^{2^{\ast}_{\alpha}}(x+z_{j}))^{\frac{1}{2}}\lambda^{N-\frac{\alpha}{2}}}{(1+\lambda^{2}|y|^{2})^{N-\frac{\alpha}{2}}}\right]^{\frac{N}{N-\frac{\alpha}{2}}}dy\right)^{\frac{N-\frac{\alpha}{2}}{N}}\\
		=&O(\frac{1}{\lambda^{2N+1-\alpha}}),
		\endaligned$$
		where we have used the following direct calculation 
		\begin{equation}\label{z1}
			\aligned
			&\int_{\mathbb{R}^{N}}\left[\frac{(1-\xi^{2^{\ast}_{\alpha}}(x+z_{j}))^{\frac{1}{2}}\lambda^{N-\frac{\alpha}{2}}}{(1+\lambda^{2}|x |^{2})^{N-\frac{\alpha}{2}}}\right]^{\frac{N}{N-\frac{\alpha}{2}}}dx\leq C \int_{\mathbb{R}^{N}\setminus B_{\delta-\vartheta} (0)}\left[\frac{\lambda^{N-\frac{\alpha}{2}}}{(1+\lambda^{2}|x |^{2})^{N-\frac{\alpha}{2}}}\right]^{\frac{N}{N-\frac{\alpha}{2}}}dx
			=O(\frac{1}{\lambda^{N}})\ .
			\endaligned
		\end{equation}
		Likewise 
		$$
		\int_{\mathbb{R}^N}\int_{\mathbb{R}^N}
		\frac{|\xi U_{z_j,\lambda}(x)|^{2^{\ast}_{\alpha}}
			(1-\xi^{2^{\ast}_{\alpha}}(y))U_{z_j,\lambda}^{2^{\ast}_{\alpha}-1}(y)\frac{\partial U_{z_j,\lambda}}{\partial \lambda}(y)}{|x-y|^{\alpha}}dxdy=O(\frac{1}{\lambda^{2N+1-\alpha}}).
		$$
		Thus, we have
		$$\aligned
		\int_{\mathbb{R}^N}\int_{\mathbb{R}^N}
		&\frac{|U_{z_j,\lambda}(x)|^{2^{\ast}_{\alpha}}
			U_{z_j,\lambda}^{2^{\ast}_{\alpha}-1}(y)\frac{\partial U_{z_j,\lambda}}{\partial \lambda}(y)}{|x-y|^{\alpha}}dxdy-\int_{\mathbb{R}^N}\int_{\mathbb{R}^N}
		\frac{|Z_{z_j,\lambda}(x)|^{2^{\ast}_{\alpha}}
			Z_{z_j,\lambda}^{2^{\ast}_{\alpha}-1}(y)\frac{\partial Z_{z_j,\lambda}}{\partial \lambda}(y)}{|x-y|^{\alpha}}dxdy\\
		&=O(\frac{1}{\lambda^{2N+1-\alpha}}),
		\endaligned$$
		where $j=1,2, \cdot\cdot\cdot,m$. 
		
		\noindent Thus,
		$$\aligned
		\int_{\mathbb{R}^N}\int_{\mathbb{R}^N}
		&\frac{|Z_{\overline{r},\overline{x}'',\lambda}^{\ast}(x)|^{2^{\ast}_{\alpha}}
			(Z_{\overline{r},\overline{x}'',\lambda}^{\ast})^{2^{\ast}_{\alpha}-1}(y)\frac{\partial Z_{\overline{r},\overline{x}'',\lambda}^{\ast}}{\partial \lambda}(y)}{|x-y|^{\alpha}}dxdy-\int_{\mathbb{R}^N}\int_{\mathbb{R}^N}
		\frac{|Z_{\overline{r},\overline{x}'',\lambda}(x)|^{2^{\ast}_{\alpha}}
			Z_{\overline{r},\overline{x}'',\lambda}^{2^{\ast}_{\alpha}-1}(y)\frac{\partial Z_{\overline{r},\overline{x}'',\lambda}}{\partial \lambda}(y)}{|x-y|^{\alpha}}dxdy\\
		&=O(\frac{m^{2\cdot2^{\ast}_{\alpha}}}{\lambda^{2N+1-\alpha}})=O(\frac{m}{\lambda^{3+\varepsilon}}).
		\endaligned$$
	\end{proof}
	
	\begin{lem}\label{P3}
		We have
		$$
		\frac{\partial J(Z_{\overline{r},\overline{x}'',\lambda})}{\partial \lambda}=\frac{\partial J(Z_{\overline{r},\overline{x}'',\lambda}^{\ast})}{\partial \lambda}+O(\frac{m}{\lambda^{3+\varepsilon}})\ .
		$$
	\end{lem}
	\begin{proof}
	Observe that
		\begin{equation}\label{dc_1}\aligned
		\frac{\partial J(Z_{\overline{r},\overline{x}'',\lambda}^{\ast})}{\partial \lambda}-\frac{\partial J(Z_{\overline{r},\overline{x}'',\lambda})}{\partial \lambda}
		=&\int_{\mathbb{R}^{N}}\Delta Z_{\overline{r},\overline{x}'',\lambda}^{\ast}\frac{\partial Z_{\overline{r},\overline{x}'',\lambda}^{\ast}}{\partial \lambda}dx-\int_{\mathbb{R}^{N}}\Delta Z_{\overline{r},\overline{x}'',\lambda}\frac{\partial Z_{\overline{r},\overline{x}'',\lambda}}{\partial \lambda}dx\\
		&+\Big(\int_{\mathbb{R}^N}\int_{\mathbb{R}^N}
		\frac{K(x)K(y)|Z_{\overline{r},\overline{x}'',\lambda}^{\ast}(x)|^{2^{\ast}_{\alpha}}
			(Z_{\overline{r},\overline{x}'',\lambda}^{\ast})^{2^{\ast}_{\alpha}-1}(y)\frac{\partial Z_{\overline{r},\overline{x}'',\lambda}^{\ast}}{\partial \lambda}(y)}{|x-y|^{\alpha}}dxdy\\
		&-\int_{\mathbb{R}^N}\int_{\mathbb{R}^N}
		\frac{K(x)K(y)|Z_{\overline{r},\overline{x}'',\lambda}(x)|^{2^{\ast}_{\alpha}}
			Z_{\overline{r},\overline{x}'',\lambda}^{2^{\ast}_{\alpha}-1}(y)\frac{\partial Z_{\overline{r},\overline{x}'',\lambda}}{\partial \lambda}(y)}{|x-y|^{\alpha}}dxdy\Big)\ .
		\endaligned
		\end{equation}
		Concerning the first two terms in \eqref{dc_1} we have 
		$$\aligned
		\int_{\mathbb{R}^{N}}&\Delta Z_{\overline{r},\overline{x}'',\lambda}^{\ast}\frac{\partial Z_{\overline{r},\overline{x}'',\lambda}^{\ast}}{\partial \lambda}dx-\int_{\mathbb{R}^{N}}\Delta Z_{\overline{r},\overline{x}'',\lambda}\frac{\partial Z_{\overline{r},\overline{x}'',\lambda}}{\partial \lambda}dx\\
		&=\int_{\mathbb{R}^{N}}\Delta Z_{\overline{r},\overline{x}'',\lambda}^{\ast}\frac{\partial Z_{\overline{r},\overline{x}'',\lambda}^{\ast}}{\partial \lambda}dx-\int_{\mathbb{R}^{N}}\xi(\xi\Delta Z_{\overline{r},\overline{x}'',\lambda}^{\ast}+Z_{\overline{r},\overline{x}'',\lambda}^{\ast}\Delta\xi +2\nabla\xi\nabla Z_{\overline{r},\overline{x}'',\lambda}^{\ast})\frac{\partial Z_{\overline{r},\overline{x}'',\lambda}^{\ast}}{\partial \lambda}dx.
		\endaligned$$
		By Lemma \ref{P2}, we have
		\begin{multline*}
		\int_{\mathbb{R}^{N}}(1-\xi^{2})\Delta Z_{\overline{r},\overline{x}'',\lambda}^{\ast}\frac{\partial Z_{\overline{r},\overline{x}'',\lambda}^{\ast}}{\partial \lambda}dx
		=\sum_{j=1}^{m}\int_{\mathbb{R}^N}\int_{\mathbb{R}^N}
		\frac{(1-\xi^{2}(x))|U_{z_j,\lambda}(x)|^{2^{\ast}_{\alpha}}
			U_{z_j,\lambda}^{2^{\ast}_{\alpha}-1}(y)\frac{\partial Z_{\overline{r},\overline{x}'',\lambda}^{\ast}}{\partial \lambda}(y)}{|x-y|^{\alpha}}dxdy\\
		=O(\frac{m^{2}}{\lambda^{2N+1-\alpha}}).
		\end{multline*}
		By inspection one has 
		\begin{equation}\label{z1}
			\aligned
			\left|\int_{\mathbb{R}^N}\xi\Delta\xi
			U_{z_j,\lambda}(y)\frac{\partial U_{z_j,\lambda}}{\partial \lambda}(y)dy\right|
			&\leq C\left|\int_{\mathbb{R}^{N}}\frac{\xi(y)\lambda^{\frac{N-2}{2}}}{(1+\lambda^{2}|y-z_{j} |^{2})^{\frac{N-2}{2}}}\frac{\lambda^{\frac{N-4}{2}}}
			{(1+\lambda^{2}|y-z_{j} |^{2})^{\frac{N-2}{2}}}dy\right|\\
			&\leq C\int_{\mathbb{R}^{N}}\frac{\xi(y+ z_{j})\lambda^{N-3}}{(1+\lambda^{2}|y |^{2})^{N-2}}dy\\
			&\leq C\int_{\mathbb{R}^{N}\setminus B_{\delta-\vartheta} (0)}\frac{\lambda^{N-3}}{\lambda^{2N-4}|y |^{2N-4}}dy
			=O(\frac{1}{\lambda^{N-1}}).
			\endaligned
		\end{equation}
		So we have
		$$
		\int_{\mathbb{R}^{N}}\xi Z_{\overline{r},\overline{x}'',\lambda}^{\ast}\Delta\xi \frac{\partial Z_{\overline{r},\overline{x}'',\lambda}^{\ast}}{\partial \lambda}dx=O(\frac{m^{2}}{\lambda^{N-1}}),
		$$
		and similarly we get
		$$
		\int_{\mathbb{R}^{N}}\xi\nabla\xi\nabla Z_{\overline{r},\overline{x}'',\lambda}^{\ast}\frac{\partial Z_{\overline{r},\overline{x}'',\lambda}^{\ast}}{\partial \lambda}dx
		=O(\frac{m^{2}}{\lambda^{N-1}}).
		$$
		We can conclude that
		\begin{equation}\label{z11}
			\int_{\mathbb{R}^{N}}\Delta Z_{\overline{r},\overline{x}'',\lambda}^{\ast}\frac{\partial Z_{\overline{r},\overline{x}'',\lambda}^{\ast}}{\partial \lambda}dx-\int_{\mathbb{R}^{N}}\Delta Z_{\overline{r},\overline{x}'',\lambda}\frac{\partial Z_{\overline{r},\overline{x}'',\lambda}}{\partial \lambda}dx=O(\frac{m^{2}}{\lambda^{N-1}})+O(\frac{m^{2}}{\lambda^{2N+1-\alpha}})=O(\frac{m}{\lambda^{3+\varepsilon}}).
		\end{equation}
		
		\noindent Let us estimate the remaining terms in \eqref{dc_1}. By Lemma \ref{P2} one has 
		\begin{equation}\label{z22}
			\aligned
			&\int_{\mathbb{R}^N}\int_{\mathbb{R}^N}
			\frac{K(x)K(y)|Z_{\overline{r},\overline{x}'',\lambda}^{\ast}(x)|^{2^{\ast}_{\alpha}}
				(Z_{\overline{r},\overline{x}'',\lambda}^{\ast})^{2^{\ast}_{\alpha}-1}(y)\frac{\partial Z_{\overline{r},\overline{x}'',\lambda}^{\ast}}{\partial \lambda}(y)}{|x-y|^{\alpha}}dxdy\\
			&\hspace{8mm}-\int_{\mathbb{R}^N}\int_{\mathbb{R}^N}
			\frac{K(x)K(y)|Z_{\overline{r},\overline{x}'',\lambda}(x)|^{2^{\ast}_{\alpha}}
				Z_{\overline{r},\overline{x}'',\lambda}^{2^{\ast}_{\alpha}-1}(y)\frac{\partial Z_{\overline{r},\overline{x}'',\lambda}}{\partial \lambda}(y)}{|x-y|^{\alpha}}dxdy\\
			&\leq C\left|\int_{\mathbb{R}^N}\int_{\mathbb{R}^N}
			\frac{|Z_{\overline{r},\overline{x}'',\lambda}^{\ast}(x)|^{2^{\ast}_{\alpha}}
				(Z_{\overline{r},\overline{x}'',\lambda}^{\ast})^{2^{\ast}_{\alpha}-1}(y)\frac{\partial Z_{\overline{r},\overline{x}'',\lambda}^{\ast}}{\partial \lambda}(y)}{|x-y|^{\alpha}}dxdy\right. \\
			&\left. \hspace{8mm}-\int_{\mathbb{R}^N}\int_{\mathbb{R}^N}
			\frac{|Z_{\overline{r},\overline{x}'',\lambda}(x)|^{2^{\ast}_{\alpha}}
				Z_{\overline{r},\overline{x}'',\lambda}^{2^{\ast}_{\alpha}-1}(y)\frac{\partial Z_{\overline{r},\overline{x}'',\lambda}}{\partial \lambda}(y)}{|x-y|^{\alpha}}dxdy\right|\\
			&=O(\frac{m}{\lambda^{3+\varepsilon}})
			\endaligned
		\end{equation}
		
\noindent 		Combining \eqref{z22} with \eqref{z11} yields the proof. 
	\end{proof}

	\begin{lem}\label{D2}
		We have
		$$
		\frac{\partial J(Z_{\overline{r},\overline{x}'',\lambda})}{\partial \lambda}=m\Big(-\frac{A_{1}}{\lambda^{3}}+\sum_{j=2}^{m}
		\frac{A_{2}}{\lambda^{N-1}|z_{1}-z_{j}|^{N-2}}+O(\frac{1}{\lambda^{3+\varepsilon}})\Big),
		$$
		for some positive constants $A_j$, $j=1, 2$.
	\end{lem}
	\begin{proof}
		By Lemma \ref{P3}, we know that
		\begin{equation}\label{D3}
			\frac{\partial J(Z_{\overline{r},\overline{x}'',\lambda})}{\partial \lambda}=\frac{\partial J(Z_{\overline{r},\overline{x}'',\lambda}^{\ast})}{\partial \lambda}+O(\frac{m}{\lambda^{3+\varepsilon}}).
		\end{equation}
		Let us calculate 
		\begin{multline}\label{ExpenL4.4}
			\frac{\partial J(Z_{\overline{r},\overline{x}'',\lambda}^{\ast})}{\partial \lambda}=-\int_{\mathbb{R}^{N}}\Delta Z_{\overline{r},\overline{x}'',\lambda}^{\ast}\frac{\partial Z_{\overline{r},\overline{x}'',\lambda}^{\ast}}{\partial \lambda}dx\\
			-\int_{\mathbb{R}^{N}}K(r,x'')\Big(|x|^{-\alpha}\ast K(r,x'')|Z_{\overline{r},\overline{x}'',\lambda}^{\ast}|^{2^{\ast}_{\alpha}}\Big)(Z_{\overline{r},\overline{x}'',\lambda}^{\ast})^{2^{\ast}_{\alpha}-1}
			\frac{\partial Z_{\overline{r},\overline{x}'',\lambda}^{\ast}}{\partial \lambda}dx\\
			=\sum_{j=1}^{m}\int_{\mathbb{R}^{N}}\Big(|x|^{-\alpha}\ast |U_{z_j,\lambda}|^{2^{\ast}_{\alpha}}\Big)U_{z_j,\lambda}^{2^{\ast}_{\alpha}-1}\frac{\partial Z_{\overline{r},\overline{x}'',\lambda}^{\ast}}{\partial \lambda}dx\\
			-\int_{\mathbb{R}^{N}}K(r,x'')\Big(|x|^{-\alpha}\ast K(r,x'')|Z_{\overline{r},\overline{x}'',\lambda}^{\ast}|^{2^{\ast}_{\alpha}}\Big)(Z_{\overline{r},\overline{x}'',\lambda}^{\ast})^{2^{\ast}_{\alpha}-1}\frac{\partial Z_{\overline{r},\overline{x}'',\lambda}^{\ast}}{\partial \lambda}dx\\
			=\sum_{j=1}^{m}\int_{\mathbb{R}^{N}}\Big(|x|^{-\alpha}\ast |U_{z_j,\lambda}|^{2^{\ast}_{\alpha}}\Big)U_{z_j,\lambda}^{2^{\ast}_{\alpha}-1}\frac{\partial Z_{\overline{r},\overline{x}'',\lambda}^{\ast}}{\partial \lambda}dx\\
			-\int_{\mathbb{R}^{N}}\Big(|x|^{-\alpha}\ast |Z_{\overline{r},\overline{x}'',\lambda}^{\ast}|^{2^{\ast}_{\alpha}}\Big)(Z_{\overline{r},\overline{x}'',\lambda}^{\ast})^{2^{\ast}_{\alpha}-1}
			\frac{\partial Z_{\overline{r},\overline{x}'',\lambda}^{\ast}}{\partial \lambda}dx\\
			+\int_{\mathbb{R}^{N}}\Big(|x|^{-\alpha}\ast |Z_{\overline{r},\overline{x}'',\lambda}^{\ast}|^{2^{\ast}_{\alpha}}\Big)(Z_{\overline{r},\overline{x}'',\lambda}^{\ast})^{2^{\ast}_{\alpha}-1}
			\frac{\partial Z_{\overline{r},\overline{x}'',\lambda}^{\ast}}{\partial \lambda}dx\\
			-\int_{\mathbb{R}^{N}}K(r,x'')\Big(|x|^{-\alpha}\ast |Z_{\overline{r},\overline{x}'',\lambda}^{\ast}|^{2^{\ast}_{\alpha}}\Big)(Z_{\overline{r},\overline{x}'',\lambda}^{\ast})^{2^{\ast}_{\alpha}-1}
			\frac{\partial Z_{\overline{r},\overline{x}'',\lambda}^{\ast}}{\partial \lambda}dx\\
			+\int_{\mathbb{R}^{N}}K(r,x'')\Big(|x|^{-\alpha}\ast |Z_{\overline{r},\overline{x}'',\lambda}^{\ast}|^{2^{\ast}_{\alpha}}\Big)(Z_{\overline{r},\overline{x}'',\lambda}^{\ast})^{2^{\ast}_{\alpha}-1}
			\frac{\partial Z_{\overline{r},\overline{x}'',\lambda}^{\ast}}{\partial \lambda}dx\\
			-\int_{\mathbb{R}^{N}}K(r,x'')\Big(|x|^{-\alpha}\ast K(r,x'')|Z_{\overline{r},\overline{x}'',\lambda}^{\ast}|^{2^{\ast}_{\alpha}}\Big)(Z_{\overline{r},\overline{x}'',\lambda}^{\ast})^{2^{\ast}_{\alpha}-1}
			\frac{\partial Z_{\overline{r},\overline{x}'',\lambda}^{\ast}}{\partial \lambda}dx\\
			=\sum_{j=1}^{m}\int_{\mathbb{R}^{N}}\Big(|x|^{-\alpha}\ast |U_{z_j,\lambda}|^{2^{\ast}_{\alpha}}\Big)U_{z_j,\lambda}^{2^{\ast}_{\alpha}-1}\frac{\partial Z_{\overline{r},\overline{x}'',\lambda}^{\ast}}{\partial \lambda}dx\\
			-\int_{\mathbb{R}^{N}}\Big(|x|^{-\alpha}\ast |Z_{\overline{r},\overline{x}'',\lambda}^{\ast}|^{2^{\ast}_{\alpha}}\Big)(Z_{\overline{r},\overline{x}'',\lambda}^{\ast})^{2^{\ast}_{\alpha}-1}
			\frac{\partial Z_{\overline{r},\overline{x}'',\lambda}^{\ast}}{\partial \lambda}dx\\
			+\frac{1}{2\cdot2^{\ast}_{\alpha}}\int_{\mathbb{R}^{N}}(1-K(r,x''))\frac{\partial }{\partial \lambda}\Big[\Big(|x|^{-\alpha}\ast \big((1+K(r,x''))|Z_{\overline{r},\overline{x}'',\lambda}^{\ast}|^{2^{\ast}_{\alpha}}\big)\Big)|Z_{\overline{r},\overline{x}'',\lambda}^{\ast}|^{2^{\ast}_{\alpha}}\Big]dx\\
			=:\cq_0+\cq_1.
		\end{multline}
		
		In order to estimate $\cq_0$, observe that
		\begin{equation}\label{5.1}
			\aligned
			&\int_{\mathbb{R}^{N}}\Bigg[\Big(|x|^{-\alpha}\ast |Z_{\overline{r},\overline{x}'',\lambda}^{\ast}|^{2^{\ast}_{\alpha}}\Big)(Z_{\overline{r},\overline{x}'',\lambda}^{\ast})^{2^{\ast}_{\alpha}-1}-\sum_{j=1}^{m}\Big(|x|^{-\alpha}\ast |U_{z_j,\lambda}|^{2^{\ast}_{\alpha}}\Big)U_{z_j,\lambda}^{2^{\ast}_{\alpha}-1}\Bigg] \frac{\partial Z_{\overline{r},\overline{x}'',\lambda}^{\ast}}{\partial \lambda}dx\\
			=&m\int_{\Omega_{1}}\Bigg[\Big(|x|^{-\alpha}\ast |Z_{\overline{r},\overline{x}'',\lambda}^{\ast}|^{2^{\ast}_{\alpha}}\Big)|Z_{\overline{r},\overline{x}'',\lambda}^{\ast}|^{2^{\ast}_{\alpha}-1}-\sum_{j=1}^{m}\Big(|x|^{-\alpha}\ast |U_{z_j,\lambda}|^{2^{\ast}_{\alpha}}\Big)U_{z_j,\lambda}^{2^{\ast}_{\alpha}-1}\Bigg]\frac{\partial Z_{\overline{r},\overline{x}'',\lambda}^{\ast}}{\partial \lambda}dx\\
			=& m\Bigg(\int_{\Omega_{1}}\Bigg[2^{\ast}_{\alpha}\Big(|x|^{-\alpha}\ast |U_{z_1,\lambda}^{2^{\ast}_{\alpha}-1}\sum_{i=2}^{m}U_{z_i,\lambda}|\Big)U_{z_1,\lambda}^{2^{\ast}_{\alpha}-1}+(2^{\ast}_{\alpha}-1)\Big(|x|^{-\alpha}\ast |U_{z_1,\lambda}|^{2^{\ast}_{\alpha}}\Big)U_{z_1,\lambda}^{2^{\ast}_{\alpha}-2}\sum_{i=2}^{m}U_{z_i,\lambda}\Bigg]\frac{\partial U_{z_1,\lambda}}{\partial \lambda}dx\\
			&+O(\frac{1}{\lambda^{3+\varepsilon}})\Bigg)\ .
			\endaligned\end{equation}
		Applying Lemmas \ref{P0}, we have 
		\begin{multline*}
		\int_{\Omega_{1}}\Big[\Big(|x|^{-\alpha}\ast |U_{z_1,\lambda}|^{2^{\ast}_{\alpha}}\Big)U_{z_1,\lambda}^{2^{\ast}_{\alpha}-2}U_{z_i,\lambda}\Big]\frac{\partial U_{z_1,\lambda}}{\partial \lambda}dx\\
		=-C\Big[\int_{\Omega_{1}}\frac{\lambda^{\frac{\alpha}{2}}}{(1+\lambda^{2}|x-z_{1} |^{2})^{\frac{\alpha}{2}}}\frac{\lambda^{\frac{4-\alpha}{2}}}{(1+\lambda^{2}|x-z_{1} |^{2})^{\frac{4-\alpha}{2}}}\frac{\lambda^{\frac{N-2}{2}}}{(1+\lambda^{2}|x-z_{i} |^{2})^{\frac{N-2}{2}}}\frac{\lambda^\frac{N-4}{2}}{(1+\lambda^{2}|x-z_{1} |^{2})^{\frac{N-2}{2}}}dx\\
		-\int_{\Omega_{1}}\frac{\lambda^{\frac{\alpha}{2}}}{(1+\lambda^{2}|x-z_{1} |^{2})^{\frac{\alpha}{2}}}\frac{\lambda^{\frac{4-\alpha}{2}}}{(1+\lambda^{2}|x-z_{1} |^{2})^{\frac{4-\alpha}{2}}}\frac{\lambda^{\frac{N-2}{2}}}{(1+\lambda^{2}|x-z_{i} |^{2})^{\frac{N-2}{2}}}\frac{2\lambda^\frac{N-4}{2}}{(1+\lambda^{2}|x-z_{1} |^{2})^{\frac{N}{2}}}dx\Big]\\
		=-C\Big[\int_{\Omega_{1}}\frac{\lambda^{\frac{N}{2}}}{(1+\lambda^{2}|x-z_{1} |^{2})^{\frac{N+2}{2}}}\frac{\lambda^{\frac{N-2}{2}}}{(1+\lambda^{2}|x-z_{i} |^{2})^{\frac{N-2}{2}}}dx\\
		-\int_{\Omega_{1}}\frac{2\lambda^{\frac{N}{2}}}{(1+\lambda^{2}|x-z_{1} |^{2})^{\frac{N+4}{2}}}\frac{\lambda^{\frac{N-2}{2}}}{(1+\lambda^{2}|x-z_{i} |^{2})^{\frac{N-2}{2}}}dx\Big]\\
		=-\frac{C}{\lambda^{N-1}|z_{1}-z_{i}|^{N-2}}, \quad i=2,\dots,m \ .
		\end{multline*}
				
		\noindent Consequently,
		\begin{equation}\label{5.2}
			\aligned
			\sum_{i=2}^{m}\int_{\Omega_{1}}\Big[\Big(|x|^{-\alpha}\ast |U_{z_1,\lambda}|^{2^{\ast}_{\alpha}}\Big)U_{z_1,\lambda}^{2^{\ast}_{\alpha}-2}U_{z_i,\lambda}\Big]\frac{\partial U_{z_1,\lambda}}{\partial \lambda}dx
			=-\sum_{i=2}^{m}\frac{C}{\lambda^{N-1}|z_{1}-z_{i}|^{N-2}}\ .
			\endaligned\end{equation}
		Notice that 
		\begin{multline*}
		\int_{\Omega_{1}}\Big(|x|^{-\alpha}\ast |U_{z_1,\lambda}^{2^{\ast}_{\alpha}-1}U_{z_i,\lambda}|\Big)U_{z_1,\lambda}^{2^{\ast}_{\alpha}-1}\frac{\partial U_{z_1,\lambda}}{\partial \lambda}dx\\
		=\frac{1}{2^{\ast}_{\alpha}}\Big[\frac{\partial }{\partial \lambda}\int_{\Omega_{1}}\Big(|y|^{-\alpha}\ast |U_{z_1,\lambda}|^{2^{\ast}_{\alpha}}\Big) U_{z_1,\lambda}^{2^{\ast}_{\alpha}-1}(y)U_{z_i,\lambda}(y)dy\\
		-\int_{\Omega_{1}}\Big(|y|^{-\alpha}\ast |U_{z_1,\lambda}|^{2^{\ast}_{\alpha}}\Big)\frac{\partial U_{z_1,\lambda}^{2^{\ast}_{\alpha}-1}(y)}{\partial \lambda}U_{z_i,\lambda}(y)dy\\
		-\int_{\Omega_{1}}\Big(|y|^{-\alpha}\ast |U_{z_1,\lambda}|^{2^{\ast}_{\alpha}}\Big)U_{z_1,\lambda}^{2^{\ast}_{\alpha}-1}(y)\frac{\partial U_{z_i,\lambda}}{\partial \lambda}(y)dy\Big].
		\end{multline*}
		Then we have
		$$\aligned
		&\int_{\Omega_{1}}\Big(|x|^{-\alpha}\ast |U_{z_1,\lambda}^{2^{\ast}_{\alpha}-1}U_{z_i,\lambda}|\Big)U_{z_1,\lambda}^{2^{\ast}_{\alpha}-1}\frac{\partial U_{z_1,\lambda}}{\partial \lambda}dx\\
		=&C\frac{\partial }{\partial \lambda}\int_{\Omega_{1}}\frac{\lambda^{\frac{\alpha}{2}}}{(1+\lambda^{2}|y- z_{1}|^{2})^{\frac{\alpha}{2}}}\frac{\lambda^\frac{N+2-\alpha}{2}}{(1+\lambda^{2}|y- z_{1}|^{2})^{\frac{N+2-\alpha}{2}}}\frac{\lambda^{\frac{N-2}{2}}}{(1+\lambda^{2}|y- z_{i}|^{2})^{\frac{N-2}{2}}}dy\\
		&+C\bigg[\int_{\Omega_{1}}\frac{\lambda^{\frac{\alpha}{2}}}{(1+\lambda^{2}|y- z_{1}|^{2})^{\frac{\alpha}{2}}}\frac{\lambda^\frac{N-\alpha}{2}}{(1+\lambda^{2}|y- z_{1}|^{2})^{\frac{N+2-\alpha}{2}}}\frac{\lambda^{\frac{N-2}{2}}}{(1+\lambda^{2}|y- z_{i}|^{2})^{\frac{N-2}{2}}}dy\\
		&\hspace{8mm}-\int_{\Omega_{1}}\frac{\lambda^{\frac{\alpha}{2}}}{(1+\lambda^{2}|y- z_{1}|^{2})^{\frac{\alpha}{2}}}\frac{2\lambda^{\frac{N-\alpha}{2}}}{(1+\lambda^{2}|y- z_{1}|^{2})^{\frac{N+4-\alpha}{2}}}\frac{\lambda^{\frac{N-2}{2}}}{(1+\lambda^{2}|y- z_{i}|^{2})^{\frac{N-2}{2}}}dy\bigg]\\
		&+C\bigg[\int_{\Omega_{1}}\frac{\lambda^{\frac{\alpha}{2}}}{(1+\lambda^{2}|y- z_{1}|^{2})^{\frac{\alpha}{2}}}\frac{\lambda^{\frac{N+2-\alpha}{2}}}{(1+\lambda^{2}|y- z_{1}|^{2})^{\frac{N+2-\alpha}{2}}}\frac{\lambda^\frac{N-4}{2}}{(1+\lambda^{2}|y- z_{i}|^{2})^{\frac{N-2}{2}}}dy\\
		&\hspace{8mm}-\int_{\Omega_{1}}\frac{\lambda^{\frac{\alpha}{2}}}{(1+\lambda^{2}|y- z_{1}|^{2})^{\frac{\alpha}{2}}}\frac{\lambda^{\frac{N+2-\alpha}{2}}}{(1+\lambda^{2}|y- z_{1}|^{2})^{\frac{N+2-\alpha}{2}}}\frac{2\lambda^\frac{N-4}{2}}{(1+\lambda^{2}|y- z_{i}|^{2})^{\frac{N}{2}}}dy\bigg]\\
		=&\frac{\partial }{\partial \lambda}\frac{C}{\lambda^{N-2}|z_{1}-z_{i}|^{N-2}}+\frac{2C}{\lambda^{N-1}|z_{1}-z_{i}|^{N-2}}
		=-\frac{C}{\lambda^{N-1}|z_{1}-z_{i}|^{N-2}}.
		\endaligned$$
		So, we obtain
		\begin{equation}\label{5.6}
			\sum_{i=2}^{m}\int_{\Omega_{1}}\Big(|x|^{-\alpha}\ast |U_{z_1,\lambda}^{2^{\ast}_{\alpha}-1}U_{z_i,\lambda}|\Big)U_{z_1,\lambda}^{2^{\ast}_{\alpha}-1}\frac{\partial U_{z_1,\lambda}}{\partial \lambda}dx=- \sum_{i=2}^{m}\frac{C}{\lambda^{N-1}|z_{1}-z_{i}|^{N-2}}.
		\end{equation}
		Combining \eqref{5.1}-\eqref{5.6}, we can conclude that
		$$\aligned
		\int_{\mathbb{R}^{N}}&\Big(|x|^{-\alpha}\ast |Z_{\overline{r},\overline{x}'',\lambda}^{\ast}|^{2^{\ast}_{\alpha}}\Big)|Z_{\overline{r},\overline{x}'',\lambda}^{\ast}|^{2^{\ast}_{\alpha}-1}
		\frac{\partial Z_{\overline{r},\overline{x}'',\lambda}^{\ast}}{\partial \lambda}dx
		-\sum_{j=1}^{m}\int_{\mathbb{R}^{N}}\Big(|x|^{-\alpha}\ast |U_{z_j,\lambda}|^{2^{\ast}_{\alpha}}\Big)U_{z_j,\lambda}^{2^{\ast}_{\alpha}-1}\frac{\partial Z_{\overline{r},\overline{x}'',\lambda}^{\ast}}{\partial \lambda}dx\\
		&=m\Big(-\sum_{j=2}^{m}
		\frac{A_{2}}{\lambda^{N-1}|z_{1}-z_{j}|^{N-2}}+O(\frac{1}{\lambda^{3+\varepsilon}})\Big),
		\endaligned$$
		which implies
		\begin{equation}\label{5.11}
			\aligned
		    \cq_0=m\Big(\sum_{j=2}^{m}
			\frac{A_{2}}{\lambda^{N-1}|z_{1}-z_{j}|^{N-2}}+O(\frac{1}{\lambda^{3+\varepsilon}})\Big).
			\endaligned
		\end{equation}

		Since the function K is bounded and satisfies the assumption $\textbf{(K2)}$, we estimate $\cq_1$ as follows
			\begin{multline*}
			\frac{1}{2\cdot2^{\ast}_{\alpha}}\int_{\mathbb{R}^{N}}(1-K(r,x''))\frac{\partial }{\partial \lambda}\Big[\Big(|x|^{-\alpha}\ast \big((1+K(r,x''))|Z_{\overline{r},\overline{x}'',\lambda}^{\ast}|^{2^{\ast}_{\alpha}}\big)\Big)|Z_{\overline{r},\overline{x}'',\lambda}^{\ast}|^{2^{\ast}_{\alpha}}\Big]dx\\
			=m\Bigg(\frac{1}{2\cdot2^{\ast}_{\alpha}}\int_{\Omega_{1}}(1-K(r,x''))\frac{\partial }{\partial \lambda}\Big[\Big(|x|^{-\alpha}\ast \big((1+K(r,x''))|U_{z_1,\lambda}|^{2^{\ast}_{\alpha}}\big)\Big)|U_{z_1,\lambda}|^{2^{\ast}_{\alpha}}\Big]dx+O(\frac{1}{\lambda^{3+\varepsilon}})\Bigg)\\
			=m\Big(\frac{C}{2^{\ast}_{\alpha}}\int_{\Omega_{1}}(1-K(r,x''))\frac{\partial }{\partial \lambda}\Big[\Big(|x|^{-\alpha}\ast |U_{z_1,\lambda}|^{2^{\ast}_{\alpha}}\Big)|U_{z_1,\lambda}|^{2^{\ast}_{\alpha}}\Big]dx+O(\frac{1}{\lambda^{3+\varepsilon}})\Big)\\
			=m\Big(-C\int_{B_{\lambda^{-(\frac{1}{2}+\varepsilon)}} (x_{0})}\Big(\sum_{i,j=1}^{m}\frac{1}{2}
			\frac{\partial ^{2} K(x_{0})}{\partial x_i\partial x_j}(x_{i}-x_{0i})(x_{j}-x_{0j})\\
			+\sum_{i,j,k=1}^{m}\frac{1}{6}
			\frac{\partial ^{3} K(x_{0}+\tau (x-x_0))}{\partial x_{i}\partial x_{j}\partial x_{k}}(x_{i}-x_{0i})(x_{j}-x_{0j})(x_{k}-x_{0k})\Big) \frac{1}{2^{\ast}_{\alpha}}\frac{\partial \Big(|x|^{-\alpha}\ast |U_{z_1,\lambda}|^{2^{\ast}_{\alpha}}\Big)|U_{z_1,\lambda}|^{2^{\ast}_{\alpha}}}{\partial \lambda}\\
			-C\int_{B^C_{\lambda^{-(\frac{1}{2}+\varepsilon)}} (x_{0})}
			(K(x)-1)\frac{1}{2^{\ast}_{\alpha}}\frac{\partial \Big(|x|^{-\alpha}\ast |U_{z_1,\lambda}|^{2^{\ast}_{\alpha}}\Big)|U_{z_1,\lambda}|^{2^{\ast}_{\alpha}}}{\partial \lambda}+O(\frac{1}{\lambda^{3+\varepsilon}})\Big)\\
			=m\Big(-C\int_{B_{\frac{3}{2}\lambda^{-(\frac{1}{2}+\varepsilon)}} (z_{1})}\Big(\sum_{i,j=1}^{m}\frac{1}{2}
			\frac{\partial ^{2} K(x_{0})}{\partial x_i\partial x_j}(x_{i}-x_{0i})(x_{j}-x_{0j})+O( ||x-x_{0}||^{3})\Big)\frac{1}{2^{\ast}_{\alpha}}\frac{\partial \Big(|x|^{-\alpha}\ast |U_{z_1,\lambda}|^{2^{\ast}_{\alpha}}\Big)|U_{z_1,\lambda}|^{2^{\ast}_{\alpha}}}{\partial \lambda}\\
			+O\Big(\int_{B^C_{\frac{1}{2}\lambda^{-(\frac{1}{2}+\varepsilon)}} (z_{1})}
			\frac{1}{\lambda}\Big(|x|^{-\alpha}\ast |U_{z_1,\lambda}|^{2^{\ast}_{\alpha}}\Big)|U_{z_1,\lambda}|^{2^{\ast}_{\alpha}}\Big)+O(\frac{1}{\lambda^{3+\varepsilon}})\Big)\\
			=m\Big(-C\int_{B_{\frac{3}{2}\lambda^{-(\frac{1}{2}+\varepsilon)}} (z_{1})}\Big(\sum_{i,j=1}^{m}\frac{1}{2}
			\frac{\partial ^{2} K(x_{0})}{\partial x_i\partial x_j}(x_{i}-x_{0i})(x_{j}-x_{0j})+O( ||x-x_{0}||^{3})\Big)\frac{1}{2^{\ast}_{\alpha}}\frac{\partial \Big(|x|^{-\alpha}\ast |U_{z_1,\lambda}|^{2^{\ast}_{\alpha}}\Big)|U_{z_1,\lambda}|^{2^{\ast}_{\alpha}}}{\partial \lambda}\\
			+O(\frac{1}{\lambda^{3+\varepsilon}})\Big)\\
			=m\Big(-\frac{C}{2^{\ast}_{\alpha}}\int_{B_{\frac{3}{2}\lambda^{-(\frac{1}{2}+\varepsilon)}} (z_{1})}\Big(\sum_{i,j=1}^{m}\frac{1}{2}
			\frac{\partial ^{2} K(x_{0})}{\partial x_i\partial x_j}(x_{i}-x_{0i})(x_{j}-x_{0j})\Big)\frac{\partial \Big((|x|^{-\alpha}\ast |U_{z_1,\lambda}|^{2^{\ast}_{\alpha}})U_{z_1,\lambda}^{2^{\ast}_{\alpha}}\Big)}{\partial\lambda}\\
			-\frac{C}{2^{\ast}_{\alpha}}\int_{B_{\frac{3}{2}\lambda^{-(\frac{1}{2}+\varepsilon)}} (z_{1})}O(||x-x_{0}||^{3})\frac{\partial \Big(|x|^{-\alpha}\ast |U_{z_1,\lambda}|^{2^{\ast}_{\alpha}}\Big)|U_{z_1,\lambda}|^{2^{\ast}_{\alpha}}}{\partial \lambda}+O(\frac{1}{\lambda^{3+\varepsilon}})\Big)\\
			=m\Big(-\frac{C}{2^{\ast}_{\alpha}}\frac{\partial}{\partial\lambda}\int_{B_{\frac{3}{2}\lambda^{-(\frac{1}{2}+\varepsilon)}} (0)}\Big(\sum_{i,j=1}^{m}\frac{1}{2}
			\frac{\partial ^{2} K(x_{0})}{\partial x_i\partial x_j}(y_{i}+z_{1i}-x_{0i})(y_{j}+z_{1j}-x_{0j})\Big) \Big(|y|^{-\alpha}\ast |U_{0,\lambda}|^{2^{\ast}_{\alpha}}\Big)U_{0,\lambda}^{2^{\ast}_{\alpha}}\\
			+\frac{C}{2^{\ast}_{\alpha}}\int_{B_{\frac{3}{2}\lambda^{-(\frac{1}{2}+\varepsilon)}} (0)} O(||y+z_{1}-x_{0}||^{3})\frac{1}{\lambda}\Big(|y|^{-\alpha}\ast |U_{0,\lambda}|^{2^{\ast}_{\alpha}}\Big)U_{0,\lambda}^{2^{\ast}_{\alpha}}+O(\frac{1}{\lambda^{3+\varepsilon}})\Big)\\
			=m\Big(-\frac{C}{2^{\ast}_{\alpha}}\frac{\partial}{\partial\lambda}\int\Big(\sum_{i,j=1}^{m}\frac{1}{2}
			\frac{\partial ^{2} K(x_{0})}{\partial x_i\partial x_j}(\frac{y_{i}}{\lambda}+z_{1i}-x_{0i})(\frac{y_{j}}{\lambda}+z_{1j}-x_{0j})\Big) \Big(|y|^{-\alpha}\ast |U_{0,1}|^{2^{\ast}_{\alpha}}\Big)U_{0,1}^{2^{\ast}_{\alpha}}\\
			-\frac{C}{2^{\ast}_{\alpha}}\frac{\partial}{\partial\lambda}\int_{B^C_{\frac{3}{2}\lambda^{-(\frac{1}{2}+\varepsilon)}} (0)}\Big(\sum_{i,j=1}^{m}\frac{1}{2}
			\frac{\partial ^{2} K(x_{0})}{\partial x_i\partial x_j}(y_{i}+z_{1i}-x_{0i})(y_{j}+z_{1j}-x_{0j})\Big) \Big(|y|^{-\alpha}\ast |U_{0,\lambda}|^{2^{\ast}_{\alpha}}\Big)U_{0,\lambda}^{2^{\ast}_{\alpha}}\\
			+\frac{C}{2^{\ast}_{\alpha}}\int O(||\frac{y}{\lambda}+z_{1}-x_{0}||^{3})\frac{1}{\lambda}\Big(|y|^{-\alpha}\ast |U_{0,1}|^{2^{\ast}_{\alpha}}\Big)U_{0,1}^{2^{\ast}_{\alpha}}+O(\frac{1}{\lambda^{3+\varepsilon}})\Big)
			\end{multline*}
		\begin{equation}\label{cq12}
			\aligned
			=&m\Big(-\frac{C}{2^{\ast}_{\alpha}}\frac{\partial}{\partial\lambda}\int\Big(\sum_{i,j=1}^{m}\frac{1}{2}
			\frac{\partial ^{2} K(x_{0})}{\partial x_i ^{2}}(\frac{y_{i}}{\lambda}+z_{1i}-x_{0i})^{2}\Big) \Big(|y|^{-\alpha}\ast |U_{0,1}|^{2^{\ast}_{\alpha}}\Big)U_{0,1}^{2^{\ast}_{\alpha}}+O(\frac{1}{\lambda^{3+\varepsilon}})\Big)\\
			=&m\Big(-\frac{C}{2^{\ast}_{\alpha}} \frac{\partial}{\partial\lambda}\int\frac{1}{2\lambda^{2}}\sum_{i,j=1}^{m}\frac{\partial ^{2} K(x_{0})}{\partial x_i ^{2}}
			{y_{i}}^{2} \Big(|y|^{-\alpha}\ast |U_{0,1}|^{2^{\ast}_{\alpha}}\Big)U_{0,1}^{2^{\ast}_{\alpha}}
			+O(\frac{1}{\lambda^{3+\varepsilon}})\Big)\\
			=&m\Big(\frac{C}{2^{\ast}_{\alpha}}\frac{1}{\lambda^{3}}\frac{\Delta K(x_{0})}{N}\int
			{y}^{2} \Big(|y|^{-\alpha}\ast |U_{0,1}|^{2^{\ast}_{\alpha}}\Big)U_{0,1}^{2^{\ast}_{\alpha}}
			+O(\frac{1}{\lambda^{3+\varepsilon}})\Big)\\
			=&m\Big(-\frac{A_{1}}{\lambda^{3}} +O(\frac{1}{\lambda^{3+\varepsilon}})\Big).
			\endaligned
		\end{equation}
			
	    So we conclude that
		\begin{equation}\label{5.12}
			\cq_1
			=m\Big(-\frac{A_{1}}{\lambda^{3}} +O(\frac{1}{\lambda^{3+\varepsilon}})\Big).
		\end{equation}
	Combine \eqref{5.11} and \eqref{5.12} to obtain
	$$
	\frac{\partial J(Z_{\overline{r},\overline{x}'',\lambda})}{\partial \lambda}=m\Big(-\frac{A_{1}}{\lambda^{3}}+\sum_{j=2}^{m}
	\frac{{A}_{2}}{\lambda^{N-1}|z_{1}-z_{j}|^{N-2}}+O(\frac{1}{\lambda^{3+\varepsilon}})\Big),
	$$
	where $A_j$, $j=1, 2$ are some positive constants.
\end{proof}



\section{Proof of the main result}
\noindent Next we look for a suitable triplet $(\overline{r},\overline{x}'',\lambda)$ such that the function $Z_{\overline{r},\overline{x}'',\lambda}+
\phi_{\overline{r},\overline{x}'',\lambda}$ turns out to be a solution to \eqref{CFL}. For this purpose, we will establish local Poho\v{z}aev identities to localize bubbles.

\begin{lem}\label{D1}
Suppose that $(\overline{r},\overline{x}'',\lambda)$ satisfies
\begin{equation}\label{d1}
	\int_{D_{\rho}}\Big(-\Delta u_{m}
	-K(r,x'')\Big(|x|^{-\alpha}\ast (K(r,x'')|u_{m}|^{2^{\ast}_{\alpha}})\Big)u_{m}^{2^{\ast}_{\alpha}-1}\Big)\langle x,\nabla u_{m}\rangle dx=0,
\end{equation}
\begin{equation}\label{d2}
	\int_{D_{\rho}}\Big(-\Delta u_{m}
	-K(r,x'')\Big(|x|^{-\alpha}\ast (K(r,x'')|u_{m}|^{2^{\ast}_{\alpha}})\Big)u_{m}^{2^{\ast}_{\alpha}-1}\Big)\frac{\partial u_{m}}{\partial x_{i}} dx=0, i=3,\cdots,N,
\end{equation}
and
\begin{equation}\label{d3}
	\int_{\mathbb{R}^{N}}\Big(-\Delta u_{m}-K(r,x'')\Big(|x|^{-\alpha}\ast (K(r,x'')|u_{m}|^{2^{\ast}_{\alpha}})\Big)u_{m}^{2^{\ast}_{\alpha}-1}\Big)\frac{\partial Z_{\overline{r},\overline{x}'',\lambda}}{\partial \lambda} dx=0,
\end{equation}
where $u_{m}=Z_{\overline{r},\overline{x}'',\lambda}+
\phi_{\overline{r},\overline{x}'',\lambda}$ and $D_{\rho}=\{(r, x''):|(r, x'')-(r_0, x_0'' )|\leq\rho\}$ with $\rho\in(2\delta, 5\delta)$. Then $c_{i}=0, i=1,\cdot\cdot\cdot,N$.
\end{lem}
\begin{proof}
Since $Z_{\overline{r},\overline{x}'',\lambda}=0$ in $\mathbb{R}^{N}\backslash D_{\rho}$, we see that if \eqref{d1}-\eqref{d3} hold, then
\begin{equation}\label{d4}
	\sum_{l=1}^{N}c_{l}\sum_{j=1}^{m}\int_{\mathbb{R}^{N}}\Big[(2^{\ast}_{\alpha}-1)\Big(|x|^{-\alpha}\ast |Z_{z_j,\lambda}|^{2^{\ast}_{\alpha}}\Big)Z_{z_j,\lambda}^{2^{\ast}_{\alpha}-2}Z_{j,l} +2^{\ast}_{\alpha}\Big(|x|^{-\alpha}\ast (Z_{z_j,\lambda}^{2^{\ast}_{\alpha}-1}
	Z_{j,l})\Big)Z_{z_j,\lambda}^{2^{\ast}_{\alpha}-1}\Big]vdx=0,
\end{equation}
for $v=\langle x,\nabla u_{m}\rangle$, $\frac{\partial u_{m}}{\partial x_{i}}, i=3,\cdot\cdot\cdot,N$ and $\frac{\partial Z_{\overline{r},\overline{x}'',\lambda}}{\partial \lambda}$.

Direct calculations show that
$$\aligned
&\left|\int_{\mathbb{R}^{N}}\Big(|x|^{-\alpha}\ast |Z_{z_j,\lambda}|^{2^{\ast}_{\alpha}}\Big)Z_{z_j,\lambda}^{2^{\ast}_{\alpha}-2}Z_{j,i}
\frac{\partial Z_{z_l,\lambda}}{\partial x_{i}}dx\right|\\
&\hspace{4mm}\leq C\lambda^{2}\int_{\mathbb{R}^{N}}\frac{1}{(1+|x-\lambda z_{j} |^{2})^{2}}\frac{(x_{i}-\lambda\overline{x}_{i})^{2}}{(1+|x-\lambda z_{j} |^{2})^{\frac{N}{2}}}
\frac{1}{(1+|x-\lambda z_{l} |^{2})^{\frac{N}{2}}}dx,
\endaligned$$
where $ i=3,\cdot\cdot\cdot,N$ and $(\overline{x}_{3}, \overline{x}_{4}, \cdot\cdot\cdot,\overline{x}_{N})=\overline{x}$.
If $l=j$, we have
$$
\left|\int_{\mathbb{R}^{N}}\Big(|x|^{-\alpha}\ast |Z_{z_j,\lambda}|^{2^{\ast}_{\alpha}}\Big)Z_{z_j,\lambda}^{2^{\ast}_{\alpha}-2}Z_{j,i}
\frac{\partial Z_{z_l,\lambda}}{\partial x_{i}}dx\right|=O (\lambda^{2}).
$$
If $l\neq j$, similar to the arguments in the proof of Lemma \ref{C5}, we can prove that
$$
\left|\int_{\mathbb{R}^{N}}\Big(|x|^{-\alpha}\ast |Z_{z_j,\lambda}|^{2^{\ast}_{\alpha}}\Big)Z_{z_j,\lambda}^{2^{\ast}_{\alpha}-2}Z_{j,i}
\left(\sum_{l=1,\neq j}^{m}\frac{\partial Z_{z_l,\lambda}}{\partial x_{i}}\right)dx\right|=O (\lambda^{2-\varepsilon}),
$$
for some $\varepsilon> 0$. So we get
\begin{equation}\label{d6}
	\aligned
	&\sum_{j=1}^{m}\int_{\mathbb{R}^{N}}\Big[(2^{\ast}_{\alpha}-1)\Big(|x|^{-\alpha}\ast |Z_{z_j,\lambda}|^{2^{\ast}_{\alpha}}\Big)Z_{z_j,\lambda}^{2^{\ast}_{\alpha}-2}Z_{j,i}
	+2^{\ast}_{\alpha}\Big(|x|^{-\alpha}\ast (Z_{z_j,\lambda}^{2^{\ast}_{\alpha}-1}
	Z_{j,i})\Big)Z_{z_j,\lambda}^{2^{\ast}_{\alpha}-1}\Big]\frac{\partial Z_{\overline{r},\overline{x}'',\lambda}}{\partial x_{i}}dx\\
	&\hspace{8mm}=m(a_{1}+o(1))\lambda^{2}, \ i=3,\cdot\cdot\cdot, N,
	\endaligned
\end{equation}
for some constants $a_{1}\neq0$. Similarly, we have
\begin{equation}\label{d5}
	\aligned
	&\sum_{j=1}^{m}\int_{\mathbb{R}^{N}}\Big[(2^{\ast}_{\alpha}-1)\Big(|x|^{-\alpha}\ast |Z_{z_j,\lambda}|^{2^{\ast}_{\alpha}}\Big)Z_{z_j,\lambda}^{2^{\ast}_{\alpha}-2}Z_{j,2}+2^{\ast}_{\alpha}\Big(|x|^{-\alpha}\ast (Z_{z_j,\lambda}^{2^{\ast}_{\alpha}-1}
	Z_{j,2})\Big)Z_{z_j,\lambda}^{2^{\ast}_{\alpha}-1} \Big]\langle x',\nabla_{x'} Z_{\overline{r},\overline{x}'',\lambda}\rangle dx\\
	&\hspace{8mm}=m(a_{2}+o(1))\lambda^{2},
	\endaligned
\end{equation}
and
\begin{equation}\label{d7}
	\aligned
	&\sum_{j=1}^{m}\int_{\mathbb{R}^{N}}\Big[(2^{\ast}_{\alpha}-1)\Big(|x|^{-\alpha}\ast |Z_{z_j,\lambda}|^{2^{\ast}_{\alpha}}\Big)Z_{z_j,\lambda}^{2^{\ast}_{\alpha}-2}Z_{j,1} +2^{\ast}_{\alpha}\Big(|x|^{-\alpha}\ast (Z_{z_j,\lambda}^{2^{\ast}_{\alpha}-1}
	Z_{j,1})\Big)Z_{z_j,\lambda}^{2^{\ast}_{\alpha}-1}\Big]\frac{\partial Z_{\overline{r},\overline{y}'',\lambda}}{\partial \lambda}dx\\
	&\hspace{8mm}=\frac{m}{\lambda^{2}}(a_{3}+o(1)),
	\endaligned
\end{equation}
for some constants $a_{2}\neq0$ and $a_3>0$.

\noindent By inspection 
$$\aligned
&\int_{\mathbb{R}^{N}}\Big(|x|^{-\alpha}\ast |Z_{z_j,\lambda}|^{2^{\ast}_{\alpha}}\Big)Z_{z_j,\lambda}^{2^{\ast}_{\alpha}-2}Z_{j,1}\frac{\partial \phi_{\overline{r},\overline{x}'',\lambda}}{\partial x_{i}}dx\\
=&\int_{\mathbb{R}^N}\int_{\mathbb{R}^N}
\frac{|Z_{z_j,\lambda}(y)|^{2^{\ast}_{\alpha}}\big(Z_{z_j,\lambda}^{2^{\ast}_{\alpha}-2}(x)
	\frac{\partial Z_{j,1}}{\partial x_{i}}(x)+\frac{\partial{Z_{z_j,\lambda}^{2^{\ast}_{\alpha}-2}}}{\partial x_{i}}(x)
	Z_{j,1}(x)\big)\phi_{\overline{r},\overline{x}'',\lambda}(x)}{|x-y|^{\alpha}}dxdy\\
&\hspace{8mm}+\int_{\mathbb{R}^N}\int_{\mathbb{R}^N}
\frac{|Z_{z_j,\lambda}(y)|^{2^{\ast}_{\alpha}}(x_{i}-y_{i})Z_{z_j,\lambda}^{2^{\ast}_{\alpha}-2}(x)
	Z_{j,1}(x)\phi_{\overline{r},\overline{x}'',\lambda}(y)}{|x-y|^{\alpha+2}}dxdy,
\endaligned$$
where $j=1,2,\cdot\cdot\cdot,m$ and $i=3,\cdot\cdot\cdot,N$. By \eqref{c15} we obtain
$$\aligned
&\left|\int_{\mathbb{R}^N}\int_{\mathbb{R}^N}
\frac{|Z_{z_j,\lambda}(y)|^{2^{\ast}_{\alpha}}Z_{z_j,\lambda}^{2^{\ast}_{\alpha}-2}(x)
	\frac{\partial Z_{j,1}}{\partial x_{i}}(x)\phi_{\overline{r},\overline{x}'',\lambda}(x)}{|x-y|^{\alpha}}dxdy\right|\\
\leq &C\|\phi\|_{\ast}\int_{\mathbb{R}^N}\frac{\lambda^{\frac{\alpha}{2}}}{(1+\lambda^{2}|x-z_{j} |^{2})^{\frac{4-\alpha}{2}}}\frac{\lambda^{\frac{4-\alpha}{2}}}{(1+\lambda^{2}|x-z_{j} |^{2})^{\frac{\alpha}{2}}}\sum_{k=1}^{m} \frac{\lambda^{\frac{N-2}{2}}}{(1+\lambda|x-z_{k}|)^{\frac{N-2}{2}+\tau}}\frac{\lambda^{\frac{N}{2}}|x-z_{j} |}{(1+\lambda^{2}|x-z_{j} |^{2})^{\frac{N}{2}}}dx\\
\leq &C\|\phi\|_{\ast}\int_{\mathbb{R}^N}\frac{|x-\lambda z_{j} |}{(1+|x-\lambda z_{j} |^{2})^{\frac{N+4}{2}}}\sum_{k=1}^{m}\frac{1}{(1+|x-\lambda z_{k} |)^{\frac{N-2}{2}+\tau}}dx\\
\leq &Cm\|\phi\|_{\ast}\int_{\Omega_{1}}\frac{1}{(1+|x-\lambda z_{1} |)^{N+3}}\frac{1}{(1+|x-\lambda z_{1} |)^{\frac{N-2}{2}}}dx\\
=&Cm\|\phi\|_{\ast}=O(\frac{1}{\lambda^{\varepsilon}}).
\endaligned$$
Similarly, we have
$$\aligned
&\left|\int_{\mathbb{R}^N}\int_{\mathbb{R}^N}
\frac{|Z_{z_j,\lambda}(y)|^{2^{\ast}_{\alpha}}\frac{\partial Z_{z_j,\lambda}^{2^{\ast}_{\alpha}-2}}{\partial x_{i}}(x)
	Z_{j,1}(x)\phi_{\overline{r},\overline{x}'',\lambda}(x)}{|x-y|^{\alpha}}dxdy\right|=O(\frac{1}{\lambda^{\varepsilon}}),
\endaligned$$
and
$$
\int_{\mathbb{R}^N}\int_{\mathbb{R}^N}
\frac{|Z_{z_j,\lambda}(x)|^{2^{\ast}_{\alpha}}
	(x_{i}-y_{i})Z_{z_j,\lambda}^{2^{\ast}_{\alpha}-2}(x)Z_{j,1}(y)\phi_{\overline{r},\overline{x}'',\lambda}(y)}{|x-y|^{\alpha+2}}dxdy=O(\frac{1}{\lambda^{\varepsilon}}).
$$
The above three equalities imply that
$$
\int_{\mathbb{R}^{N}}\Big(|x|^{-\alpha}\ast |Z_{z_j,\lambda}|^{2^{\ast}_{\alpha}}\Big)Z_{z_j,\lambda}^{2^{\ast}_{\alpha}-2}Z_{j,1}\frac{\partial \phi_{\overline{r},\overline{x}'',\lambda}}{\partial x_{i}}dx=O(\frac{1}{\lambda^{\varepsilon}}).
$$
Similarly, we have
$$
\int_{\mathbb{R}^{N}}\Big(|x|^{-\alpha}\ast (Z_{z_j,\lambda}^{2^{\ast}_{\alpha}-1}
Z_{j,1})\Big)Z_{z_j,\lambda}^{2^{\ast}_{\alpha}-1}\frac{\partial \phi_{\overline{r},\overline{x}'',\lambda}}{\partial x_{i}}dx=O(\frac{1}{\lambda^{\varepsilon}}).
$$
Therefore, we conclude that
\begin{multline*}
c_{1}\sum_{j=1}^{m}\int_{\mathbb{R}^{N}}\Big[(2^{\ast}_{\alpha}-1)\Big(|x|^{-\alpha}\ast |Z_{z_j,\lambda}|^{2^{\ast}_{\alpha}}\Big)Z_{z_j,\lambda}^{2^{\ast}_{\alpha}-2}Z_{j,1}+2^{\ast}_{\alpha}\Big(|x|^{-\alpha}\ast (Z_{z_j,\lambda}^{2^{\ast}_{\alpha}-1}
Z_{j,1})\Big)Z_{z_j,\lambda}^{2^{\ast}_{\alpha}-1}\Big]\frac{\partial \phi_{\overline{r},\overline{x}'',\lambda}}{\partial x_{i}}dx\\
=o(m|c_{1}|).
\end{multline*}
Likewise
\begin{multline*}
\sum_{l=2}^{N}c_{l}\sum_{j=1}^{m}\int_{\mathbb{R}^{N}}\Big[(2^{\ast}_{\alpha}-1)\Big(|x|^{-\alpha}\ast |Z_{z_j,\lambda}|^{2^{\ast}_{\alpha}}\Big)Z_{z_j,\lambda}^{2^{\ast}_{\alpha}-2}Z_{j,l}+2^{\ast}_{\alpha}\Big(|x|^{-\alpha}\ast (Z_{z_j,\lambda}^{2^{\ast}_{\alpha}-1}
Z_{j,l})\Big)Z_{z_j,\lambda}^{2^{\ast}_{\alpha}-1}\Big]\frac{\partial \phi_{\overline{r},\overline{x}'',\lambda}}{\partial x_{i}}dx\\
=o(m\lambda^{2})\sum_{l=2}^{N}|c_{l}|,
\end{multline*}
and hence 
$$\aligned
&\sum_{l=1}^{N}c_{l}\sum_{j=1}^{m}\int_{\mathbb{R}^{N}}\Big[(2^{\ast}_{\alpha}-1)\Big(|x|^{-\alpha}\ast |Z_{z_j,\lambda}|^{2^{\ast}_{\alpha}}\Big)Z_{z_j,\lambda}^{2^{\ast}_{\alpha}-2}Z_{j,l}+2^{\ast}_{\alpha}\Big(|x|^{-\alpha}\ast (Z_{z_j,\lambda}^{2^{\ast}_{\alpha}-1}
Z_{j,l})\Big)Z_{z_j,\lambda}^{2^{\ast}_{\alpha}-1}\Big]\frac{\partial \phi_{\overline{r},\overline{x}'',\lambda}}{\partial x_{i}}dx\\
&\hspace{8mm}=o(m\lambda^{2})\sum_{l=2}^{N}|c_{l}|+o(m|c_{1}|),
\endaligned$$
where $i=3,\cdot\cdot\cdot,N$. Through the same arguments, we also have
$$\aligned
&\sum_{l=1}^{N}c_{l}\sum_{j=1}^{m}\int_{\mathbb{R}^{N}}\Big[(2^{\ast}_{\alpha}-1)\Big(|x|^{-\alpha}\ast |Z_{z_j,\lambda}|^{2^{\ast}_{\alpha}}\Big)Z_{z_j,\lambda}^{2^{\ast}_{\alpha}-2}Z_{j,l}+2^{\ast}_{\alpha}\Big(|x|^{-\alpha}\ast (Z_{z_j,\lambda}^{2^{\ast}_{\alpha}-1}
Z_{j,l})\Big)Z_{z_j,\lambda}^{2^{\ast}_{\alpha}-1}\Big]\langle x,\nabla \phi_{\overline{r},\overline{x}'',\lambda}\rangle dx\\
&\hspace{8mm}=o(m\lambda^{2})\sum_{l=2}^{N}|c_{l}|+o(m|c_{1}|).
\endaligned$$
Together with \eqref{d4}, we can now deduce that
\begin{equation}\label{d8}
	\aligned
	&\sum_{l=1}^{N}c_{l}\sum_{j=1}^{m}\int_{\mathbb{R}^{N}}\Big[(2^{\ast}_{\alpha}-1)\Big(|x|^{-\alpha}\ast |Z_{z_j,\lambda}|^{2^{\ast}_{\alpha}}\Big)Z_{z_j,\lambda}^{2^{\ast}_{\alpha}-2}Z_{j,l}+2^{\ast}_{\alpha}\Big(|x|^{-\alpha}\ast (Z_{z_j,\lambda}^{2^{\ast}_{\alpha}-1}
	Z_{j,l})\Big)Z_{z_j,\lambda}^{2^{\ast}_{\alpha}-1}\Big]vdx\\
	&\hspace{8mm}=o(m\lambda^{2})\sum_{l=2}^{N}|c_{l}|+o(m|c_{1}|)
	\endaligned
\end{equation}
holds for $v=\langle x,\nabla Z_{\overline{r},\overline{x}'',\lambda}\rangle$, $\frac{\partial Z_{\overline{r},\overline{x}'',\lambda}}{\partial x_{i}}$, $ i=3,\cdot\cdot\cdot,N$.

\noindent From
$$
\langle x,\nabla Z_{\overline{r},\overline{x}'',\lambda}\rangle=\langle x',\nabla_{x'} Z_{\overline{r},\overline{x}'',\lambda}\rangle+\langle x'',\nabla_{x''} Z_{\overline{r},\overline{x}'',\lambda}\rangle,
$$
we have 
\begin{multline}\label{d9}
	\sum_{l=1}^{N}c_{l}\sum_{j=1}^{m}\int_{\mathbb{R}^{N}}\Big[(2^{\ast}_{\alpha}-1)\Big(|x|^{-\alpha}\ast |Z_{z_j,\lambda}|^{2^{\ast}_{\alpha}}\Big)Z_{z_j,\lambda}^{2^{\ast}_{\alpha}-2}Z_{j,l}+2^{\ast}_{\alpha}\Big(|x|^{-\alpha}\ast (Z_{z_j,\lambda}^{2^{\ast}_{\alpha}-1}
	Z_{j,l})\Big)Z_{z_j,\lambda}^{2^{\ast}_{\alpha}-1}\Big]\langle x,\nabla Z_{\overline{r},\overline{x}'',\lambda}\rangle dx\\
	=c_{2}\sum_{j=1}^{m}\int_{\mathbb{R}^{N}}\Big[(2^{\ast}_{\alpha}-1)\Big(|x|^{-\alpha}\ast |Z_{z_j,\lambda}|^{2^{\ast}_{\alpha}}\Big)Z_{z_j,\lambda}^{2^{\ast}_{\alpha}-2}Z_{j,l}+2^{\ast}_{\alpha}\Big(|x|^{-\alpha}\ast (Z_{z_j,\lambda}^{2^{\ast}_{\alpha}-1}
	Z_{j,l})\Big)Z_{z_j,\lambda}^{2^{\ast}_{\alpha}-1}\Big]\langle x',\nabla_{x'} Z_{\overline{r},\overline{x}'',\lambda}\rangle dx\\
+o(m\lambda^{2})\sum_{l=3}^{N}|c_{l}|+o(m|c_{1}|)
\end{multline}
and
\begin{equation}\label{d10}
	\aligned
	&\sum_{l=1}^{N}c_{l}\sum_{j=1}^{m}\int_{\mathbb{R}^{N}}\Big[(2^{\ast}_{\alpha}-1)\Big(|x|^{-\alpha}\ast |Z_{z_j,\lambda}|^{2^{\ast}_{\alpha}}\Big)Z_{z_j,\lambda}^{2^{\ast}_{\alpha}-2}Z_{j,l}+2^{\ast}_{\alpha}\Big(|x|^{-\alpha}\ast (Z_{z_j,\lambda}^{2^{\ast}_{\alpha}-1}
	Z_{j,l})\Big)Z_{z_j,\lambda}^{2^{\ast}_{\alpha}-1}\Big]\frac{\partial Z_{\overline{r},\overline{x}'',\lambda}}{\partial x_{i}}dx\\
	&\hspace{2mm}=c_{i}\sum_{j=1}^{m}\int_{\mathbb{R}^{N}}\Big[(2^{\ast}_{\alpha}-1)\Big(|x|^{-\alpha}\ast |Z_{z_j,\lambda}|^{2^{\ast}_{\alpha}}\Big)Z_{z_j,\lambda}^{2^{\ast}_{\alpha}-2}Z_{j,l}+2^{\ast}_{\alpha}\Big(|x|^{-\alpha}\ast (Z_{z_j,\lambda}^{2^{\ast}_{\alpha}-1}
	Z_{j,l})\Big)Z_{z_j,\lambda}^{2^{\ast}_{\alpha}-1}\Big]\frac{\partial Z_{\overline{r},\overline{x}'',\lambda}}{\partial x_{i}}dx\\
	&\hspace{4mm}+o(m\lambda^{2})\sum_{l\neq1,i}^{N}|c_{l}|+o(m|c_{1}|), \ i=3,\cdot\cdot\cdot, N.
	\endaligned
\end{equation}

Combining \eqref{d8}-\eqref{d10}, we obtain
$$\aligned
&c_{2}\sum_{j=1}^{m}\int_{\mathbb{R}^{N}}\Big[(2^{\ast}_{\alpha}-1)\Big(|x|^{-\alpha}\ast |Z_{z_j,\lambda}|^{2^{\ast}_{\alpha}}\Big)Z_{z_j,\lambda}^{2^{\ast}_{\alpha}-2}Z_{j,l}+2^{\ast}_{\alpha}\Big(|x|^{-\alpha}\ast (Z_{z_j,\lambda}^{2^{\ast}_{\alpha}-1}
Z_{j,l})\Big)Z_{z_j,\lambda}^{2^{\ast}_{\alpha}-1}\Big]\langle x',\nabla_{x'} Z_{\overline{r},\overline{x}'',\lambda}\rangle dx\\
&\hspace{8mm}=o(m\lambda^{2})\sum_{l=3}^{N}|c_{l}|+o(m|c_{1}|),
\endaligned$$
and
$$\aligned
&c_{i}\sum_{j=1}^{m}\int_{\mathbb{R}^{N}}\Big[(2^{\ast}_{\alpha}-1)\Big(|x|^{-\alpha}\ast |Z_{z_j,\lambda}|^{2^{\ast}_{\alpha}}\Big)Z_{z_j,\lambda}^{2^{\ast}_{\alpha}-2}Z_{j,l}+2^{\ast}_{\alpha}\Big(|x|^{-\alpha}\ast (Z_{z_j,\lambda}^{2^{\ast}_{\alpha}-1}
Z_{j,l})\Big)Z_{z_j,\lambda}^{2^{\ast}_{\alpha}-1}\Big]\frac{\partial Z_{\overline{r},\overline{x}'',\lambda}}{\partial x_{i}} dx\\
&\hspace{8mm}=o(m\lambda^{2})\sum_{l\neq1,i}^{N}|c_{l}|+o(m|c_{1}|), \ i=3,\cdot\cdot\cdot, N,
\endaligned$$
which, together with \eqref{d6} and \eqref{d5}, implies that
\begin{equation}\label{ci}
	c_{i}=o(\frac{1}{\lambda^{2}})c_{1}, \ i=2,\cdot\cdot\cdot, N.
\end{equation}

\noindent Finally we have
\begin{multline*}
0=\sum_{l=1}^{N}c_{l}\sum_{j=1}^{m}\int_{\mathbb{R}^{N}}\Big[(2^{\ast}_{\alpha}-1)\Big(|x|^{-\alpha}\ast |Z_{z_j,\lambda}|^{2^{\ast}_{\alpha}}\Big)Z_{z_j,\lambda}^{2^{\ast}_{\alpha}-2}Z_{j,l}+2^{\ast}_{\alpha}\Big(|x|^{-\alpha}\ast (Z_{z_j,\lambda}^{2^{\ast}_{\alpha}-1}
Z_{j,l})\Big)Z_{z_j,\lambda}^{2^{\ast}_{\alpha}-1}\Big]\frac{\partial Z_{\overline{r},\overline{x}'',\lambda}}{\partial \lambda} dx\\
=c_{1}\sum_{j=1}^{m}\int_{\mathbb{R}^{N}}\Big[(2^{\ast}_{\alpha}-1)\Big(|x|^{-\alpha}\ast |Z_{z_j,\lambda}|^{2^{\ast}_{\alpha}}\Big)Z_{z_j,\lambda}^{2^{\ast}_{\alpha}-2}Z_{j,l}+2^{\ast}_{\alpha}\Big(|x|^{-\alpha}\ast (Z_{z_j,\lambda}^{2^{\ast}_{\alpha}-1}
Z_{j,l})\Big)Z_{z_j,\lambda}^{2^{\ast}_{\alpha}-1}\Big]\frac{\partial Z_{\overline{r},\overline{x}'',\lambda}}{\partial \lambda} dx\\
+o(\frac{m}{\lambda^{2}})c_{1}
=\frac{m}{\lambda^{2}}(a_{3}+o(1))c_{1}+o(\frac{m}{\lambda^{2}})c_{1}.
\end{multline*}
hence  $c_1=0$.
\end{proof}

\begin{lem}\label{D3}
The following holds 
$$\aligned
&\int_{\mathbb{R}^{N}}(-\Delta u_{m}
-K(r,x'')\Big(|x|^{-\alpha}\ast K(r,x'')|u_{m}|^{2^{\ast}_{\alpha}}\Big)u_{m}^{2^{\ast}_{\alpha}-1})\frac{\partial Z_{\overline{r},\overline{x}'',\lambda}}{\partial \lambda}dx\\
=&m\Big(-\frac{A_{1}}{\lambda^{3}}+\sum_{j=2}^{m}
\frac{A_{2}}{\lambda^{N-1}|z_{1}-z_{j}|^{N-2}}+O(\frac{1}{\lambda^{3+\varepsilon}})\Big)\\
=&m\Big(
-\frac{A_1}{\lambda^{3}}
+\frac{m^{N-2}A_3}{\lambda^{N-1}}+O(\frac{1}{\lambda^{3+\varepsilon}})\Big).
\endaligned$$
\end{lem}
\begin{proof}
$$\aligned
&\int_{\mathbb{R}^{N}}(-\Delta u_{m}-K(r,x'')\Big(|x|^{-\alpha}\ast K(r,x'')|u_{m}|^{2^{\ast}_{\alpha}}\Big)u_{m}^{2^{\ast}_{\alpha}-1})\frac{\partial Z_{\overline{r},\overline{x}'',\lambda}}{\partial \lambda} dx\\
&=\langle J'(Z_{\overline{r},\overline{x}'',\lambda}),\frac{\partial Z_{\overline{r},\overline{x}'',\lambda}}{\partial \lambda}\rangle+m\Big\langle -\Delta \phi
-(2^{\ast}_{\alpha}-1)K(r,x'')\Big(|x|^{-\alpha}\ast K(r,x'')|Z_{\overline{r},\overline{x}'',\lambda}|^{2^{\ast}_{\alpha}}\Big)Z_{\overline{r},\overline{x}'',\lambda}^{2^{\ast}_{\alpha}-2}\phi\\
&-2^{\ast}_{\alpha}K(r,x'')\Big(|x|^{-\alpha}\ast K(r,x'')Z_{\overline{r},\overline{x}'',\lambda}^{2^{\ast}_{\alpha}-1}\phi\Big)Z_{\overline{r},\overline{x}'',\lambda}^{2^{\ast}_{\alpha}-1},\frac{\partial Z_{z_1,\lambda}}{\partial \lambda}\Big\rangle\\
&-\Bigg[\int_{\mathbb{R}^{N}}K(r,x'')\Big(|x|^{-\alpha}\ast K(r,x'')|u_{m}|^{2^{\ast}_{\alpha}}\Big)u_{m}^{2^{\ast}_{\alpha}-1}\frac{\partial Z_{\overline{r},\overline{x}'',\lambda}}{\partial \lambda} dx\\
&-(2^{\ast}_{\alpha}-1)\int_{\mathbb{R}^{N}}K(r,x'')\Big(|x|^{-\alpha}\ast K(r,x'')|Z_{\overline{r},\overline{x}'',\lambda}|^{2^{\ast}_{\alpha}}\Big)Z_{\overline{r},\overline{x}'',\lambda}^{2^{\ast}_{\alpha}-2}\phi\frac{\partial Z_{\overline{r},\overline{x}'',\lambda}}{\partial \lambda} dx\\
&-2^{\ast}_{\alpha}\int_{\mathbb{R}^{N}}K(r,x'')\Big(|x|^{-\alpha}\ast K(r,x'')Z_{\overline{r},\overline{x}'',\lambda}^{2^{\ast}_{\alpha}-1}\phi\Big)Z_{\overline{r},\overline{x}'',\lambda}^{2^{\ast}_{\alpha}-1}\frac{\partial Z_{\overline{r},\overline{x}'',\lambda}}{\partial \lambda} dx\\
&-\int_{\mathbb{R}^{N}}K(r,x'')\Big(|x|^{-\alpha}\ast K(r,x'')|Z_{\overline{r},\overline{x}'',\lambda}|^{2^{\ast}_{\alpha}}\Big)Z_{\overline{r},\overline{x}'',\lambda}^{2^{\ast}_{\alpha}-1}\frac{\partial Z_{\overline{r},\overline{x}'',\lambda}}{\partial \lambda} dx\Bigg]\\
&=:\langle J'(Z_{\overline{r},\overline{x}'',\lambda}),\frac{\partial Z_{\overline{r},\overline{x}'',\lambda}}{\partial \lambda}\rangle+mI_{1}-I_{2}.
\endaligned$$

Using \eqref{c61}-\eqref{c63}, we obtain
$$
I_{1}=O(\frac{\|\phi\|_{\ast}}{\lambda^{2+\varepsilon}})
=O(\frac{1}{\lambda^{3+\varepsilon}}).
$$


By \eqref{c15} we have
$$\aligned
I_{2}
\leq&C\int_{\mathbb{R}^{N}}|N(\phi)|\left|\frac{\partial Z_{\overline{r},\overline{x}'',\lambda}}{\partial \lambda}\right| dx\\
\leq&\frac{C\|\phi\|_{\ast}^{2}}{\lambda}\int_{\mathbb{R}^{N}}\sum_{j=1}^{m}\frac{\lambda^{\frac{N+2}{2}}}{(1+\lambda|x-z_{j}|)^{\frac{N+2}{2}+\tau}}\sum_{j=1}^{m}
\frac{\lambda^{\frac{N-2}{2}}}{(1+\lambda|x-z_{j}|)^{N-2}}dx\\ 
\leq&\frac{C\|\phi\|_{\ast}^{2}}{\lambda}\sum_{j=1}^{m}\int_{\mathbb{R}^{N}}\frac{1}{(1+\lambda|x-z_{j}|)^{\frac{3N-2}{2}}}dx=O(\frac{m}{\lambda^{3+\varepsilon}}).
\endaligned$$
So we can conclude that
$$
\Big\langle J'(Z_{\overline{r},\overline{x}'',\lambda}+\phi),\frac{\partial Z_{\overline{r},\overline{x}'',\lambda}}{\partial \lambda}\Big\rangle
=\Big\langle J'(Z_{\overline{r},\overline{x}'',\lambda}),\frac{\partial Z_{\overline{r},\overline{x}'',\lambda}}{\partial \lambda}\Big\rangle+O(\frac{m}{\lambda^{3+\varepsilon}}).
$$
and thus by Lemma \ref{D2}, we conclude the proof.
\end{proof}

\noindent Observe that \eqref{d1} is equivalent to
\begin{equation}\label{d12}
\begin{split}
	-\frac{N-2}{2}\int_{D_{\rho}}|\nabla u_{m}|^{2}dx&+\frac{1}{2^{\ast}_{\alpha}}\int_{D_{\rho}}\langle x,\nabla K(x)\rangle \Big( \int_{\R^N}\frac{K(y)|u_{m}(y)|^{2^{\ast}_{\alpha}}}{|x-y|^\alpha}dy \Big) u_{m}^{2^{\ast}_{\alpha}}(x)dx\\
	&+\frac{N}{2^{\ast}_{\alpha}}\int_{D_{\rho}}\int_{\R^N}\frac{K(x)K(y)|u_{m}(y)|^{2^{\ast}_{\alpha}}|u_{m}(x)|^{2^{\ast}_{\alpha}}}{|x-y|^\alpha}dxdy\\
	&-\frac{\alpha}{2^{\ast}_{\alpha}} \int_{D_{\rho}}\int_{\R^N}x(x-y)\frac{K(x)K(y)|u_{m}(y)|^{2^{\ast}_{\alpha}}|u_{m}(x)|^{2^{\ast}_{\alpha}}}{|x-y|^{\alpha+2}}dxdy
	\\=&O\Big(\int_{\partial D_{\rho}}|\nabla \phi|^{2}ds
	+\int_{\partial D_{\rho}}\Big(\int_{\R^N} \frac{|\phi(y)|^{2^{\ast}_{\alpha}}}{|x-y|^\alpha}dy\Big)|\phi|^{2^{\ast}_{\alpha}}ds\Big).
\end{split}
\end{equation}

\noindent Since
$$
\sum_{j=1}^{m}\int_{\mathbb{R}^{N}}[(2^{\ast}_{\alpha}-1)\Big(|x|^{-\alpha}\ast |Z_{z_j,\lambda}|^{2^{\ast}_{\alpha}}\Big)Z_{z_j,\lambda}^{2^{\ast}_{\alpha}-2}Z_{j,l}\phi +2^{\ast}_{\alpha}\Big(|x|^{-\alpha}\ast |Z_{z_j,\lambda}^{2^{\ast}_{\alpha}-1}Z_{j,l}|\Big)Z_{z_j,\lambda}^{2^{\ast}_{\alpha}-1}\phi] dx=0,
$$
we obtain from \eqref{c14} that
\begin{equation}\label{d13}
\aligned
&\int_{D_{\rho}}|\nabla u_{m}|^{2}dx=\int_{D_{\rho}}K(x)\Big(|x|^{-\alpha}\ast K(x)|u_{m}|^{2^{\ast}_{\alpha}}\Big)|u_{m}|^{2^{\ast}_{\alpha}} dx+O\Big(\int_{\partial D_{\rho}}|\nabla \phi|^{2} ds\Big)\\
&+\sum_{l=1}^{N}c_{l}\sum_{j=1}^{m}\int_{\mathbb{R}^{N}}\big[(2^{\ast}_{\alpha}-1)\Big(|x|^{-\alpha}\ast |Z_{z_j,\lambda}|^{2^{\ast}_{\alpha}}\Big)Z_{z_j,\lambda}^{2^{\ast}_{\alpha}-2}Z_{j,l}
+2^{\ast}_{\alpha}\Big(|x|^{-\alpha}\ast (Z_{z_j,\lambda}^{2^{\ast}_{\alpha}-1}
Z_{j,l})\Big)Z_{z_j,\lambda}^{2^{\ast}_{\alpha}-1}\big]Z_{\overline{r},\overline{x}'',\lambda}dx.
\endaligned\end{equation}
Inserting \eqref{d13} into \eqref{d12}, we obtain
\begin{equation}\label{d14}
\aligned
\frac{1}{2^{\ast}_{\alpha}}&\int_{D_{\rho}}\langle x,\nabla K(x)\rangle \Big( \int_{\R^N}\frac{K(y)|u_{m}(y)|^{2^{\ast}_{\alpha}}}{|x-y|^\alpha}dy \Big)u_{m}^{2^{\ast}_{\alpha}}\\
=&\frac{N-2}{2}\sum_{l=1}^{N}c_{l}\sum_{j=1}^{m}\int_{\mathbb{R}^{N}}\Big[(2^{\ast}_{\alpha}-1)\Big(|x|^{-\alpha}\ast |Z_{z_j,\lambda}|^{2^{\ast}_{\alpha}}\Big)Z_{z_j,\lambda}^{2^{\ast}_{\alpha}-2}Z_{j,l}+2^{\ast}_{\alpha}\Big(|x|^{-\alpha}\ast Z_{z_j,\lambda}^{2^{\ast}_{\alpha}-1}Z_{j,l}\Big)Z_{z_j,\lambda}^{2^{\ast}_{\alpha}-1}\Big]Z_{\overline{r},\overline{x}'',\lambda}dx\\
&+\frac{\alpha}{2^{\ast}_{\alpha}}\int_{D_{\rho}}\int_{\R^N}x(x-y)\frac{K(x)K(y)|u_{m}(y)|^{2^{\ast}_{\alpha}}|u_{m}(x)|^{2^{\ast}_{\alpha}}}{|x-y|^{\alpha+2}}dxdy\\
&+\big(\frac{N-2}{2}-\frac{N}{2^{\ast}_{\alpha}}\big)\int_{D_{\rho}}\int_{\R^N}\frac{K(x)K(y)|u_{m}(y)|^{2^{\ast}_{\alpha}}|u_{m}(x)|^{2^{\ast}_{\alpha}}}{|x-y|^\alpha}dxdy\\
&+O\Big(\int_{\partial D_{\rho}}|\nabla \phi|^{2}ds
+\int_{\partial D_{\rho}}\Big(\int_{\R^N}\frac{|\phi(y)|^{2}}{|x-y|^\alpha}dy\Big)|\phi|^{2}ds\Big)\\
=&\frac{N-2}{2}\sum_{l=1}^{N}c_{l}\sum_{j=1}^{m}\int_{\mathbb{R}^{N}}\Big[(2^{\ast}_{\alpha}-1)\Big(|x|^{-\alpha}\ast |Z_{z_j,\lambda}|^{2^{\ast}_{\alpha}}\Big)Z_{z_j,\lambda}^{2^{\ast}_{\alpha}-2}Z_{j,l}+2^{\ast}_{\alpha}\Big(|x|^{-\alpha}\ast Z_{z_j,\lambda}^{2^{\ast}_{\alpha}-1}Z_{j,l}\Big)Z_{z_j,\lambda}^{2^{\ast}_{\alpha}-1}\Big]Z_{\overline{r},\overline{x}'',\lambda}dx\\
&+O\Big(\int_{\partial D_{\rho}}|\nabla \phi|^{2}ds
+\int_{\partial D_{\rho}}\Big(\int_{\R^N}\frac{|\phi(y)|^{2^{\ast}_{\alpha}}}{|x-y|^\alpha}dy\Big)|\phi|^{2^{\ast}_{\alpha}}ds\Big)\\
&+O\Big(\int_{ D_{\rho}}\int_{\R^N\backslash D_{\rho}}x(x-y)\frac{|\phi(x)|^{2^{\ast}_{\alpha}}|\phi(y)|^{2^{\ast}_{\alpha}}}{|x-y|^{\alpha+2}}dxdy\Big)+O(\frac{1}{\lambda^{2+\varepsilon}}), \ i=3,\cdot\cdot\cdot,N.
\endaligned\end{equation}

\noindent We find that \eqref{d14} is equivalent to
\begin{equation}\label{d17}
\aligned
\frac{1}{2^{\ast}_{\alpha}}\int_{D_{\rho}}\langle x,&\nabla K(x)\rangle \Big( \int_{\R^N}\frac{K(y)|u_{m}(y)|^{2}}{|x-y|^\alpha}dy \Big)u_{m}^{2}dx\\
=&O(\frac{m}{\lambda^{2+\varepsilon}})+O\Big(\int_{\partial D_{\rho}}|\nabla \phi|^{2}ds
+\int_{\partial D_{\rho}}\Big(\int_{\R^N}\frac{|\phi(y)|^{2^{\ast}_{\alpha}}}{|x-y|^\alpha}dy\Big)|\phi|^{2^{\ast}_{\alpha}}ds\Big)\\
&
+O\Big(\int_{ D_{\rho}}\int_{\R^N\backslash D_{\rho}}x(x-y)\frac{|\phi(x)|^{2^{\ast}_{\alpha}}|\phi(y)|^{2^{\ast}_{\alpha}}}{|x-y|^{\alpha+2}}dxdy\Big)+O(\frac{1}{\lambda^{2+\varepsilon}}),
\endaligned\end{equation}
where $i=3,\cdots,N$ for some small $\varepsilon>0$.

\noindent Integrating by parts in \eqref{d2}, we find it is equivalent to
\begin{multline}\label{d11}
	\int_{D_{\rho}}
	\frac{\partial K(r,x'')}{\partial x_{i}}\Big( \int_{\R^N}\frac{K(y)|u_{m}(y)|^{2^{\ast}_{\alpha}}}{|x-y|^\alpha}dy \Big) u_{m}^{2^{\ast}_{\alpha}} dx\\
	=O\Big(\int_{\partial D_{\rho}}|\nabla \phi|^{2}ds
	+\int_{\partial D_{\rho}}\Big(\int_{\R^N}\frac{|\phi(y)|^{2^{\ast}_{\alpha}}}{|x-y|^\alpha}dy\Big)|\phi|^{2^{\ast}_{\alpha}}ds\Big)\\
	+O\Big(\int_{ D_{\rho}}\int_{\R^N\backslash D_{\rho}}x(x-y)\frac{|\phi(x)|^{2^{\ast}_{\alpha}}|\phi(y)|^{2^{\ast}_{\alpha}}}{|x-y|^{\alpha+2}}dxdy\Big),
\end{multline}
where $i=3,\cdots,N$. From \eqref{d11}, we can rewrite \eqref{d17} as follows 
\begin{equation}\label{d18}
\begin{split}
	&\int_{D_{\rho}}\frac{r}{2^{\ast}_{\alpha}}\frac{\partial K(r,x'')}{\partial r} \Big( \int_{\R^N}\frac{K(y)|u_{m}(y)|^{2^{\ast}_{\alpha}}}{|x-y|^\alpha}dy \Big)u_{m}^{2^{\ast}_{\alpha}}dx\\
	&=o(\frac{m}{\lambda^{2}})+O\Big(\int_{\partial D_{\rho}}|\nabla \phi|^{2}ds
	+\int_{\partial D_{\rho}}\Big(\int_{D_{\rho}} \frac{|\phi(y)|^{2^{\ast}_{\alpha}}}{|x-y|^\alpha}\Big)|\phi|^{2^{\ast}_{\alpha}}ds\Big)\\
	&+O\Big(\int_{ D_{\rho}}\int_{\R^N\backslash D_{\rho}}x(x-y)\frac{|\phi(x)|^{2^{\ast}_{\alpha}}|\phi(y)|^{2^{\ast}_{\alpha}}}{|x-y|^{\alpha+2}}dxdy\Big),
\end{split}
\end{equation}
where $i=3,\cdot\cdot\cdot,N$. In order to estimate \eqref{d11} and \eqref{d18}, let us first prove
\begin{lem}\label{D4}
The following holds
\begin{equation}
	\int|\nabla \phi|^{2}dx
	+\int\Big(\int_{\R^N}\frac{|\phi(y)|^{2^{\ast}_{\alpha}}}{|x-y|^\alpha}dy\Big)|\phi|^{2^{\ast}_{\alpha}}dx
	=O(\frac{m}{\lambda^{2+\varepsilon}}).
\end{equation}
\end{lem}
\begin{proof}
It follows from \eqref{c14} that
$$\aligned
\int|\nabla \phi|^{2}=&\int\Big[K(x)\Big(|x|^{-\alpha}\ast (K(x)|(Z_{\overline{r},\overline{x}'',\lambda}+\phi)|^{2^{\ast}_{\alpha}})\Big)(Z_{\overline{r},\overline{x}'',\lambda}+\phi)^{2^{\ast}_{\alpha}-1}
+\Delta Z_{\overline{r},\overline{x}'',\lambda}\Big]\phi\\
=&\int\Big[K(x)\Big(|x|^{-\alpha}\ast (K(x)|(Z_{\overline{r},\overline{x}'',\lambda}+\phi)|^{2^{\ast}_{\alpha}})\Big)(Z_{\overline{r},\overline{x}'',\lambda}+\phi)^{2^{\ast}_{\alpha}-1}\\
&\hspace{4mm}-K(x)\Big(|x|^{-\alpha}\ast (K(x)|Z_{\overline{r},\overline{x}'',\lambda}|^{2^{\ast}_{\alpha}})\Big)Z_{\overline{r},\overline{x}'',\lambda}^{2^{\ast}_{\alpha}-1}\Big]\phi\\
&+\int\Big[K(x)\Big(|x|^{-\alpha}\ast (K(x)|Z_{\overline{r},\overline{x}'',\lambda}|^{2^{\ast}_{\alpha}})\Big)Z_{\overline{r},\overline{x}'',\lambda}^{2^{\ast}_{\alpha}-1}
+\Delta Z_{\overline{r},\overline{x}'',\lambda}\Big]\phi\\
=:&\phi_{1}+\phi_{2}.
\endaligned$$
Let us estimate $\phi_{1}$ by Lemma \ref{C4} as follows 
$$\aligned
\left|\phi_{1}\right|=&\int \left|N(\phi)\right|\phi
+\int\Big[(2^{\ast}_{\alpha}-1)K(x)\Big(|x|^{-\alpha}\ast (K(x)|Z_{\overline{r},\overline{x}'',\lambda}|^{2^{\ast}_{\alpha}})\Big)Z_{\overline{r},\overline{x}'',\lambda}^{2^{\ast}_{\alpha}-2}\phi\\
&\hspace{8mm}+2^{\ast}_{\alpha}K(x)\Big(|x|^{-\alpha}\ast (K(x)|Z_{\overline{r},\overline{x}'',\lambda}^{2^{\ast}_{\alpha}-1}\phi|)\Big)Z_{\overline{r},\overline{x}'',\lambda}^{2^{\ast}_{\alpha}-1}\Big]\phi\\
\leq& C\|\phi\|_{\ast}^{3}
\int\sum_{j=1}^{m}
\frac{\lambda^{\frac{N+2}{2}}}{(1+\lambda|x-z_{j}|)^{\frac{N+2}{2}+\tau}}\sum_{j=1}^{m}
\frac{\lambda^{\frac{N-2}{2}}}{(1+\lambda|x-z_{j}|)^{\frac{N-2}{2}+\tau}}\\
\leq&\frac{Cm}{\lambda^{3+3\varepsilon}}=O(\frac{m}{\lambda^{2+\varepsilon}}).
\endaligned$$
In order to estimate $\phi_{2}$, we observe that from Lemma \ref{C5} one has 
$$\aligned
\phi_{2}=&\int\Big[K(x)\Big(|x|^{-\alpha}\ast (K(x)|Z_{\overline{r},\overline{x}'',\lambda}|^{2^{\ast}_{\alpha}})\Big)Z_{\overline{r},\overline{x}'',\lambda}^{2^{\ast}_{\alpha}-1}
-\sum_{j=1}^{m}\xi\Big(|x|^{-\alpha}\ast |U_{z_j,\lambda}|^{2^{\ast}_{\alpha}}\Big)U_{z_j,\lambda}^{2^{\ast}_{\alpha}-1}\\
&\hspace{8mm}+Z_{\overline{r},\overline{x}'',\lambda}^{\ast}\Delta\xi+2\nabla\xi\nabla Z_{\overline{r},\overline{x}'',\lambda}^{\ast}\Big]\phi\\
\leq& C\frac{\|\phi\|_{\ast}}{\lambda^{1+\varepsilon}}
\int\sum_{j=1}^{m}
\frac{\lambda^{\frac{N+2}{2}}}{(1+\lambda|x-z_{j}|)^{\frac{N+2}{2}+\tau}}\sum_{j=1}^{m}
\frac{\lambda^{\frac{N-2}{2}}}{(1+\lambda|x-z_{j}|)^{\frac{N-2}{2}+\tau}}\\
\leq&\frac{Cm}{\lambda^{2+2\varepsilon}}=O(\frac{m}{\lambda^{2+\varepsilon}}).
\endaligned$$
So we conclude that
$$\aligned
\int|\nabla \phi|^{2}=O(\frac{m}{\lambda^{2+\varepsilon}}).
\endaligned$$
Note that from Proposition \ref{pro1.1}
$$\aligned
\int\Big(\int_{\R^N}\frac{|\phi(y)|^{2^{\ast}_{\alpha}}}{|x-y|^\alpha}dy\Big)|\phi|^{2^{\ast}_{\alpha}}dx
\leq C\Big(\int |\phi|^{2^{\ast}}\Big)^{\frac{2N-\alpha}{N}} \leq C\Big(\int|\nabla \phi|^{2}\Big)^{2^{\ast}_{\alpha}}\leq O(\frac{m}{\lambda^{2+\varepsilon}}).
\endaligned$$
which completes  the proof.
\end{proof}

\noindent Next step is to prove the following 
\begin{lem}\label{D5}
For any $C^1$ bounded function $g(r, x'')$, it holds
\begin{multline*}
	\int_{D_{\rho}}g(r,x'') \Big( \int_{\R^N}\frac{K(y)|u_{m}(y)|^{2^{\ast}_{\alpha}}}{|x-y|^\alpha}dy \Big) u_{m}^{2^{\ast}_{\alpha}}dx\\
	=m\Big(g(\overline{r},\overline{x}'')
	\int_{\mathbb{R}^{N}} \Big( \int_{\R^N}\frac{K(z_j+\frac{y}{\lambda})|U_{0,1}(y)|^{2^{\ast}_{\alpha}}}{|x-y|^\alpha}dy \Big)U_{0,1}^{2^{\ast}_{\alpha}}dx+o(\frac{1}{\lambda^{1-\epsilon}})\Big).
\end{multline*}
\end{lem}
\begin{proof}
Since $u_{m}=Z_{\overline{r},\overline{x}'',\lambda}+
\phi_{\overline{r},\overline{x}'',\lambda}$, we have
\begin{multline*}
\int_{D_{\rho}}g(r,x'') \Big( \int_{\R^N}\frac{K(y)|u_{m}(y)|^{2^{\ast}_{\alpha}}}{|x-y|^\alpha}dy \Big) u_{m}^{2^{\ast}_{\alpha}}dx\\
=\int_{D_{\rho}}g(r,x'') \Big( \int_{\R^N}\frac{K(y)|Z_{\overline{r},\overline{x}'',\lambda}+
		\phi_{\overline{r},\overline{x}'',\lambda}|^{2^{\ast}_{\alpha}}}{|x-y|^\alpha}dy \Big) (Z_{\overline{r},\overline{x}'',\lambda}+
	\phi_{\overline{r},\overline{x}'',\lambda})^{2^{\ast}_{\alpha}}dx\\
	=\int_{D_{\rho}}g(r,x'') \Big( \int_{\R^N}\frac{K(y)|Z_{\overline{r},\overline{x}'',\lambda}(y)|^{2^{\ast}_{\alpha}}}{|x-y|^\alpha}dy \Big) (Z_{\overline{r},\overline{x}'',\lambda})^{2^{\ast}_{\alpha}}dx\\
	+\int_{D_{\rho}}g(r,x'') \Big( \int_{\R^N}\frac{K(y)|\phi_{\overline{r},\overline{x}'',\lambda}(y)|^{2^{\ast}_{\alpha}}}{|x-y|^\alpha}dy \Big) (\phi_{\overline{r},\overline{x}'',\lambda})^{2^{\ast}_{\alpha}}dx\\
	+O\Big(\int_{D_{\rho}}g(r,x'') \Big( \int_{\R^N}\frac{K(y)|Z_{\overline{r},\overline{x}'',\lambda}(y)|^{2^{\ast}_{\alpha}}}{|x-y|^\alpha}dy \Big) Z_{\overline{r},\overline{x}'',\lambda}^{2^{\ast}_{\alpha}-1}\phi_{\overline{r},\overline{x}'',\lambda}dx\\
	 +\int_{D_{\rho}}g(r,x'') \Big( \int_{\R^N}\frac{K(y)|Z_{\overline{r},\overline{x}'',\lambda}(y)|^{2^{\ast}_{\alpha}}}{|x-y|^\alpha}dy \Big) Z_{\overline{r},\overline{x}'',\lambda}^{2^{\ast}_{\alpha}-2}\phi^2_{\overline{r},\overline{x}'',\lambda}dx\\
	+\int_{D_{\rho}}g(r,x'') \Big( \int_{\R^N}\frac{K(y)|Z_{\overline{r},\overline{x}'',\lambda}(y)|^{2^{\ast}_{\alpha}}}{|x-y|^\alpha}dy \Big) \phi_{\overline{r},\overline{x}'',\lambda}^{2^{\ast}_{\alpha}}dx\\
	+\int_{D_{\rho}}g(r,x'') \Big( \int_{\R^N}\frac{K(y)|Z_{\overline{r},\overline{x}'',\lambda}^{2^{\ast}_{\alpha}-1}\phi_{\overline{r},\overline{x}'',\lambda}(y)|}{|x-y|^\alpha}dy \Big) Z_{\overline{r},\overline{x}'',\lambda}^{2^{\ast}_{\alpha}-1}\phi_{\overline{r},\overline{x}'',\lambda}dx\\
	+\int_{D_{\rho}}g(r,x'') \Big( \int_{\R^N}\frac{K(y)|Z_{\overline{r},\overline{x}'',\lambda}^{2^{\ast}_{\alpha}-1}\phi_{\overline{r},\overline{x}'',\lambda}(y)|}{|x-y|^\alpha}dy \Big) Z_{\overline{r},\overline{x}'',\lambda}^{2^{\ast}_{\alpha}-2}\phi^2_{\overline{r},\overline{x}'',\lambda}dx\\
	+\int_{D_{\rho}}g(r,x'') \Big( \int_{\R^N}\frac{K(y)|Z_{\overline{r},\overline{x}'',\lambda}^{2^{\ast}_{\alpha}-1}(y)\phi_{\overline{r},\overline{x}'',\lambda}|}{|x-y|^\alpha}dy \Big) \phi_{\overline{r},\overline{x}'',\lambda}^{2^{\ast}_{\alpha}}dx \\
	+\int_{D_{\rho}}g(r,x'') \Big( \int_{\R^N}\frac{K(y)|Z_{\overline{r},\overline{x}'',\lambda}^{2^{\ast}_{\alpha}-2}\phi^2_{\overline{r},\overline{x}'',\lambda}(y)|}{|x-y|^\alpha}dy \Big) Z_{\overline{r},\overline{x}'',\lambda}^{2^{\ast}_{\alpha}-2}\phi^2_{\overline{r},\overline{x}'',\lambda}dx\\
	+\int_{D_{\rho}}g(r,x'') \Big( \int_{\R^N}\frac{K(y)|Z_{\overline{r},\overline{x}'',\lambda}^{2^{\ast}_{\alpha}-2}\phi^2_{\overline{r},\overline{x}'',\lambda}(y)|}{|x-y|^\alpha}dy \Big) \phi_{\overline{r},\overline{x}'',\lambda}^{2^{\ast}_{\alpha}}dx\Big)\ .
\end{multline*}
Using the estimates in Lemma \ref{C4}, we find that
$$\aligned
&\left|\int_{D_{\rho}}g(r,x'') \Big( \int_{\R^N}\frac{K(y)|Z_{\overline{r},\overline{x}'',\lambda}(y)|^{2^{\ast}_{\alpha}}}{|x-y|^\alpha}dy \Big)Z_{\overline{r},\overline{x}'',\lambda}^{2^{\ast}_{\alpha}-1}\phi_{\overline{r},\overline{x}'',\lambda}dx\right|\\
&\leq C\|\phi\|_{\ast}\int\sum_{i=1}^{m}\frac{\lambda^{\frac{\alpha}{2}}}{(1+\lambda|x-z_{i}|)^{\alpha}}\Big(\sum_{j=1}^{m}\frac{\lambda^{\frac{N-2}{2}}}{(1+\lambda|x-z_{j}|)^{\frac{N-2}{2}+\tau}}\Big)^{2^{\ast}_{\alpha}}\\
&\leq
C\|\phi\|_{\ast}\int\sum_{i=1}^{m}\frac{\lambda^{\frac{\alpha}{2}}}{(1+\lambda|x-z_{i}|)^{\alpha}}\sum_{j=1}^{m}
\frac{\lambda^{N-\frac{\alpha}{2}}}{(1+\lambda|x-z_{j}|)^{N-\frac{\alpha}{2}+\tau}}\Big(\sum_{j=1}^{m}\frac{1}{(1+\lambda|x-z_{j}|)^{\tau}}\Big)^{2^{\ast}_{\alpha}-1}\\
&\leq C\|\phi\|_{\ast}\int\sum_{i=1}^{m}\frac{1}{(1+|x-\lambda z_{i}|)^{N+\frac{\alpha}{2}+\tau}}\\
&\hspace{8mm}+C\|\phi\|_{\ast}\int\sum_{i\neq j}\frac{1}{(\lambda|z_{i}-z_{j}|)^{\tau}}\Big(\frac{1}{(1+|x-\lambda z_{i}|)^{N+\frac{\alpha}{2}}}+\frac{\lambda}{(1+|x-\lambda z_{j}|)^{N+\frac{\alpha}{2}}}\Big)\\
&\leq Cm\|\phi\|_{\ast}+C\|\phi\|_{\ast}
\leq \frac{Cm}{\lambda^{1+\varepsilon}}=o(\frac{m}{\lambda^{1-\varepsilon}}).
\endaligned$$
We can estimate the remaining terms in a similar fashion to conclude that
\begin{multline*}
\int_{D_{\rho}}g(r,x'') \Big( \int_{\R^N}\frac{K(y)|u_{m}(y)|^{2^{\ast}_{\alpha}}}{|x-y|^\alpha}dy \Big) u_{m}^{2^{\ast}_{\alpha}}dx\\
=\int_{D_{\rho}}g(r,x'') \Big( \int_{\R^N}\frac{K(y)|Z_{\overline{r},\overline{x}'',\lambda}|^{2^{\ast}_{\alpha}}}{|x-y|^\alpha}dy \Big) Z_{\overline{r},\overline{x}'',\lambda}^{2^{\ast}_{\alpha}}dx
+o(\frac{m}{\lambda^{1-\varepsilon}})\ .
\end{multline*}
Observe that
$$\aligned
&\int_{D_{\rho}}g(r,x'') \Big( \int_{\R^N}\frac{K(y)|Z_{z_j,\lambda}|^{2^{\ast}_{\alpha}}}{|x-y|^\alpha}dy \Big) Z_{z_j,\lambda}^{2^{\ast}_{\alpha}}dx\\
=&g(\overline{r},\overline{x}'')\int_{\R^N} \Big( \int_{\R^N}\frac{K(y)|Z_{z_j,\lambda}|^{2^{\ast}_{\alpha}}}{|x-y|^\alpha}dy \Big) Z_{z_j,\lambda}^{2^{\ast}_{\alpha}}dx+\int_{D_{\rho}^C}g(\overline{r},\overline{x}'') \Big( \int_{\R^N}\frac{K(y)|Z_{z_j,\lambda}|^{2^{\ast}_{\alpha}}}{|x-y|^\alpha}dy \Big) Z_{z_j,\lambda}^{2^{\ast}_{\alpha}}dx\\
&+\int_{D_{\rho}}[g(r,x'')-g(\overline{r},\overline{x}'')] \Big( \int_{\R^N}\frac{K(y)|Z_{z_j,\lambda}|^{2^{\ast}_{\alpha}}}{|x-y|^\alpha}dy \Big) Z_{z_j,\lambda}^{2^{\ast}_{\alpha}}dx\\
=&g(\overline{r},\overline{x}'')
\int_{\mathbb{R}^{N}} \Big( \int_{\R^N}\frac{K(z_j+\frac{y}{\lambda})|U_{0,1}(y)|^{2^{\ast}_{\alpha}}}{|x-y|^\alpha}dy \Big)U_{0,1}^{2^{\ast}_{\alpha}}dx+o(\frac{1}{\lambda^{1-\varepsilon}}).
\endaligned$$
Moreover, one has 
$$\aligned
&\int_{D_{\rho}}g(r,x'')\Big(|x|^{-\alpha}\ast K|Z_{z_i,\lambda}|^{2^{\ast}_{\alpha}}\Big)Z_{z_i,\lambda}^{2^{\ast}_{\alpha}-1}\sum_{j\neq i}Z_{z_j,\lambda}dx\leq C\sum_{j\neq i}\int\frac{\lambda^{\frac{N+2}{2}}}{(1+\lambda|x-z_{i}|)^{N+2}}\frac{\lambda^{\frac{N-2}{2}}}{(1+\lambda|x-z_{j}|)^{N-2}}\\
&\leq C\sum_{j\neq i}\frac{1}{(\lambda|z_{i}-z_{j}|)^{N-2}}\int
\Big(\frac{\lambda^{N}}{(1+\lambda|x-z_{i}|)^{N+2}}+\frac{\lambda^{N}}{(1+\lambda|x-z_{j}|)^{N+2}}\Big)dx\\
&\leq C\sum_{j\neq i}(\frac{m}{\lambda})^{N-2}=O(\frac{m}{\lambda^{1+\varepsilon}})=o(\frac{1}{\lambda^{1-\varepsilon}}),
\endaligned$$
and 
$$\aligned
&\int_{D_{\rho}}g(r,x'')\Big(|x|^{-\alpha}\ast K|Z_{z_i,\lambda}|^{2^{\ast}_{\alpha}}\Big)Z_{z_i,\lambda}^{2^{\ast}_{\alpha}-2}\big(\sum_{j\neq i}Z_{z_j,\lambda}\big)^2dx\leq C\sum_{j\neq i}\int\frac{\lambda^{2}}{(1+\lambda|x-z_{i}|)^{4}}\frac{\lambda^{N-2}}{(1+\lambda|x-z_{j}|)^{2N-4-\tau}}\\
&\leq C\sum_{j\neq i}\frac{1}{(\lambda|z_{i}-z_{j}|)^{4}}\int
\Big(\frac{\lambda^{N}}{(1+\lambda|x-z_{i}|)^{2N-4-\tau}}+\frac{\lambda^{N}}{(1+\lambda|x-z_{j}|)^{2N-4-\tau}}\Big)dx\\
&\leq C\sum_{j\neq i}(\frac{m}{\lambda})^{4}=O(\frac{m}{\lambda^{1+\varepsilon}})=o(\frac{1}{\lambda^{1-\varepsilon}}),
\endaligned$$
and similarly 
$$\aligned
&\int_{D_{\rho}}g(r,x'')\Big(|x|^{-\alpha}\ast K|Z_{z_i,\lambda}|^{2^{\ast}_{\alpha}}\Big)(\sum_{j\neq i}Z_{z_j,\lambda})^{2^{\ast}_{\alpha}}dx\leq \frac{Cm}{\lambda^{1+\varepsilon}}=o(\frac{1}{\lambda^{1-\varepsilon}}).
\endaligned$$
We can deduce that all the remaining terms are bounded by $\frac{Cm}{\lambda^{1+\varepsilon}}$. As a consequence,
\begin{multline*}
\int_{D_{\rho}}g(r,x'') \Big( \int_{\R^N}\frac{K(y)|u_{m}(y)|^{2^{\ast}_{\alpha}}}{|x-y|^\alpha}dy \Big) u_{m}^{2^{\ast}_{\alpha}}dx\\
=m\Big(g(\overline{r},\overline{x}'')
\int_{\mathbb{R}^{N}} \Big( \int_{\R^N}\frac{K(z_j+\frac{y}{\lambda})|U_{0,1}(y)|^{2^{\ast}_{\alpha}}}{|x-y|^\alpha}dy \Big)U_{0,1}^{2^{\ast}_{\alpha}}dx+o(\frac{1}{\lambda^{1-\epsilon}})\Big).
\end{multline*}
\end{proof}

\noindent Lemma \ref{D4} implies that
\begin{multline*}
\int_{D_{4\delta}\backslash D_{3\delta}}
|\nabla \phi|^{2}dx+\int_{D_{4\delta}\backslash D_{3\delta}}\Big(\int_{D_{\rho}} \frac{|\phi(y)|^{2^{\ast}_{\alpha}}}{|x-y|^\alpha}dy\Big)|\phi|^{2^{\ast}_{\alpha}}dx\\
+\int_{ D_{4\delta}\backslash D_{3\delta}}\int_{\R^N\backslash D_{\rho}}x(x-y)\frac{|\phi(x)|^{2^{\ast}_{\alpha}}|\phi(y)|^{2^{\ast}_{\alpha}}}{|x-y|^{\alpha+2}}dxdy
=O(\frac{m}{\lambda^{2+\varepsilon}}),
\end{multline*}
where $i=3,\cdot\cdot\cdot,N$. As a consequence, we can find $\rho\in(3\delta,4\delta)$, such that
$$
\int_{\partial D_{\rho}}|\nabla \phi|^{2}ds
+\int_{\partial D_{\rho}}\Big(\int_{D_{\rho}} \frac{|\phi(y)|^{2^{\ast}_{\alpha}}}{|x-y|^\alpha}dy\Big)|\phi|^{2^{\ast}_{\alpha}}ds+
\int_{ D_{\rho}}\int_{\R^N\backslash D_{\rho}}x(x-y)\frac{|\phi(x)|^{2^{\ast}_{\alpha}}|\phi(y)|^{2^{\ast}_{\alpha}}}{|x-y|^{\alpha+2}}dxdy=O(\frac{m}{\lambda^{2+\varepsilon}}),
$$
where $i=3,\cdot\cdot\cdot,N$. By Lemma \ref{D5}, for any $C^1$ function $g(r, x'')$, it holds
\begin{multline*}
\int_{D_{\rho}}g(r,x'') \Big( \int_{\R^N}\frac{K(y)|u_{m}(y)|^{2^{\ast}_{\alpha}}}{|x-y|^\alpha}dy \Big) u_{m}^{2^{\ast}_{\alpha}}dx\\
=m\Big(g(\overline{r},\overline{x}'')
\int_{\mathbb{R}^{N}} \Big( \int_{\R^N}\frac{K(z_j+\frac{y}{\lambda})|U_{0,1}(y)|^{2^{\ast}_{\alpha}}}{|x-y|^\alpha}dy \Big)U_{0,1}^{2^{\ast}_{\alpha}}dx+o(\frac{1}{\lambda^{1-\epsilon}})\Big).
\end{multline*}
From \eqref{d11} and \eqref{d18}, we obtain
$$
m(\frac{\partial K(\overline{r},\overline{x}'')}{\partial \overline{x}_{i}}\int_{\mathbb{R}^{N}} \Big( \int_{\R^N}\frac{K(z_j+\frac{y}{\lambda})|U_{0,1}(y)|^{2^{\ast}_{\alpha}}}{|x-y|^\alpha}dy \Big)U_{0,1}^{2^{\ast}_{\alpha}}dx+o(\frac{1}{\lambda^{1-\epsilon}}))=o(\frac{m}{\lambda^{2}}),
$$
and
$$
m(\frac{\overline{r}}{2^{\ast}_{\alpha}}\frac{\partial K(\overline{r},\overline{x}'')}{\partial \overline{r}}\int_{\mathbb{R}^{N}} \Big( \int_{\R^N}\frac{K(z_j+\frac{y}{\lambda})|U_{0,1}(y)|^{2^{\ast}_{\alpha}}}{|x-y|^\alpha}dy \Big)U_{0,1}^{2^{\ast}_{\alpha}}dx+o(\frac{1}{\lambda^{1-\epsilon}}))=o(\frac{m}{\lambda^{2}}).
$$
As the function $K$ is bounded, the equations to determine $(\overline{r},\overline{x}'')$ are the following 
\begin{equation}\label{d21}
\frac{\partial K(\overline{r},\overline{x}'')}{\partial \overline{x}_{i}}=o(\frac{1}{\lambda^{1-\epsilon}}),  i=3,\cdot\cdot\cdot,N,
\end{equation}
and
\begin{equation}\label{d22}
\frac{\partial K(\overline{r},\overline{x}'')}{\partial \overline{r}}=o(\frac{1}{\lambda^{1-\epsilon}}).
\end{equation}

\noindent
{\bf Proof of Lemma \ref{EXS1}.}  We have proved that \eqref{d1}, \eqref{d2} and \eqref{d3} are equivalent respectively to \begin{equation}\label{d21}
\frac{\partial K(\overline{r},\overline{x}'')}{\partial \overline{x}_{i}}=o(\frac{1}{\lambda^{1-\epsilon}}),  i=3,\cdot\cdot\cdot,N,
\end{equation}

\begin{equation}\label{d22}
\frac{\partial K(\overline{r},\overline{x}'')}{\partial \overline{r}}=o(\frac{1}{\lambda^{1-\epsilon}}), 
\end{equation}
and 
$$
-\frac{A_1}{\lambda^{3}}+\frac{m^{N-2}A_3}{\lambda^{N-1}}=O(\frac{1}{\lambda^{3+\varepsilon}}).
$$
Let $\lambda=tm^{\frac{N-2}{N-4}}$, so that $t\in[L_0, L_1]$. Therefore, we have 
\begin{equation}\label{d23}
-\frac{A_1}{t^{3}}+\frac{A_3}{t^{N-1}}=o(1), t\in[L_0, L_1].
\end{equation}
Let
$$
F(t, \overline{r},\overline{x}'')=\big(\nabla_{\overline{r},\overline{x}''}(K(\overline{r},\overline{x}'')),
-\frac{A_1}{t^{3}}+\frac{A_3}{t^{N-1}}\big).
$$
Then
$$
\mbox{deg}(F(t, \overline{r},\overline{x}''), [L_0, L_1]\times B_{\theta}((r_{0},x_{0}'')))
=-\mbox{deg}(\nabla_{\overline{r},\overline{x}''}(K(\overline{r},x_{0}'')),B_{\frac{1}{\lambda^{1-\theta}}}((r_{0},x_{0}'')))\neq0.
$$
So, \eqref{d21}, \eqref{d22} and \eqref{d23} have a solution $t_{m}\in[L_0, L_1]$, $(\overline{r}_{m},\overline{x}_{m}'')\in B_{\frac{1}{\lambda^{1-\theta}}}((r_{0},x_{0}''))$.
$\hfill{} \Box$

\begin{appendix}
\section{Some basic estimates}
\begin{lem}\label{B2} (Lemma B.1, \cite{WY1}) For each fixed $k$ and $j$, $k\neq j$,  let
	$$
	g_{k,j}(x)=\frac{1}{(1+|x-z_{j}|)^{\alpha}}\frac{1}{(1+|x-z_{k}|)^{\beta}},
	$$
	where $\alpha\geq 1$ and $\beta\geq1$ are two constants.
	Then, for any constants $0<\delta\leq\min\{\alpha,\beta\}$, there is a constant $C>0$, such that
	$$
	g_{k,j}(x)\leq\frac{C}{|z_{k}-z_{j}|^{\delta}}\Big(\frac{1}{(1+|x-z_{j}|)^{\alpha+\beta-\delta}}+\frac{1}{(1+|x-z_{k}|)^{\alpha+\beta-\delta}}\Big).
	$$
\end{lem}

\begin{lem}\label{B3} (Lemma B.2, \cite{WY1}) For any constant $0<\delta<N-2$, $N\geq5$, there is a constant $C>0$, such that
	$$
	\int_{\mathbb{R}^{N}}\frac{1}{|x-y|^{N-2}}\frac{1}{(1+|y|)^{2+\delta}}dy\leq \frac{C}{(1+|x|)^{\delta}}.
	$$
\end{lem}

\begin{lem}\label{B4} 
	For $N>5$ and $1\leq i\leq m$, there is a constant $C>0$, such that
	$$
	|x|^{-\alpha}\ast \frac{\lambda^{N-\frac{\alpha}{2}}}{(1+\lambda|x-z_{i}|)^{\frac{3N+2}{2}-\alpha+\eta}}\leq \frac{C\lambda^{\frac{\alpha}{2}}}{(1+\lambda|x-z_{i}|)^{\min\{\alpha,\frac{N+2}{2}\}}},
	$$
	where $\eta>0$.
\end{lem}
\begin{proof}
	Notice first that
	$$\aligned
	|x|^{-\alpha}\ast \frac{\lambda^{N-\frac{\alpha}{2}}}{(1+\lambda|x-z_{i}|)^{\frac{3N+2}{2}-\alpha+\eta}}=&
	\int_{\mathbb{R}^{N}}\frac{1}{|y|^{\alpha}}\frac{\lambda^{N-\frac{\alpha}{2}}}{(1+\lambda|x-z_{i}-y|)^{\frac{3N+2}{2}-\alpha+\eta}}dy\\
	=&
	\int_{\mathbb{R}^{N}}\frac{1}{|y|^{\alpha}}\frac{\lambda^{\frac{\alpha}{2}}}{(1+|\lambda x-\lambda z_{i}-y|)^{\frac{3N+2}{2}-\alpha+\eta}}dy.
	\endaligned$$
	
	Let $d=\frac{\lambda}{2}|x-z_{i}|>1$. Then, we have
	$$\aligned
	\int_{B_{d}(0)}\frac{1}{|y|^{\alpha}}\frac{1}{(1+|\lambda x-\lambda z_{i}-y|)^{\frac{3N+2}{2}-\alpha+\eta}}dy&\leq \frac{C}{(1+ d)^{\frac{3N+2}{2}-\alpha+\eta}}\int_{B_{d}(0)}\frac{1}{|y|^{\alpha}}dy\\
	&\leq \frac{C d^{N-\alpha}}{(1+ d)^{\frac{3N+2}{2}-\alpha+\eta}}\leq\frac{ C}{(1+ d)^{\frac{N+2}{2}+\eta}},
	\endaligned$$
	and
	$$\aligned
	&\int_{B_{d}(\lambda x-\lambda z_{i})}\frac{1}{|y|^{\alpha}}\frac{1}{(1+|\lambda x-\lambda z_{i}-y|)^{\frac{3N+2}{2}-\alpha+\eta}}dy\\
	&\hspace{5mm}\leq \frac{1}{d^{\alpha}}\int_{B_{d}(0)}\frac{1}{(1+|y|)^{\frac{3N+2}{2}-\alpha+\eta}}dy
	\leq \frac{C}{(1+ d)^{\min\{\alpha, \frac{N+2}{2}\}}}.
	\endaligned$$

\noindent Assume $y\in\mathbb{R}^{N}\backslash (B_{d}(0)\cup B_{d}(\lambda x-\lambda z_{i}))$. Then,
	$$
	|\lambda x-\lambda z_{i}-y|\geq\frac{1}{2}|\lambda x-\lambda z_{i}|, |y|\geq\frac{1}{2}|\lambda x-\lambda z_{i}|
	$$
	and we have
	$$
	\frac{1}{|y|^{\alpha}}\frac{1}{(1+|\lambda x-\lambda z_{i}-y|)^{\frac{3N+2}{2}-\alpha+\eta}}\leq \frac{ C}{(1+ d)^{\frac{N+2}{2}}}\frac{1}{|y|^{\alpha}}\frac{1}{(1+|\lambda x-\lambda z_{i}-y|)^{N-\alpha+\eta}}.
	$$
	If $|y|\leq2|\lambda x-\lambda z_{i}|$, then
	$$
	\frac{1}{|y|^{\alpha}}\frac{1}{(1+|\lambda x-\lambda z_{i}-y|)^{N-\alpha+\eta}}\leq
	\frac{C}{|y|^{\alpha}(1+|\lambda x-\lambda z_{i}|)^{N-\alpha+\eta}}\leq
	\frac{C_{1}}{|y|^{\alpha}(1+|y|)^{N-\alpha+\eta}}.
	$$
	If $|y|\geq 2|\lambda x-\lambda z_{i}|$, then $|\lambda x-\lambda z_{i}-y|\geq|y|-|\lambda x-\lambda z_{i}|\geq\frac{1}{2}|y|$. As a consequence,
	$$
	\frac{1}{|y|^{\alpha}}\frac{1}{(1+|\lambda x-\lambda z_{i}-y|)^{N-\alpha+\eta}}\leq
	\frac{C}{|y|^{\alpha}(1+|y|)^{N-\alpha+\eta}}.
	$$
	Thus, we have
	$$\aligned
	&\int_{\mathbb{R}^{N}\backslash (B_{d}(0)\cup B_{d}(\lambda x-\lambda z_{i}))}\frac{1}{|y|^{\alpha}}\frac{1}{(1+|\lambda x-\lambda z_{i}-y|)^{\frac{3N+2}{2}-\alpha+\eta}}dy\\
	&\leq \frac{ C}{(1+ d)^{\frac{N+2}{2}}}\int_{\mathbb{R}^{N}}\frac{1}{|y|^{\alpha}}\frac{1}{(1+|y|)^{N-\alpha+\eta}}dy
	\leq \frac{ C}{(1+ d)^{\frac{N+2}{2}}}.
	\endaligned$$
\end{proof}
\noindent Using \eqref{REL} and the identity (see (37) in \cite{DHQWF} for instance)
\begin{equation}\label{eq3.18}
	\int_{\R^N}\frac{1}{|x-y|^{2s}}\Big(\frac{1}{1+|y|^{2}}\Big)^{N-s}dy
	=I(s)\Big(\frac{1}{1+|x|^{2}}\Big)^{s},\ \ 0 < s < \frac{N}{2},
\end{equation}
where
$$
I(s)=\frac{\pi^{\frac{N}{2}}\Gamma(\frac{N-2s}{2})}{\Gamma(N-s)}, \ \ \mbox{and }\Gamma(s)=\int_0^{+\infty} x^{s-1}e^{-x}\,dx, s>0,
$$
we have
$$
|x|^{-\alpha}\ast |U_{z,\lambda}(x)|^{2^{\ast}_{\alpha}}
=\int_{\R^N}\frac{U_{z,\lambda}^{2^{\ast}_{\alpha}}(y)}{|x-y|^{\alpha}}dy
=C\Big(\frac{\lambda}{1+\lambda^2|x-z|^{2}}\Big)^{\frac{\alpha}{2}},
$$
where $N\geq9$ and $C=I(\frac{\alpha}{2})C^{2^{\ast}_{\alpha}}(N,\alpha)$. Finally, we have the following lemma
\begin{lem}\label{P0}
	For $N=6$ and $1\leq i\leq m$, there is a constant $C>0$, such that
	$$
	|x|^{-\alpha}\ast |U_{z_{i},\lambda}(x)|^{2^{\ast}_{\alpha}}
	=C\Big(\frac{\lambda}{1+\lambda^2|x-z|^{2}}\Big)^{\frac{\alpha}{2}}.
	$$
\end{lem}



\vspace{1cm}
\end{appendix}


\begin{thebibliography}{99}

\bibitem{BC}
\newblock A. Bahri, J. Coron,
\newblock \emph{The scalar-curvature problem on the standard three-dimensional sphere},
\newblock J. Funct. Anal, \textbf{95} (1991), 106--172.


\bibitem{BG}
\newblock G. Bianchi,
\newblock \emph{Non-existence and symmetry of solutions to the scalar-curvature equation},
\newblock Commun. Partial Differ. Equ., \textbf{21} (1996), 229--234.

\bibitem{CNY}
\newblock  D. Cao, E. Noussair, S.Yan,
\newblock \emph{On the scalar curvature equation $-\Delta u=(1+\varepsilon K)u^{\frac{N+2}{N-2}}$ in $\mathbb{R}^N$},
\newblock Calc. Var. Partial Differ. Equ., \textbf{15} (2002), 403--419.

\bibitem{CPY}
\newblock D. Cao, S. Peng, S. Yan,
\newblock \emph{Singularly perturbed methods for nonlinear elliptic problems},
\newblock Cambridge University Press, \textbf{} (2021).

\bibitem{CT} D. Cassani, C. Tarsi, 
{\em Schr\"odinger-Newton equations in dimension two via a Pohozaev-Trudinger log-weighted inequality.} Calc. Var. Partial Differential Equations 60 (2021), no. 5, Paper No. 197, 31 pp.

\bibitem{CVZ} D. Cassani, J. Van Schaftingen, J. Zhang, {\em Groundstates for Choquard type equations with Hardy-Littlewood-Sobolev lower critical exponent.} Proc. R. Soc. Edinb. Sect. A 150 (2020), 1377-1400.


\bibitem{CZ} D. Cassani, J. Zhang,
{\em Choquard-type equations with Hardy-Littlewood-Sobolev upper-critical growth.} Adv. Nonlinear Anal. 8 (2019), 1184-1212.

\bibitem{CL}
\newblock  C.C. Chen, C.S. Lin,
\newblock \emph{Prescribing scalar curvature on $S^{n}$. I. A priori estimates},
\newblock J. Differ. Geom., \textbf{57} (2001), 67--171.

\bibitem{CLO1}
\newblock W. Chen, C. Li \& B. Ou.
\newblock \emph{Classification of solutions for an integral equation}.
\newblock Comm. Pure Appl. Math. \textbf{59} (2006), 330--343.

\bibitem{CLO2}
\newblock W. Chen, C. Li \& B. Ou.
\newblock \emph{Classification of solutions for a system of integral equations}.
\newblock Comm. Partial Differential Equations. \textbf{30} (2005), 59--65.

\bibitem{CY}
\newblock  C. Chang, P. Yang,
\newblock \emph{A perturbation result in prescribing scalar curvature on $S^{n}$},
\newblock Duke Math. J., \textbf{64} (1991), 27--69.


\bibitem{DFM}
\newblock M. del Pino, P. Felmer \& M. Musso,
\newblock \emph{Two-bubble solutions in the super-critical Bahri-Coron's problem},
\newblock  Calc. Var. Partial Differential Equations, \textbf{16} (2003), 113--145.

\bibitem{DHQWF}
\newblock W. Dai, J. Huang, Y. Qin, B. Wang \& Y. Fang,
\newblock \emph{Regularity and classification of solutions to static Hartree
	equations involving fractional Laplacians},
\newblock Discrete Contin. Dyn. Syst., \textbf{39} (2019), 1389--1403.


\bibitem{DLY}
\newblock Y. Deng, C.-S. Lin \& S. Yan,
\newblock \emph{On the prescribed scalar curvature problem in $\R^N$, local uniqueness
	and periodicity},
\newblock J. Math. Pures Appl., \textbf{104} (2015), 1013--1044.


\bibitem{DY}
\newblock L. Du \& M. Yang,
\newblock \emph{Uniqueness and nondegeneracy of solutions for a
	critical nonlocal equation},
\newblock Discrete Contin. Dyn. Syst., \textbf{39} (2019),  5847--5866.

\bibitem{GHLN}
\newblock Y. Guo, Y. Hu, T, Liu \&J. Nie,
\newblock \emph{Non-degeneracy of the Bubble Solutions for the Fractional Prescribed Curvature Problem and Applications},
\newblock J. Geom. Anal., \textbf{33} (2023), 141.

\bibitem{GHPS}
\newblock L. Guo, T. Hu,  S. Peng and W. Shuai,
\newblock \emph{Existence and uniqueness of solutions for Choquard equation involving Hard-Littlewood-Sobolev critical exponent},
\newblock  Calc. Var. partial Diff. Equ.,
\textbf {58} (2019), 128, 34 pp.

\bibitem{GL}
\newblock Y. Guo, B. Li,
\newblock \emph{Infinitely many solutions for the prescribed curvature problem of polyharmonic operator},
\newblock Calc. Var. Partial Differ. Equ., \textbf{46} (2013), 809--836.

\bibitem{GLP}
\newblock Q. Guo, J. Liu, S. Peng,
\newblock \emph{Existence and non-degeneracy of positive multi-bubbling solutions to critical elliptic systems of Hamiltonian type},
\newblock J. Differ. Equ., \textbf{355} (2023), 16--61.


\bibitem{GMPY}
\newblock Y. Guo, M. Musso, S. Peng \&S. Yan,
\newblock \emph{Non-degeneracy of multi-bubbling solutions for the prescribed scalar curvature equations and applications},
\newblock J. Funct. Anal., \textbf{279} (2020), 108553.

\bibitem{GMYZ}
\newblock F. Gao, V. Moroz, M. Yang, S. Zhao,
\newblock \emph{Construction of infinitely many solutions for a critical Choquard equation via local Pohozaev identities}.
\newblock Calc. Var. Partial Differential Equations., \textbf{61} (2022), 222.

\bibitem{GN}
\newblock Y. Guo, J. Nie,
\newblock \emph{Infinitely many non-radial solutions for the prescribed curvature problem of fractional operator},
\newblock Discrete Contin. Dyn. Syst., \textbf{36} (2016), 6873--6898.


\bibitem{GNNT}
\newblock Y. Guo, J. Nie, M. Niu \& Z, Tang
\newblock \emph{Local uniqueness and periodicity for the prescribed scalar curvature problem of fractional operator in $\R^N$},
\newblock Calc. Var. Partial Differential Equations., \textbf{56} (2017), 118.




\bibitem{GY}
\newblock F. Gao \& M. Yang,
\newblock \emph{On the Brezis-Nirenberg type critical problem for nonlinear Choquard equation},
\newblock Sci China Math, \textbf{61} (2018), 1219--1242.

\bibitem{H}
\newblock Z. Han,
\newblock \emph{Prescribing Gaussian curvature on $S^2$},
\newblock Duke Math. J., \textbf{61} (1990), 679--703.


\bibitem{Lei}
\newblock Y. Lei,
\newblock \emph{Liouville theorems and classification results for a
	nonlocal Schr\"odinger equation},
\newblock  Discrete Contin. Dyn. Syst., \textbf{38} (2018), 5351--5377.

\bibitem{L1}
\newblock Y. Y. Li,
\newblock \emph{On $-\Delta u=K(x)u^5$ in $R^3$},
\newblock Commun. Pure Appl. Math., \textbf{46} (1993), 303--340.

\bibitem{L2}
\newblock Y. Y. Li,
\newblock \emph{Prescribed scalar curvature on $S^3$, $S^4$ and related problems},
\newblock J. Funct. Anal., \textbf{118} (1993), 43--118.

\bibitem{L3}
\newblock Y. Y. Li,
\newblock \emph{Prescribed scalar curvature on $S^n$ and related problems},
\newblock J. Differ. Equ., \textbf{120} (1995), 319--410.

\bibitem{Lieb1}
\newblock E. H. Lieb,
\newblock \emph{Existence and uniqueness of the minimizing solution of Choquard's nonlinear equation},
\newblock  Studies in Appl. Math., \textbf{57} (1976/77), 93--105.


\bibitem{Ls}
\newblock P.-L. Lions,
\newblock \emph{The Choquard equation and related questions},
\newblock Nonlinear Anal., \textbf{4} (1980), 1063--1072.



\bibitem{LL}
\newblock E. H. Lieb \& M. Loss, \newblock Analysis,
\newblock \emph{Gradute Studies in Mathematics}, AMS,
Providence, Rhode island, 2001.

\bibitem{LLTX}
\newblock X. Li, C. Liu, X. Tang, G. Xu,
\newblock \emph{Nondegeneracy of positive bubble solutions for generalized energy-critical Hartree equations}. Preprint 2023, \url{https://arxiv.org/pdf/2304.04139}.

\bibitem{LWX}
\newblock Y. Li, J. Wei \& H. Xu,
\newblock \emph{Multi-bump solutions of $-\Delta u = K(x)u^{\frac{n+2}{n-2}}$ on lattices in $\R^n$},
\newblock J. Reine Angew. Math. {\bf 743} (2018), 163--211.


\bibitem{MS1}
\newblock V. Moroz \& J. Van Schaftingen,
\newblock \emph{Groundstates of nonlinear Choquard equations: Existence, qualitative properties and decay asymptotics},
\newblock J. Funct. Anal., \textbf{265} (2013), 153--184.

\bibitem{MS2}
\newblock V. Moroz \& J. Van Schaftingen,
\newblock \emph{Existence of groundstates for a class of nonlinear Choquard equations,}
\newblock Trans. Amer. Math. Soc.,\textbf{367} (2015), 6557--6579.

\bibitem{guide}
\newblock V. Moroz \& J. Van Schaftingen,
\newblock \emph{A guide to the Choquard equation},
\newblock J. Fixed Point Theory Appl., \textbf{19} (2017), 773--813.

\bibitem{P1}
\newblock S. Pekar,
\newblock \emph{Untersuchung\"{u}ber die Elektronentheorie der Kristalle},
\newblock Akademie Verlag, Berlin, 1954.

\bibitem{PWW}
\newblock S. Peng, C. Wang, S. Wei,
\newblock \emph{Construction of solutions for the prescribed scalar curvature problem via local Pohozaev identities},
\newblock J. Differ. Equ., \textbf{267} (2019), 2503--2530.



\bibitem{PWY}
\newblock S. Peng, C. Wang \& S. Yan,
\newblock \emph{Construction of solutions via local Pohozaev identities},
\newblock  J. Funct. Anal., \textbf{274} (2018), 2606--2633.

\bibitem{R} G. Romani, 
{\em Schr\"odinger-Poisson systems with zero mass in the Sobolev limiting case.} Preprint 2023, \url{https://arxiv.org/pdf/2310.08460.pdf}.


\bibitem{WY1}
\newblock J. Wei \& S. Yan,
\newblock \emph{Infinitely many solutions for the prescribed scalar curvature problem on $S^{N}$},
\newblock J. Funct. Anal., \textbf{258} (2010), 3048--3081.


\bibitem{YAN}
\newblock S. Yan,
\newblock \emph{Concentration of solutions for the scalar curvature equation on $R^N$},
\newblock J. Differ. Equ., \textbf{163} (2000), 239--264.


\bibitem{YZ1}
\newblock M. Yang \& S. Zhao,
\newblock \emph{Blow-up behavior of solutions to critical Hartree equations on bounded domain}. 
\newblock  J. Geom. Anal. \textbf{191} (2023),33.



\end{thebibliography}
\end{document}